\documentclass[a4paper, 12pt]{amsart}
\usepackage{amsfonts, amsthm, amssymb, amsmath, stmaryrd}
\usepackage{mathrsfs,array}
\usepackage{xy}
\usepackage{verbatim}
\usepackage{amscd} 
\usepackage{tikz}
\usepackage{tikz-cd}
\usetikzlibrary{matrix,arrows,decorations.pathmorphing}
\input xy
\xyoption{all}
\usepackage[unicode]{hyperref}

\newtheorem{thm}{Theorem}[subsection]

\newtheorem{cor}[thm]{Corollary}
\newtheorem{ques}[thm]{Question}
\newtheorem{lemma}[thm]{Lemma}
\newtheorem{lem}[thm]{Lemma} 
\newtheorem{prop}[thm]{Proposition}
\newtheorem{conj}[thm]{Conjecture} \theoremstyle{definition}
 \theoremstyle{definition}
\newtheorem{defn}[thm]{Definition} \theoremstyle{remark}
\newtheorem{rem}[thm]{\bf Remark}
\newtheorem{remark}[thm]{\bf Remark}

\newtheorem{assumption}[thm]{\bf Assumption}

\DeclareMathOperator{\Hom}{Hom}
\DeclareMathOperator{\End}{End}
\DeclareMathOperator{\Ext}{Ext}

\DeclareMathOperator{\im}{im}

\DeclareMathOperator{\modd}{mod}

\DeclareMathOperator{\Spec}{Spec}

\DeclareMathOperator{\Gal}{Gal}

\DeclareMathOperator{\Frob}{Frob}

\DeclareMathOperator{\Sym}{Sym}

\DeclareMathOperator{\ad}{ad}
\DeclareMathOperator{\Supp}{Supp}

\def\cont{\mathrm{cont}}

\def\Ind{\mathrm{Ind}}

\def\ps{\mathrm{ps}}
\def\loc{\mathrm{loc}}
\def\ord{\mathrm{ord}}

\def\Irr{\mathrm{Irr}}

\newcommand{\Z}{\mathbb{Z}}
\newcommand{\F}{\mathbb{F}}
\newcommand{\Q}{\mathbb{Q}}

\newcommand{\T}{\mathbb{T}}
\newcommand{\A}{\mathbb{A}}

\newcommand{\tr}{\mathrm{tr}}
\newcommand{\GL}{\mathrm{GL}}

\newcommand{\cO}{\mathcal{O}}

\newcommand{\cOf}{C_\cO^\mathrm{f}}
\newcommand{\kq}{\mathfrak{q}}
\newcommand{\kp}{\mathfrak{p}}

\newcommand{\kP}{\mathfrak{P}}
\newcommand{\kC}{\mathfrak{C}}
\newcommand{\kB}{\mathfrak{B}}

\newcommand{\ka}{\mathfrak{a}}

\newcommand{\km}{\mathfrak{m}}
\newcommand{\mm}{\mathrm{m}}
\newcommand{\DAi}{(D\otimes_F \A_F^\infty)^\times}
\newcommand{\AFi}{(\A_F^\infty)^\times}

\newcommand{\RNum}[1]{\uppercase\expandafter{\romannumeral #1\relax}}
\newcommand{\mainresultpart}[1]{%
	\par\addvspace{1.2\baselineskip}%
	\noindent{\bfseries #1}\par\nobreak
	\addvspace{0.4\baselineskip}%
}
\newcommand{\overbar}[1]{\mkern 1.5mu\overline{\mkern-1.5mu#1\mkern-1.5mu}\mkern 1.5mu}

\begin{document}
	\title{Modular points and dimensions of Eisenstein deformation spaces}
	\author{Xinyao Zhang}
	\address{Graduate School of Mathematical Sciences, University of Tokyo, Komaba, Meguro, Tokyo 153-8914, Japan}
	\email{zhangxy96@g.ecc.u-tokyo.ac.jp}
	\begin{abstract}
		In this article, we study two-dimensional Eisenstein deformation spaces, focusing on the Zariski density of modular points and the dimensions of their irreducible components. Over \(\mathbb Q\), under explicit generic hypotheses, we prove that every irreducible component has characteristic zero and dimension \(4\), and that modular points are Zariski dense. Over certain abelian totally real fields, we prove a big \(R=\T\) theorem for the union of the high-dimensional components. As an arithmetic application, we prove the weight-one Fontaine--Mazur conjecture over \(\mathbb Q\) for \(p\geq 5\), with no hypothesis on the residual representation. In a complementary direction, we prove a local-to-global finiteness theorem for pseudo-deformation rings. Combining this with ordinary finiteness and modularity-theoretic inputs, we prove Mazur's dimension conjecture over \(\mathbb Q\) in both the residually reducible and residually irreducible cases, and deduce Emerton's equidimensionality conjecture for the \(p\)-adic big Hecke algebra for every odd prime.
	\end{abstract}
	
	\maketitle
	
	\tableofcontents
	
	\section{Introduction}
	 
	 Let $E$ be a number field and $p$ be a prime. Let $\Sigma$ be a finite set of finite places of $E$ containing all those lying above $p$. Let $\bar{\rho} : \Gal(\overline{E}/E) \to \GL_2(\overline{\F_p})$ be a continuous Galois representation unramified outside $\Sigma$. By classical Galois deformation theory (see \cite[Proposition 2.3.1, page 42]{Berger_2013}), there exists a universal pseudo-deformation ring $R_E^\ps$ parametrizing all pseudo-representations of $\tr \bar{\rho}$ (for the profinite group $G_{E, \Sigma}$). The geometry of the universal pseudo-deformation space $\operatorname{Spec} R_E^{\mathrm{ps}}$ raises two closely related kinds of questions. The first concerns its intrinsic geometry: the dimensions of its irreducible components and whether they are of
	 characteristic zero. In the residually irreducible case over $\mathbb Q$, the corresponding expected-dimension problem is usually referred to as Mazur's dimension conjecture. The second concerns the automorphic locus: whether the deformation space contains sufficiently many modular or automorphic points, and in particular whether these points are Zariski dense. These two aspects are closely intertwined, since smooth automorphic points can be used to determine the dimensions of irreducible components, while dimension estimates often play an essential role in producing automorphic points.
	 
	 The results of this paper address both aspects in the Eisenstein case. Under
	 explicit generic hypotheses, we determine the geometry of the relevant
	 pseudo-deformation spaces and prove that modular points are Zariski dense.
	 In a complementary direction, without these generic hypotheses, we establish
	 a local-to-global finiteness theorem and use it to determine the dimensions of
	 all irreducible components over \(\mathbb Q\).
	 
	 
	 \subsection{Previous work}
	 We first recall some previous progress in the residually irreducible case. Note that when $\bar{\rho}$ is irreducible, $R_E^\ps$ is just the universal deformation ring of $\bar{\rho}$ (see \cite[Theorem 2.4.1, page 44]{Berger_2013}). In this setting, the dimension problem and the density of modular points have often been studied simultaneously. A common strategy is to produce sufficiently regular modular points on the irreducible components of the deformation space, use the smoothness of these points to determine the dimensions of the components, and then apply an infinite fern argument to obtain Zariski density.
	 
	 In \cite{gouvea}, Gouv\^ea and Mazur first studied this question in the case $E=\Q$, $\Sigma = \{p\}$, and \(\bar\rho\) is irreducible and unobstructed; the modularity hypothesis in the original argument is now automatic by Serre's modularity theorem
	 \cite{Khare_2009a, Khare_2009}. In this setting, classical deformation theory already shows that  \(R_{\mathbb Q}^{\mathrm{ps}}\) is a power series ring of the expected dimension. Using the theory of eigencurves, or the so-called infinite fern, they further proved that modular points are Zariski dense and that \(R_{\mathbb Q}^{\mathrm{ps}}\) is isomorphic to a suitable big Hecke algebra.
	 
	 B\"ockle subsequently refined this strategy in \cite{B_ckle_2001}.  Under certain local hypotheses, he proved that every irreducible component of the universal deformation space contains a smooth modular point.  The existence of such a point determines the dimension of the corresponding component. In this way, the expected-dimension statement and the Zariski density
	 of modular points are obtained as parts of the same argument.  Emerton later developed related ideas in \cite{emerton11} concerning big $R=\mathbb T$ theorems and the two-dimensional Fontaine-Mazur conjecture in the non-ordinary case.
	
	
	 	There are also higher-dimensional and more general versions of this strategy.
	 For example, Allen proved analogous results for polarized Galois
	 representations over CM fields and in certain Hilbert modular settings in \cite{Allen_2019}.  A key input in his work is the finiteness of global deformation rings over products of local deformation rings proved in \cite{Allen_2014}, combined with vanishing results for the geometric Bloch-Kato Selmer group (see the main result in \cite{Allen_2016}) and Chenevier's infinite fern \cite{chenevier2011}.
	
	The residually reducible case is more subtle from both points of view, since the natural objects are pseudo-deformation rings, whose geometry can be quite complicated. More precisely, unlike classical deformation rings, their tangent spaces cannot in general be studied directly by the Galois cohomology of a single residual representation. In this direction, Deo proved in \cite{Deo2023} big $R=\mathbb T$ type results over $\mathbb Q$ under additional hypotheses, including a cyclicity assumption on a certain Galois cohomology group. Under these hypotheses, Deo can work with classical deformation rings using the theory of generalized matrix algebras (GMAs). His proof follows B\"ockle's strategy, replacing the small $R=\mathbb T$ input by Pan's ordinary finiteness theorem \cite[Theorem 5.1.1]{pan2022fontaine}, and then applies Kisin's vanishing result for the geometric Bloch-Kato Selmer group \cite{Kisin_2004}.
	
	As a summary, the known expected-dimension results have
	generally depended on substantial automorphy-theoretic input. A typical argument
	first places a sufficiently regular automorphic point on every irreducible
	component and then uses the vanishing of an adjoint Selmer group to determine
	the local dimension at that point. Consequently, Mazur's dimension conjecture
	has frequently appeared together with \(R=\mathbb T\) theorems and density
	results, although it is not merely a formal consequence of a previously known
	big \(R=\mathbb T\) theorem.
	
	The present paper studies both aspects of the Eisenstein pseudo-deformation
	space directly. For the Zariski density results, we follow the general
	philosophy of the classical strategy, while developing new arguments to control
	pseudo-deformation rings without reducing to classical deformation rings. For
	the expected-dimension problem, we introduce a different local-to-global
	finiteness strategy, which does not take a previously established big
	\(R=\mathbb T\) theorem for the same global pseudo-deformation ring as an input.
		
	\subsection{Main results}
	The main results of this paper fall into two complementary parts, corresponding
	to the two directions of the interaction described above. The first concerns the
	automorphic locus in Eisenstein pseudo-deformation spaces. We use dimension
	estimates and connectedness arguments to place modular points on irreducible
	components, and then combine ordinary finiteness and Selmer-group vanishing
	to prove Zariski-density and big \(R=\mathbb T\) results, with an application to
	the irregular Fontaine--Mazur conjecture. We also treat the residually
	irreducible weight-one case over $\mathbb Q$ for $p\geq5$, using the
	finiteness results of Appendix~\ref{A}.
	
	The second part concerns the intrinsic geometry of global deformation spaces.
	Here modularity-theoretic input is used in the opposite direction: it yields a
	local-to-global finiteness theorem for pseudo-deformation rings, allowing the
	geometry of the local ring at \(p\) to impose sharp upper bounds on the global
	deformation space. We use this to prove the expected-dimension statements over
	\(\mathbb Q\) and to deduce applications to big Hecke algebras. Thus, in the
	first part dimension estimates help produce automorphic points, whereas in the
	second part automorphy controls dimensions through local-to-global finiteness.
	
	
	To state our result more precisely, we first fix some notations. Let $p$ be an odd prime. Let $K$ be a $p$-adic local field with ring of integers $\cO$, a uniformizer $\varpi$ and the residue field $\F$. We use $\varepsilon$ (resp. $\omega$) to denote the $p$-adic (resp. mod $p$) cyclotomic character.
	
	\mainresultpart{Part I: Modular points and big \(R=\mathbb T\) theorems}
	
  Our first result concerns the case $F=\mathbb Q$.  Let $\Sigma$ be a finite set of rational primes containing $p$, and let
	\(	\bar\chi:G_{\mathbb Q,\Sigma}\longrightarrow \mathbb F^\times
	\) be a continuous odd character.  Let $R^{\mathrm{ps}}_{\mathbb Q}$ be the
	universal pseudo-deformation ring of the residual pseudo-representation
	$1+\bar\chi$. The following result is an analogue of \cite[Theorem 1.2.3]{emerton11}.
	
	\begin{thm}[Theorem \ref{big R=T Q}]\label{1.1}Keep notations as above. Assume that $\bar{\chi}|_{G_{\mathbb{Q}_p}} \ne \omega^{\pm 1}$.
		Then every irreducible component of
		\(
		\operatorname{Spec}R^{\mathrm{ps}}_{\mathbb Q}
		\)
		is of characteristic zero and has Krull dimension $4$.  Moreover, modular
		points are Zariski dense in
		\(
		\operatorname{Spec}R^{\mathrm{ps}}_{\mathbb Q}.
		\)
	\end{thm}
	
	The proof is inspired by the strategies of \cite{pan2022fontaine} and \cite{Deo2023}, but an essential new difficulty is that the pseudo-deformation ring cannot be identified with the deformation ring of a chosen non-split residual representation. Such an identification is available in Deo's setting because of the cyclicity hypothesis
	on the relevant Galois cohomology group. Without it, neither the tangent space nor the
	dimensions of the pseudo-deformation ring are controlled by a single Galois
	cohomology group. The main technical problem is therefore to obtain a sharp lower bound for every irreducible component of \(\operatorname{Spec}R_{\mathbb Q}^{\mathrm{ps}}\). In Section~\ref{sec:lower}, we prove lower bounds for the dimensions of these components by combining a connectedness-dimension argument with a detailed analysis of the reducible locus. The condition $\bar{\chi}|_{G_{\mathbb{Q}_p}} \ne \omega^{\pm 1}$ enters in our analysis of the ordinary locus: local pseudo-deformation theory (see \cite[Appendix B]{Pa_k_nas_2013}), together with Pan's ordinary finiteness results \cite[Theorem 5.1.1 \& Theorem 6.1.1]{pan2022fontaine}, gives the required dimension control and produces suitable ordinary modular points. At the resulting regular algebraic cuspidal points, Newton--Thorne's vanishing theorem for adjoint Bloch--Kato Selmer groups \cite{NewtonThorne}, combined with local--global compatibility and the genericity of the local automorphic components, gives the expected smoothness. The density statement then follows from the usual infinite fern argument.
	
	Combining the Eisenstein argument with the results of Appendix~\ref{A}, we obtain the weight-one Fontaine--Mazur conjecture over $\mathbb Q$ for $p\geq 5$ without any residual hypothesis.

	\begin{thm}[Theorem~\ref{thm:irregular-FM-Eisenstein} \& Corollary~\ref{cor:irregular-FM-nonEis}]\label{1.2}
		Let $p\geq 5$ and let $\rho:G_{\mathbb Q}\to\mathrm{GL}_2(\overline{\mathbb Q}_p)$ be a continuous, odd, irreducible representation unramified outside finitely many primes. If $\rho|_{G_{\mathbb Q_p}}$ is Hodge--Tate of weights $\{0,0\}$, then $\rho$ arises from a classical cuspidal eigenform of weight $1$. In particular, $\rho$ has finite image.
	\end{thm}

	\begin{remark}\label{rem:intro-weight-one-FM}
		1) We explain some previous results. Pilloni--Stroh \cite[Theorem 0.2]{PS16} proved weight-one Fontaine--Mazur over totally real fields under a Taylor--Wiles hypothesis, with an additional condition in the exceptional case at $p=5$. Over $\mathbb Q$, Pan \cite[Corollary 6.4.9]{Pan_2022} combined his classicality theorem with Emerton's pro-modularity theorem \cite[Theorem 1.2.3]{emerton11}, which requires Taylor--Wiles and certain local hypotheses. Berger--Klosin \cite{BK23} obtained $R=\mathbb T$ results for certain ordinary Eisenstein weight-one deformation problems, yielding classical CM forms in the quadratic, locally split case under further hypotheses. Theorem~\ref{1.2} also covers the cases in which the Taylor--Wiles hypothesis fails, without imposing a local genericity condition.

		2) The main new input is pro-modularity, and the final step is Pan's classicality theorem in \cite{Pan_2022}. In the Eisenstein case, we use the componentwise density of Theorem~\ref{1.1} where it applies and treat the exceptional residual ratios by a separate ordinary argument. More generally, the pro-modularity arguments combine finiteness results with regular Fontaine--Mazur to find regular modular points, and use Newton--Thorne's adjoint Selmer vanishing and the infinite fern to identify the components containing them. See Sections~\ref{irr fm} and~\ref{app:irregular-FM-nonEis}.

		3) At $p=3$, Theorems~\ref{thm:irregular-FM-Eisenstein} and~\ref{thm:irregular-FM-nonEis-locred} also treat the exceptional residual ratio when the characteristic-zero local representation is reducible, using a characteristic-zero ordinary dimension estimate. The locally irreducible exceptional case remains untreated; for locally absolutely irreducible residual representations, cyclotomic self-twists can also obstruct the local regularity argument. See Remarks~\ref{rem:irregular-FM-Eis-p3} and~\ref{rem:irregular-FM-nonEis-p3}.
	\end{remark}
	
	We also obtain a version of the density result of Theorem~\ref{1.1} over certain totally
	real fields.  Let $F$ be an abelian totally real field of even degree in which
	$p$ splits completely, and let $\Sigma$ be a finite set of finite places of $F$
	containing all places above $p$.  Let
	\(
	\bar\chi:G_{F,\Sigma}\longrightarrow \mathbb F^\times
	\)
	be a continuous totally odd character which extends to $G_{\mathbb Q}$, and let
	\(
	\chi:G_{F,\Sigma}\longrightarrow \mathcal O^\times
	\)
	be a de Rham lift of $\bar\chi$.  Put $\bar T=1+\bar\chi$, and let
	$R^{\mathrm{ps},\chi}_F$ be the universal pseudo-deformation ring of $\bar T$
	with determinant $\chi$. Let $R^{\mathrm{ps},\chi}_{F,0}$ denote the reduced closed subscheme given by the union of all irreducible components of $R^{\mathrm{ps},\chi}_F$ of dimension at least $1+2[F:\mathbb Q]$.  For a suitable tame level $U^p$, let $\mathbb T_\psi(U^p)_\mathfrak m$ be the corresponding big Hecke algebra of the completed cohomology of quaternionic forms, localized at the Eisenstein maximal ideal associated with $\bar T$.
	
	\begin{thm}[Theorem \ref{big R=T quaternion}]
		Assume that
		\(
		\bar\chi|_{G_{F_v}}\neq \omega^{\pm 1}\)
		for all \( v\mid p
		\).
		Then there exists a tame level $U^p$ such that
		\(
		\operatorname{Spec} R^{\mathrm{ps},\chi}_{F,0}
		=
		\operatorname{Spec} \mathbb T_\psi(U^p)_{\mathfrak m}^{\operatorname{red}} .
		\)
		Moreover, $R^{\mathrm{ps},\chi}_{F,0}$ is equidimensional of Krull dimension
		\(
		1+2[F:\mathbb Q].
		\)
	\end{thm}
	
	\begin{remark}
		The proof follows the same general strategy as in the case $F=\mathbb Q$. The assumptions that $F$ is abelian and $\bar{\chi}$ can be extended to a character of  $G_{\mathbb Q}$ are for the ordinary Fontaine-Mazur conjecture proved in \cite{Skinner_1999} and \cite{pan2022fontaine}. The additional ingredient over $F$ is a lower bound for the dimension of the big Hecke algebra, obtained from local-global compatibility for completed cohomology and Pa\v{s}k\=unas' theory. See \cite[Theorem 3.6.1]{pan2022fontaine} and \cite[Corollary 3.3.10]{Zhang:2024aa}.
	\end{remark}

	\mainresultpart{Part II: Local-to-global finiteness and dimensions}
	
    We now turn from the automorphic locus to the second main theme of the paper,
    namely the intrinsic dimensions of global deformation spaces. In the range
    where the preceding big \(R=\mathbb T\) theorem is available, the
    equidimensionality of the relevant pseudo-deformation ring is already obtained
    as part of the argument. The dimension-theoretic method developed below,
    however, applies in greater generality and does not take the preceding big \(R=\mathbb T\) theorem as an input.
    
    Let $\Sigma$ be a finite set of rational primes containing $p$, and let
    \(
    \bar T:G_{\mathbb Q,\Sigma}\longrightarrow \mathbb F
    \)
    be a two-dimensional pseudo-representation satisfying the oddness condition
    (see Assumption \ref{odd ass}).  Let
    \(
    R^{\mathrm{ps}}_{\mathbb Q}
    \)
    be the universal pseudo-deformation ring of $\bar T$ without fixed determinant.
	
    \begin{thm}[Theorem \ref{mazur thm eisenstein}]
    	\label{thm:intro-Mazur-Eis}
    	Assume that $\bar T$ is reducible.  Then every irreducible component of
    	\(
    	\operatorname{Spec} R^{\mathrm{ps}}_{\mathbb Q}
    	\)
    	has Krull dimension $4$ and is of characteristic zero.
    \end{thm}
	
     In the residually irreducible case, the universal pseudo-deformation ring
     coincides with the usual universal deformation ring.  We prove the following
     analogue in Appendix~\ref{A}, which is often referred to as Mazur's dimension conjecture in the literature.
     
     \begin{thm}[Theorem \ref{mazur non-eisenstein} and Corollaries \ref{cor:nonEis-flat}, \ref{cor:nonordinary-flat}]
     	\label{thm:intro-Mazur-nonEis}
     	Assume that $\bar T=\operatorname{tr}\overline\rho$, where
     	\(
     	\overline\rho:G_{\mathbb Q,\Sigma}\longrightarrow \mathrm{GL}_2(\mathbb F)
     	\)
     	is a continuous, odd and absolutely irreducible Galois representation.  Let $R_{\mathbb Q}$ be the universal deformation ring of $\overline\rho$ without fixed determinant. Then $R_{\mathbb Q}$ is an $\cO$-flat local complete intersection of Krull dimension $4$.
     \end{thm}
     
    	\begin{rem}
	
		1) In the residually reducible case, one should not expect the universal
		pseudo-deformation ring to be a local complete intersection ring in general.  We do not know whether the ring 	\(
		 R^{\mathrm{ps}}_{\mathbb Q}
		\) is reduced either. See \cite{Deo2022} for more examples and discussions.

	  2) Compared with Theorem~\ref{1.1}, Theorem~\ref{thm:intro-Mazur-Eis} removes the local genericity hypothesis, but it does not assert the Zariski density of modular points. Likewise, the dimension assertion in Theorem~\ref{thm:intro-Mazur-nonEis} requires neither a Taylor--Wiles hypothesis nor a local genericity condition and is proved without assuming a big $R=\mathbb T$ theorem over $\mathbb Q$.

	  3) The flatness assertion in Theorem~\ref{thm:intro-Mazur-nonEis} uses a separate argument. Ordinary finiteness (Theorem~\ref{thm:ordinary-finiteness-nonEis}), analogous to \cite[Theorem 5.1.1 \& Theorem 6.1.1]{pan2022fontaine}, excludes characteristic-$p$ components when the local residual representation is reducible, and fixed-type finiteness (Theorem~\ref{thm:nonordinary-type-finiteness}), analogous to \cite[Theorem 1.1.1]{EGPZ}, does so when it is absolutely irreducible. 
		\end{rem}
	
	The common input for the two dimension assertions is a local-to-global
	finiteness theorem for pseudo-deformation rings. It gives the upper bound
	without first having to place a smooth modular point on every irreducible
	component.
	
	Let \(\chi\) be a de Rham lift of \(\det\bar T\). Let
	\(R^{\mathrm{ps},\chi}_{\mathbb Q}\) be the global pseudo-deformation ring of
	\(\bar T\) with fixed determinant \(\chi\), and let
	\(R^{\mathrm{ps},\chi}_p\) be the corresponding local pseudo-deformation ring
	at \(p\). There is a natural morphism
	\(
	R^{\mathrm{ps},\chi}_p\longrightarrow
	R^{\mathrm{ps},\chi}_{\mathbb Q}
	\).
	Our finiteness theorem is the following.
	
	
	
	
	\begin{thm}[Proposition \ref{finiteness 2} \& Proposition \ref{finiteness irr}]
		\label{thm:intro-finiteness}
	The natural map
		$R^{\mathrm{ps},\chi}_p \to R^{\mathrm{ps},\chi}_{\mathbb Q}$
		is finite.
	\end{thm}
	
    \begin{remark}
    	1) Theorem \ref{thm:intro-finiteness} may be regarded as a pseudo-deformation
    	analogue of the finiteness theorem of \cite[Theorem 1]{Allen_2014}.  In the corresponding situations, our finiteness results can also be used to study the  unramified Fontaine-Mazur conjecture.
    	
    	2) The use of the local pseudo-deformation ring, rather than a framed local
    	deformation ring (as in \cite{Allen_2014}), is essential for obtaining a sharp bound. See Remark \ref{6.3.2}.
   
    	3) The proof using potential pro-modularity and the \(p\)-adic Langlands
    	correspondence for \(\mathrm{GL}_2(\mathbb Q_p)\) applies in both residual
    	settings. In the residually reducible case, Section \ref{sec finite 1}
    	gives another proof using ordinary finiteness.
    \end{remark} 
	
	We now indicate how the finiteness theorem gives the upper bound in Mazur's
	dimension conjecture.  By the recent progress in local deformation theory in \cite{B_ckle_2023}, \(R^{\mathrm{ps},\chi}_p\) is an \(\mathcal O\)-flat
	integral domain of Krull dimension \(4\).  Hence the finiteness theorem gives
	an upper bound
	\(
	\dim R^{\mathrm{ps},\chi}_{\mathbb Q} \leq 4.
	\)
	It remains to rule out the possibility that an irreducible component of	\(\Spec R^{\mathrm{ps},\chi}_{\mathbb Q}\)  has dimension \(4\).  We briefly
	explain the idea to exclude the case $\dim R^{\mathrm{ps},\chi}_p = \dim R^{\mathrm{ps},\chi}_{\mathbb Q}$, assuming for simplicity that the fixed determinant \(\chi\) is crystalline. 
	
	Suppose first that we are in the residually reducible case. By Theorem \ref{thm:intro-finiteness}, the
	ordinary locus of $\Spec R^{\mathrm{ps},\chi}_p $, which has Krull
	dimension \(3\), would give rise to a global ordinary pseudo-deformation of dimension $3$.  This contradicts Pan's ordinary finiteness theorem,
	which bounds the corresponding ordinary global pseudo-deformation space by dimension \(2\).  
	
	In the residually irreducible case, local-global compatibility and
	Proposition \ref{paskunas} show that the natural map from
	\(R_p^{\mathrm{ps},\chi}\) to the big Hecke algebra of a fixed, sufficiently
	small tame level has non-zero kernel. If a global component had dimension
	\(4\), finiteness would make the map from \(R_p^{\mathrm{ps},\chi}\) to that
	component injective. Every irreducible regular crystalline local point
	would then lift to a global point, which is automorphic by
	\cite[Theorem 1.0.2]{Zhang:2024aa}. A uniform bound on the tame conductors
	places these eigensystems in one big Hecke algebra. By the local density
	result of Proposition \ref{density of crys}, its kernel would be zero,
	giving the desired contradiction.
	
	As a direct consequence, we get the equidimensionality of the big Hecke algebra of the completed cohomology of quaternionic forms for $\Q$.
	
	\begin{thm}[Theorem \ref{dim of T}]
		For a tame level $U^p$, each irreducible component of the big Hecke algebra $\T_\psi(U^p)$ has Krull dimension $3$ and is of characteristic $0$.
	\end{thm}
	
	For Emerton's big Hecke algebra, the universal pseudo-deformation rings give the upper bound in both residual settings. Together with the componentwise lower bound of Theorem~\ref{lower bound}, this proves Emerton's conjecture \cite[Conjecture 2.9]{Emerton2011} for all odd primes. See Section~\ref{emerton's conj} for the precise definition of $\T$.
	
	\begin{thm}[Theorem \ref{main thm for emerton}]
		The big Hecke algebra $\T$ (in the sense of Emerton) is equidimensional of Krull dimension $4$ when $p \ge 3$.
	\end{thm}

   	\begin{remark}
   		The dimension formula above is also the one entering the conjectural
   		description of completed cohomology proposed by Calegari-Emerton in \cite{CE12}. Their Iwasawa-codimension formulae involve the dimensions of the corresponding big Hecke algebras, and hence the expected dimensions of global deformation rings; see \cite[page 256, (6)]{CE12}. Thus our results verify the deformation-theoretic dimension predicted
   		by this picture in the present two-dimensional setting.
   	\end{remark}
   
   Taken together, the results above provide both automorphic and
   dimension-theoretic control of Eisenstein pseudo-deformation spaces. This
   framework suggests further arithmetic applications. One natural direction is
   to develop residually reducible analogues of the universal zeta morphisms and
   Iwasawa main conjectures studied in \cite{Nakamura2023} and
   \cite{Fouquet:2021aa}. It would also be interesting to investigate whether the
   methods developed here admit higher-dimensional analogues, especially in light
   of \cite{Thorne_2014} and \cite{Allen_2020}.

    This paper is organized as follows. In Section \ref{sec 2}, we prepare some results for the universal pseudo-deformation rings. In Section \ref{sec 3}, we recall Pašk\=unas theory in \cite{Pa_k_nas_2013}, \cite{Pa_k_nas_2021} and \cite{Paskunas_2021}, and give some corollaries based on the local-global compatibility results at $p$. In Section \ref{sec 4}, we recall the known cases of Fontaine-Mazur conjecture for $\GL_2$ and potential pro-modularity arguments in \cite{Zhang:2024ab} and \cite{Zhang:2024aa}. In Section \ref{sec 5}, we prove the main Zariski density results for modular points in the pseudo-deformation spaces in the residually reducible case.  We first treat the case $F=\mathbb Q$, derive an application to the irregular Fontaine-Mazur
    conjecture, and then discuss the corresponding results over certain totally real fields. In Section \ref{sec 6}, we first prove some finiteness results relating local and global pseudo-deformation rings, use them to establish Mazur's dimension conjecture in the Eisenstein case over $\mathbb Q$, and deduce Emerton's equidimensionality conjecture for the big Hecke algebra. Appendix \ref{A} establishes Mazur's dimension conjecture in the residually irreducible case, supplying the non-Eisenstein input for Emerton's conjecture. It also proves ordinary and fixed-type finiteness, analogous to \cite[Theorems 5.1.1 \& 6.1.1]{pan2022fontaine} and \cite[Theorem 1.1.1]{EGPZ}, and uses these results to establish $\cO$-flatness and, for $p\geq 5$, the weight-one Fontaine--Mazur conjecture.
    
	\subsection{Notations}
	
		Let $p$ be a prime. We always denote a $p$-adic local field by $K$. We also denote its uniformizer by $\varpi$, its ring of integers by $\cO$ and the residue field by $\F$. We fix an embedding \(K\hookrightarrow \overline{\mathbb Q}_p\), where \(\overline{\mathbb Q}_p\) is a fixed algebraic closure of \(\mathbb Q_p\), and
		an isomorphism \(\iota_p:\overline{\mathbb Q}_p\simeq \mathbb C\).
	
	We use $\mathcal{C}_{\mathcal{O}}$ to denote the category of complete noetherian local $\cO$-algebras with residue field $\F$. For an $\cO$-module $M$, we use $M^\vee$ to denote its Pontryagin dual $\Hom_{\cO}^{\cont}(M, K/\cO)$, and we use $M^d$ to denote its Schikhof dual $\Hom_{\cO}^{\cont}(M, \cO)$. 
	
	For a prime ideal $\kp$ of a commutative ring $R$, we denote its residue field by $k(\kp)$. For the dimension of the prime $\kp$, we mean the Krull dimension of $R/\kp$. Let $R_\kp$ be the localization at $\kp$. We write $\widehat{R_\kp}$ for its $\kp$-adic completion. We say a commutative ring $R$ is a CNL ring if it is a complete noetherian local ring, and we usually denote its maximal ideal by $\mathfrak{m}_R$ unless otherwise stated. We use $R^{\operatorname{red}}$ to denote the maximal reduced quotient of the ring $R$.
	
	Suppose $F$ is a number field with maximal order $O_F$. We write $F_+$ for the set of totally positive elements in $F$. For any finite place $v$, we write $F_v$ (resp. $\cO_{F_v}$) for the completion of $F$ (resp. $\cO_F$) at $v$, $\varpi_v$ for a uniformizer of $F_v$, $k(v)$ for the residue field, $\operatorname{Nm}(v)$ for the norm of $v$ (in $\Q$),  $G_{F_v}$ for the decomposition group above $v$, $I_{F_v}$ (or just $I_v$ if there is no confusion for the number field $F$) for its inertia group and $\Frob_v$ for a geometric Frobenius element in $G_{F_v}/I_{F_v}$. If $l$ is a rational prime, then we denote $O_F\otimes_\Z \Z_l$ by $\cO_{F,l}$. The adele ring of $F$ will be denoted by $\A_F$ and $\A_F^\infty$ will be the ring of finite adeles. Suppose $\Sigma$ is a finite set of places of $F$. We use $G_{F, \Sigma}$ for the Galois group of the maximal extension of $F$ unramified outside $\Sigma$ and all infinite places. The absolute Galois group of $F$ is denoted by $G_F=\Gal(\overbar{F}/F)$. 
	
	We use $\varepsilon$ to denote the $p$-adic cyclotomic character and $\omega$ to denote the mod $p$ cyclotomic character. Our convention for the Hodge-Tate weight of $\varepsilon$ is $-1$.

	\subsection{Acknowledgements}
	The author would like to thank his supervisor Takeshi Saito for his constant help and encouragement. The author would also like to thank Vytautas Pašk\=unas for sharing ideas concerning Section \ref{sec pas} and for many helpful discussions and comments. The author also wants to thank Yufan Luo for his interest in this work and for his help with Proposition \ref{prop:nonordinary-Hecke-finiteness}.

	\subsection{Declaration of AI use}
	OpenAI's GPT-6 suggested the initial proof strategy for the fixed-type finiteness argument in Section~\ref{app:nonordinary-finiteness} (Proposition~\ref{prop:nonordinary-Hecke-finiteness} and Theorem~\ref{thm:nonordinary-type-finiteness}). The context provided to the model included the author's earlier work \cite{Zhang:2024aa}, the first two subsections of Appendix~\ref{A}, and the author's discussions with Yufan Luo. The author developed the strategy into the proof presented here and independently verified the argument. The other proof strategies were developed by the author. GPT was also used to edit the exposition, check mathematical arguments, and suggest refinements to proofs. 
	
	\section{Galois deformation theory}\label{sec 2}
	
	In this section, we collect some results about Galois deformation theory.
	
	\subsection{Preliminaries in commutative algebra}
	
	In this subsection, let $p$ be an odd prime and $K$ be a $p$-adic local field with its ring of integers $\mathcal{O}$, a uniformizer $\varpi$ and its residue field $\mathbb{F}$.
	
	\begin{defn} \cite[Remark 19.2.5]{localcoh}
		For a positive integer $r$, denote by $S(r)$ the set of all ordered pairs $(A,B)$ of non-empty subsets of $\{1, ..., r\}$ for which $A \cup B=\{1, ..., r\}$. Let $R$ be a Noetherian ring with minimal primes $\mathfrak{p}_1, ..., \mathfrak{p}_r$. The \textit{connectedness dimension} $c(R)$ of $R$ is defined to be $$
		c(R) = \min_{(A,B) \in S(r)} \{\dim R/((\cap_{i \in A} \mathfrak{p}_i)+(\cap_{j \in B} \mathfrak{p}_j))\}.$$ 
	\end{defn}
	
	It is easy to see that the connectedness dimension can also be defined as $$c(R)=\min_{Z_1,Z_2}\{\dim (Z_1\cap Z_2)\}, $$
	where $Z_1,Z_2$ are non-empty unions of irreducible components of $\Spec R$ such that $Z_1\cup Z_2=\Spec R$.
	
	\begin{prop}\label{connected dim}
		 Let $(R, \mathfrak{m}_R) $ be a Noetherian local ring, and let $\widehat{R}$ be its $ \mathfrak{m}_R$-adic completion. Then we have $c(R) \ge c(\widehat{R})$.
	\end{prop}
	
	\begin{proof}
		See \cite[Lemma 19.3.1]{localcoh}.
	\end{proof}
	
	Write $\mathcal{C}_{\mathcal{O}}$ for the category of CNL $\mathcal{O}$-algebras with residue field $\mathbb{F}$. The following result collects some useful facts about the completed tensor product of finitely many objects.
	
	\begin{prop}\label{com alg}
		Let $n $ be a positive integer. For any $1 \le i \le n $, assume that $R_i \in \mathcal{C}_{\mathcal{O}}$ is $\cO$-flat. (Here we say that an ideal $I$ of $R_i$ is $\cO$-flat if the ring $R_i/I$ is $\cO$-flat.)
		
		1) The completed tensor product ${\widehat{\otimes}_\cO} R_i$ is $\cO$-flat.
		
		2) If all $R_i $ are $\cO$-flat domains, $1 \le i \le n$, then $\dim {\widehat{\otimes}_\cO} R_i = 1+ \sum_{i=1}^{n} (\dim R_i -1)$.
		
		3) For any ideal $\kp_i$ of $R_i $ such that $R_i/\kp_i$ is $\cO$-flat, we have $$ 
		\widehat{\otimes}_\cO (R_i/\kp_i) =(\widehat{\otimes}_\cO R_i)/(\kp_1, ..., \kp_n). $$
		
		
	\end{prop}
	
	\begin{proof}
		The first assertion follows from \cite[Lemma 1.3]{Thorne_2014}. The second one follows from 1) and \cite[Lemma 1.4]{Thorne_2014}. The third one follows from 1) and \cite[Lemma 4.2.4]{pan2022fontaine}.

	\end{proof}
	
   \begin{prop}\label{lci}
		 Assume that \(R, S \in \mathcal{C}_{\mathcal{O}}\) are local complete intersection rings over \(\mathcal O\). Assume
		moreover that \(S\) is \(\mathcal O\)-flat. Then
		\(
		R\widehat{\otimes}_{\mathcal O} S
		\)
		is a local complete intersection ring over $\cO$. 
	\end{prop}
	
	\begin{proof}
		Choose presentations
		\[
		R\simeq \mathcal O[[x_1,\ldots,x_m]]/(f_1,\ldots,f_r),
		\qquad
		S\simeq \mathcal O[[y_1,\ldots,y_n]]/(g_1,\ldots,g_s),
		\]
		where \(f_1,\ldots,f_r\) and \(g_1,\ldots,g_s\) are regular sequences in
		\(\mathcal O[[x_1,\ldots,x_m]]\) and \(\mathcal O[[y_1,\ldots,y_n]]\), respectively.
		
		Put
		\(
		A:=\mathcal O[[x_1,\ldots,x_m]]\) and \(
		B:=\mathcal O[[y_1,\ldots,y_n]].
		\)
		Write
		\(
		T:=A\widehat{\otimes}_{\mathcal O}B
		\).
		Then we have
		\(
		R\widehat{\otimes}_{\mathcal O}S
		\simeq
		T/(f_1,\ldots,f_r,g_1,\ldots,g_s).
		\)
		It therefore suffices to show that
		\(
		g_1,\ldots,g_s,f_1,\ldots,f_r
		\)
		is a \(T\)-regular sequence.
		
		Since \(T\) is flat over \(B\) (by \cite[Lemma 1.3]{Thorne_2014}), and
		\(g_1,\ldots,g_s\) is a \(B\)-regular sequence, by \cite[Lemma 10.68.5]{stacks-project}, the sequence
		\(g_1,\ldots,g_s\) remains regular in \(T\). 
		Since \(S\) is \(\mathcal O\)-flat, the natural map
		\(
		A=\mathcal O[[x_1,\ldots,x_m]]
		\longrightarrow
		S[[x_1,\ldots,x_m]] \cong A\widehat{\otimes}_{\mathcal O}S
		\)
		is flat. Therefore, by \cite[Lemma 10.68.5]{stacks-project}, \(f_1,\ldots,f_r\) remains a regular sequence in
		\(S[[x_1,\ldots,x_m]]\). Thus
		\(
		g_1,\ldots,g_s,f_1,\ldots,f_r
		\)
		is a regular sequence in the regular local ring \(T\). Hence
		\(R\widehat{\otimes}_{\mathcal O}S\) is a local complete intersection ring over $\cO$.
	\end{proof}
	
	\begin{prop}\label{finite-after-reduction}
		Let \(A\to B\) be a homomorphism of commutative rings, and assume that
		\(B\) is Noetherian. Then \(A\to B\) is finite if and only if the induced
		map
		\(
		A\longrightarrow B^{\operatorname{red}}
		\)
		is finite.
	\end{prop}
	
	\begin{proof}
		One implication is immediate. Conversely, let
		\(
		\mathfrak n:=\sqrt{(0)}
		\)
		be the nilradical of \(B\). Since \(B\) is Noetherian, the ideal
		\(\mathfrak n\) is nilpotent; say \(\mathfrak n^r=0\). Moreover, for every
		\(i\geq 1\), the \(B/\mathfrak n\)-module
		\(
		\mathfrak n^i/\mathfrak n^{i+1}
		\)
		is finitely generated. Since \(B/\mathfrak n=B^{\operatorname{red}}\) is
		finite over \(A\), it follows that every
		\(\mathfrak n^i/\mathfrak n^{i+1}\) is a finite \(A\)-module.
		
		Starting with \(B/\mathfrak n\), and using successively the exact sequences
		\[
		0\longrightarrow
		\mathfrak n^i/\mathfrak n^{i+1}
		\longrightarrow
		B/\mathfrak n^{i+1}
		\longrightarrow
		B/\mathfrak n^i
		\longrightarrow 0,
		\]
		we conclude by induction that \(B/\mathfrak n^r=B\) is a finite
		\(A\)-module.
	\end{proof}
	
	\subsection{Pseudo-representations}\label{sec for pseudo def}
	
	We follow the conventions of \cite{pan2022fontaine}.
	
	\begin{defn} \label{pseudo}
		For a profinite group $G$ and a topological commutative ring $R$ in which $2$ is invertible, a two-dimensional \textit{pseudo-representation} is a continuous function $T: G \to R$ such that 
		
		1) $T(1)=2$,
		
		2) $T(\sigma \tau)=T(\tau \sigma)$ for all $ \sigma, \tau \in G$,
		
		3) $T(\gamma \delta \eta)+T(\gamma \eta \delta)-T(\gamma \eta)T(\delta)-T(\eta \delta)T(\gamma)-T(\delta \gamma)T(\eta)+T(\gamma)T(\delta)T(\eta)=0,$
		for any $\delta, \gamma,\eta\in G$.
		
		The \textit{determinant} $\operatorname{det}(T)$ of a pseudo-representation $T$ is a character of $G$ defined by $$ \operatorname{det}(T): G \to R^\times, ~
		\operatorname{det}(T)(\sigma)=\frac{1}{2}(T(\sigma)^2-T(\sigma^2)), \quad \sigma \in G.$$
		
	\end{defn}
	
	Sometimes, we may take $G$ as some Galois group of a totally real number field. In this case, pseudo-representations will have more properties as a complex conjugation exists in $G$. We make the following assumption:
	
	\begin{assumption}\label{odd ass}
		There exists an order $2$ element $c \in G$ such that $T(c)=0$.
	\end{assumption}
		
	\begin{defn}\label{a,d}
		Suppose Assumption \ref{odd ass} holds for $G$. For a pseudo-representation $T$ and $\sigma, \tau \in G$,  we define:
		
		1) $a(\sigma)=\frac{1}{2}(T(c\sigma)+T(\sigma))$,
		
		2) $d(\sigma)=\frac{1}{2}(-T(c\sigma)+T(\sigma))$,
		
		3) $y(\sigma, \tau)=a(\sigma \tau)-a(\sigma)a(\tau).$
	\end{defn}
	
	We can check that this definition is equivalent to the one in \cite[Section 2.4]{Skinner_1999}.
	
	\begin{defn}\label{construction}\cite[2.1.5]{pan2022fontaine}
		Suppose Assumption \ref{odd ass} holds for $G$. Assume that $R$ is either a field or a DVR. For a pseudo-representation $T:G\to R$, an associated representation $\rho$ may be constructed as follows.
		
		1) $ \rho(\sigma)=
		\begin{pmatrix}
			a(\sigma) & ~\\
			~& d(\sigma)
		\end{pmatrix}$, if all $y(\sigma, \tau)=0$. We call this case \textit{reducible}.
		
		2) If $y(\sigma, \tau) \ne 0 $ for some $ \sigma, \tau \in G$, choose $\sigma_0, \tau_0$ such that $\frac{y(\sigma, \tau)}{y(\sigma_0, \tau_0)} \in R$ for any $ \sigma, \tau \in G$. Define $ \rho(\sigma)=
		\begin{pmatrix}
			a(\sigma) & \frac{y(\sigma, \tau_0)}{y(\sigma_0, \tau_0)}\\
			y(\sigma_0, \sigma)& d(\sigma)
		\end{pmatrix}$.  We call this case \textit{irreducible}.
	\end{defn}
	
	Assume the profinite group $G$ satisfies Mazur's finiteness condition $\Phi_p$ (see \cite[2.1.2 Assumption 1]{pan2022fontaine}). Let $\mathcal{O}$ be a complete DVR of characteristic $0$ with residue field $\mathbb{F}$ of characteristic $p>2$. Let $T_\F: G \to \mathbb{F}$ be a two-dimensional pseudo-representation. In our case, the functor sending each object $R$ of $\mathcal{C}_{\mathcal{O}}$ to the set of pseudo-representations $T: G \to R$ which lift $T_\F$ is pro-represented by a CNL $\cO$-algebra $R^{\ps}_{T_\F}$. This is the universal pseudo-deformation ring of the residual pseudo-representation $T_\F$.
	
	\subsection{Local deformation rings}\label{sec for local}
	
	We record some properties of local pseudo-deformation rings.
	
	Let $p$ be an odd prime and let $K$ be a $p$-adic local field with its ring of integers $\mathcal{O}$, a uniformizer $\varpi$ and its residue field $\mathbb{F}$. Let $\chi :G_{\Q_p} \to \mathcal{O}^\times$ be a continuous character with its residual character $\bar{\chi}$.
	
	Let $T: G_{\Q_p} \to \F$ be a two-dimensional pseudo-representation with determinant $\bar{\chi}$. As the profinite group $ G_{\Q_p}$ satisfies Mazur's finiteness condition $\Phi_p$ (for example, see \cite[Theorem 7.5.12]{Neukirch_2008}), there exists a universal pseudo-deformation ring $R_p^{\mathrm{ps},\chi}$ parametrizing all the pseudo-deformations of $T$ with determinant $\chi$. The ring $R_p^{\mathrm{ps},\chi}$ is a CNL $\cO$-algebra.
	
	\begin{prop}\label{local pseudo}
		The ring $R_p^{\mathrm{ps},\chi}$ is an $\cO$-flat integral domain of Krull dimension $4$.
	\end{prop}
	
	\begin{proof}
		If $T=\mathbf{1}+\mathbf{1}$ is the sum of two trivial characters after some twist, then by \cite[Corollary 9.13]{Pa_k_nas_2013}, $R_p^{\mathrm{ps},\chi}$ is a power series ring over $\cO$ of relative dimension $3$. Thus, our result holds in this case.
		
		Now we assume $T$ is multiplicity-free (in the sense of \cite[Definition 1.4.1]{bc09}). By \cite[Corollary 5.11]{B_ckle_2023}, $R_p^{\mathrm{ps},\chi}$ is an $\cO$-flat integral domain. We only need to show that $R_p^{\mathrm{ps},\chi}$ is of Krull dimension $4$.
		
		If $T$ is reducible, after a finite-order twist and the corresponding change of determinant, we may assume that $T=\mathbf{1}+\bar\chi$. Then $\operatorname{Ext}_{\F[G_{\Q_p}]}^1 (\bar{\chi}, \mathbf{1})$ is non-trivial (for example, see \cite[Theorem 7.3.2]{Neukirch_2008}). We can choose a non-split extension $0 \to \mathbf{1} \to \bar{\rho} \to \bar{\chi} \to 0$. If $T$ is absolutely irreducible, it defines an absolutely irreducible Galois representation $\bar{\rho} $ of $ G_{\Q_p}$. In each case, there exists a universal deformation ring $R_{\bar{\rho}}^{\chi}$ parametrizing all deformations of $\bar{\rho}$ with determinant $\chi$. Combining \cite[Theorem 1.5]{B_ckle_2023} and \cite[Proposition 2.1]{Khare_2009}, we know that  $R_{\bar{\rho}}^{\chi}$ is $\cO$-flat of Krull dimension $4$. 
		
		Choose an $\cO'$-irreducible lifting $\rho$ of $\bar{\rho}$ (using \cite[Proposition 1.12]{B_ckle_2023}), where $\cO'$ is some finite normal extension of $\cO$ with fraction field $K'$. Then we have natural maps $R_p^{\mathrm{ps},\chi} \to R_{\bar{\rho}}^{\chi} \to \cO' $, and suppose $\rho$ defines the prime $\kp$ (resp. $\kp'$) of $ R_p^{\mathrm{ps},\chi}$ (resp. $ R_{\bar{\rho}}^{\chi}$). Consider the universal deformation ring $R_{\rho}$ of the representation $\rho \otimes K'$ with determinant $\chi$, and by \cite[Corollary 2.2.2]{pan2022fontaine}, $R_{\rho}$ is isomorphic to $ \widehat{(R_p^{\mathrm{ps},\chi})_{\kp}}$. By \cite[Proposition 9.5]{kisin2003overconvergent}, $R_{\rho}$ is isomorphic to $\widehat{(R_{\bar{\rho}}^{\chi}[\frac{1}{p}])_{\kp'}} \otimes K' $, which is of Krull dimension $3$. Since $R_p^{\mathrm{ps},\chi}$ is a complete Noetherian local domain and $\dim R_p^{\mathrm{ps},\chi}/\kp=1$, the dimension formula gives $\dim R_p^{\mathrm{ps},\chi}=\dim (R_p^{\mathrm{ps},\chi})_{\kp}+\dim R_p^{\mathrm{ps},\chi}/\kp=4$.
	\end{proof}
	
	\begin{rem}
		 In this article, we are concerned with the case where $T$ is reducible. If we are not in a special case $T=\mathbf{1}+\omega$ and $p=3$, the ring $R_p^{\mathrm{ps},\chi}$ has an explicit description. See \cite[Corollary 9.13, B.5,  Proposition B.17 and Remark B.28]{Pa_k_nas_2013} for more details.
		
		Other ways to prove the dimension assertion are to use \cite[Theorem 1]{Bockle_2023} or \cite[Corollary 13.31 \& 15.3]{PQ26}. In general, for any $p$-adic local field $L$, the $n$-dimensional local pseudo-deformation ring with fixed determinant has dimension $(n^2-1)[L:\Q_p]+1 $.
	\end{rem}
	
	If the pseudo-representation $T$ is multiplicity-free, for a prime $\kp$ of $R_p^{\mathrm{ps},\chi}$, we say that $\kp$ is \textit{ordinary} if the pseudo-representation defined by $\kp $ is ordinary, i.e. a sum of two characters. Let $R_p^{\ps, \ord}$ be the local pseudo-deformation ring parametrizing all ordinary liftings of $1+\bar{\chi}|_{G_{\mathbb{Q}_p}}$ with determinant $\chi$. The existence of such ring is proved in \cite[Proposition 1.5.1]{bc09}.
	
	\begin{lem}\label{dim of v ord}
		Assume that $T$ is multiplicity-free. The ring $R_p^{\ps, \ord}$ is a power series ring over $\cO$ of relative dimension $2$. 
	\end{lem}
	
	\begin{proof}
		Outside the special case \(\bar\chi=\omega\) and \(p=3\), this is proved in \cite[Corollary B.6, Remark B.28]{Pa_k_nas_2013}. It remains to treat the special case \(\bar\chi=\omega\) and \(p=3\).
		
			Let $(R_1, \psi_1^{\operatorname{univ}})$ and $(R_{\bar{\chi}}, \psi_{\bar{\chi}}^{\operatorname{univ}})$ be the universal deformation of the $1$-dimensional character $\mathbf{1}$ and $\bar{\chi}|_{G_{\mathbb{Q}_p}}$ respectively. It is well known that both $R_1$ and $ R_{\bar{\chi}}$ are power series rings over $\cO$ of relative dimension $2$. On the one hand, by the universal property, we have natural maps $ R_p^{\ps, \ord} \to R_1 \to R_p^{\ps, \ord}$. The first one is defined by the pseudo-representation $ \psi_1^{\operatorname{univ}}+\chi( \psi_1^{\operatorname{univ}})^{-1}$. If $T=\psi_1+\psi_{\bar{\chi}}$ is the universal ordinary pseudo-deformation with determinant $\chi$, then the second one is defined by $\psi_1$ (lifting the trivial character). Hence, the composite map is an identity, which implies that the map $R_1 \to R_p^{\ps, \ord} $ is a surjection. On the other hand, let $R_p^{\ps, \operatorname{red}}$ be the universal ordinary pseudo-deformation ring (without fixed determinant). By \cite[Proposition 2.3.2]{Zhang:2024aa}, this is a power series ring over $\cO$ of relative dimension $4$. By definition, we have $R_p^{\ps, \ord}= R_p^{\ps, \operatorname{red}} \otimes_{R_{\bar{\chi}}} \cO$, where the map $R_{\bar{\chi}} \to \cO $ is defined by $\chi|_{G_{\mathbb{Q}_p}}$. Then by Krull's principal ideal theorem, we know that $R_p^{\ps, \ord}$ is of dimension at least $3$. Thus, $R_p^{\ps, \ord}$ is a power series ring over $\cO$ of relative dimension $2$.
	\end{proof}
	
	\begin{rem}\label{pseudo not free}
		If $T$ is not multiplicity-free, one can calculate the reducible locus directly by the explicit description of $R_p^{\mathrm{ps},\chi}$ in \cite[Section 9]{Pa_k_nas_2013}. In this case, $R_p^{\ps, \ord}$ is not regular.
		
	\end{rem}
	
	Let $\bar{\rho}: G_{\Q_p} \to \GL_2(\F)$ be a Galois representation such that $\operatorname{tr} \bar{\rho}= T$; if $T$ is reducible, we assume that $\bar{\rho}$ is non-split. We write $R_{\bar{\rho}}^{\square, \chi}$ for the universal framed deformation ring of $\bar{\rho}$ with determinant $\chi$. Assume that $\chi$ is a crystalline character lifting $\bar{\chi}$. We say that a $\overbar{\Q_p}$-point $x \in \Spec R_{\bar{\rho}}^{\square, \chi}$ is \textit{irreducible regular crystalline} if the associated Galois representation $G_{\Q_p} \to \GL_2(\overbar{\Q_p})$ is irreducible and crystalline with distinct Hodge-Tate weights.
	
	\begin{prop}\label{local framed deformation ring} Assume that $\chi$ is a crystalline character lifting $\bar{\chi}$. 
		
		1) The ring $R_{\bar{\rho}}^{\square, \chi} $ is an $\cO$-flat integral domain of relative dimension $6$.
		
		2) The set of irreducible regular crystalline points in $\Spec R_{\bar{\rho}}^{\square, \chi} $ is Zariski dense.
	\end{prop}
	
	\begin{proof}
		 The first part is \cite[Corollary 5.4 \& Theorem 5.6]{B_ckle_2023}. The second part is a combination of \cite[Theorem 1.2]{Bockle:2022aa} and \cite[Corollary 5.7]{B_ckle_2023}.
	\end{proof}
	
	We say that a prime of $R_p^{\mathrm{ps},\chi}$ is \textit{irreducible, regular crystalline} if it is the inverse image of the composite map $R_p^{\mathrm{ps},\chi} \to R_{\bar{\rho}}^{\square,\chi} \twoheadrightarrow R_{\bar{\rho}}^{\square,\chi}/\kp$, where $\kp$ is an irreducible regular crystalline point for some framed local deformation ring $R_{\bar{\rho}}^{\square,\chi}$.
	
	\begin{prop} \label{density of crys}
		Assume that $\chi$ is a crystalline character lifting $\bar{\chi}$. 
		
		1) The set of irreducible regular crystalline points in $\Spec R_p^{\mathrm{ps},\chi} $ is Zariski dense.
		
		2) There exist integers $a \ne b$ such that there are infinitely many irreducible, crystalline points in $\Spec R_p^{\mathrm{ps},\chi}$ with Hodge-Tate weights $\{a, b\}$.
	\end{prop}
	
	\begin{proof}
		Let $\alpha:R_p^{\mathrm{ps},\chi}\to R_{\bar\rho}^{\square,\chi}$ be the natural map induced by the trace of the universal framed deformation. By Proposition \ref{local framed deformation ring}, we may choose an irreducible one-dimensional point $\mathfrak q\in\Spec R_{\bar\rho}^{\square,\chi}$ of characteristic $0$, and put $\mathfrak q^{\mathrm{ps}}:=\alpha^{-1}(\mathfrak q)$. Let $\rho_{\mathfrak q}$ be the corresponding absolutely irreducible representation. After the faithfully flat scalar extension $k(\mathfrak q^{\mathrm{ps}})\to k(\mathfrak q)$, \cite[Corollary 2.2.2]{pan2022fontaine} identifies the completion of $\alpha_{\mathfrak q}$ with the natural morphism from the unframed to the framed deformation ring of $\rho_{\mathfrak q}$. The proof of \cite[Proposition 2.1]{Khare_2009} applies without change to Artinian local $k(\mathfrak q)$-algebras and shows that this morphism is formally smooth of relative dimension $4-\dim_{k(\mathfrak q)}H^0(G_{\Q_p},\operatorname{ad}\rho_{\mathfrak q})=3$. The faithful flatness of these maps shows that $\alpha_{\mathfrak q}$ is injective. Since $R_p^{\mathrm{ps},\chi}$ is a domain by Proposition \ref{local pseudo}, it follows that $\alpha$ is injective. Thus the induced morphism on spectra is dominant, and the first assertion follows from Proposition \ref{local framed deformation ring}.
		
		For the second assertion, choose an irreducible regular crystalline point $\widetilde{\kp}$ of weights $\{a,b\}$ and an irreducible component $C$ of $\Spec R_{\bar\rho}^{\square,\operatorname{crys},\chi,[a,b]}[1/p]$ containing it. By \cite[Theorem 3.3.8]{Kisin2008} or \cite[Corollary 3.3.3]{Allen_2019}, $C$ has dimension $4$. Its absolutely irreducible locus $C^{\operatorname{irr}}$ is nonempty and open. At any closed point of this locus, the completed local ring of the fibre over its image in $\Spec R_p^{\mathrm{ps},\chi}[1/p]$ has dimension at most $3$: after extending residue fields, it is a quotient of the completed fibre of the formally smooth map of relative dimension $3$ in (1). If there were only finitely many images of closed points, their closed fibres would cover $C^{\operatorname{irr}}$ by the Jacobson property of the generic fibre. One of those fibres would then contain $C^{\operatorname{irr}}$ and have dimension $4$ at a closed point, a contradiction.
	\end{proof}

	\subsection{Global deformation rings}
	In this subsection, we collect some results about global pseudo-deformation rings.
	
	Let $p$ be an odd prime and let $K$ be a $p$-adic local field with its ring of integers $\mathcal{O}$, a uniformizer $\varpi$ and its residue field $\mathbb{F}$. 
	
	Let $F$ be an abelian totally real field. For simplicity, we suppose that $p$ splits completely in $F$ (but it may not be necessary in this subsection). Let $\Sigma_p$ be the set of places of $F$ lying above $p$ and let $\Sigma$ be a finite set of finite places of $F$ containing $\Sigma_p$. Let $\chi :G_{F, \Sigma} \to \mathcal{O}^\times$ be a continuous totally odd character with its residual character $\bar{\chi}$.
	
	Let $T$ be the pseudo-representation $1+\bar{\chi}$. As the profinite group $G_{F, \Sigma} $ satisfies Assumption \ref{odd ass} and Mazur's finiteness assumption (see \cite[Corollary 10.11.15]{Neukirch_2008}), there exists a universal pseudo-deformation $R_{F}^{\ps}$ parametrizing all the pseudo-deformations of $T$. Let $R_{\bar{\chi}}$ be the universal deformation ring parametrizing all the liftings of $\bar{\chi}$. Write $ R_{F}^{\ps, \chi} = R_{F}^{\ps} \widehat{\otimes}_{R_{\bar{\chi}}, \chi} \cO$, where the map $R_{\bar{\chi}} \to \cO $ is given by $\chi$. Then $ R_{F}^{\ps, \chi}$ is the universal pseudo-deformation ring parametrizing all the pseudo-deformations of $T$ with determinant $\chi$.
		
    \begin{prop}\label{det}
    	1) Let $\Gamma$ be the maximal pro-$p$ abelian quotient of $G_{F, \Sigma} $. Then there exists a CNL $\cO$-algebra isomorphism $R_{F}^{\ps} \to R_{F}^{\ps, \chi} \widehat{\otimes}_\cO \cO[[\Gamma]] $.
    	
    	2) Let $\chi': G_{F,\Sigma} \to \cO^\times$ be a character lifting $\bar{\chi}$. Let $R_{F}^{\ps, \chi'}$ be the universal pseudo-deformation ring parametrizing all the pseudo-deformations of $T$ with determinant $\chi'$. Then there exists a CNL $\cO$-algebra isomorphism $ R_{F}^{\ps, \chi} \to R_{F}^{\ps, \chi'}$.
    \end{prop}
	
	\begin{proof}
		This is an analogue of \cite[Lemma 6.1.2]{Allen_2019}. We sketch the proof here.
		
		Let $T^\chi: G_{F, \Sigma} \to R_{F}^{\ps, \chi}$ be the universal pseudo-deformation of $T$ with determinant $\chi$ and let $D: G_{F, \Sigma} \to \cO[[\Gamma]]^\times $ be the universal deformation of the trivial character $\mathbf{1}$. As $p$ is odd, there exists a unique character $\tilde{\chi}$ of $G_{F, \Sigma}$ lifting $\mathbf{1}$ satisfying $(\tilde{\chi})^2=\chi\tilde{\bar{\chi}}^{-1}$, where $ \tilde{\bar{\chi}}$ is the Teichm\"uller lifting of $ \bar{\chi}$. Then $T^\chi \otimes (\tilde{\chi}D): G_{F, \Sigma} \to R_{F}^{\ps, \chi} \widehat{\otimes}_\cO \cO[[\Gamma]]$ defines a pseudo-deformation of $T$.
		
		For a pseudo-deformation $T_A: G_{F, \Sigma} \to A$ lifting $T$, $A \in \mathcal{C}_{\mathcal{O}}$, there exists a unique character $d_A$ of $G_{F, \Sigma}$ (lifting $\mathbf{1}$) satisfying $(d_A)^2= \det(T_A)\chi^{-1}$. Then $T_A \otimes (d_A)^{-1} : G_{F, \Sigma} \to A$ defines a pseudo-deformation of $T$ with determinant $\chi$. Thus, we get natural maps $R_{F}^{\ps, \chi} \to A $ and $ \cO[[\Gamma]] \to A$ (induced by $\tilde{\chi}^{-1}d_A$) by the universal property. By the uniqueness of the universal object, we have  $R_{F}^{\ps} \cong R_{F}^{\ps, \chi} \widehat{\otimes}_\cO \cO[[\Gamma]] $. This shows the first claim.
		
		For the second one, we may consider the pseudo-representation $ T^\chi \otimes \mu : G_{F, \Sigma} \to R_{F}^{\ps, \chi}$, where $\mu : G_{F, \Sigma} \to \cO^\times$ is the unique character lifting $\mathbf{1}$ satisfying $\mu^2 = \chi'\chi^{-1}$. This defines a pseudo-deformation of $T$ with determinant $\chi'$, and we get the map $ R_{F}^{\ps, \chi'} \to R_{F}^{\ps, \chi}$ by the universal property. Similarly, we can also obtain the map $ R_{F}^{\ps, \chi} \to R_{F}^{\ps, \chi'}$. The composite maps $R_{F}^{\ps, \chi'} \to R_{F}^{\ps, \chi} \to R_{F}^{\ps, \chi'} $ and $R_{F}^{\ps, \chi} \to R_{F}^{\ps, \chi'} \to R_{F}^{\ps, \chi} $ are both identities. This shows the second claim.
		
			\end{proof}
	
	The following result is crucial to our main results.
	
	\begin{prop}\label{irreducible point}
		Let $\kp$ be a one-dimensional prime of $ R_{F}^{\ps, \chi}$ such that the corresponding Galois representation is absolutely irreducible in the sense of Definition \ref{construction}. Here we do not require $\varpi \notin \kp$.
		
	   1)	 We have $c((R_{F}^{\ps, \chi})_{\kp}) \ge 2[F: \Q]-1$.
		
		2) For any irreducible component $C$ of $ \Spec R_{F}^{\ps, \chi}$ containing $\kp$, we have $\dim C \ge 2[F:\Q]+1$.
	\end{prop}
	
	\begin{proof}
		For the first assertion, it is proved in \cite[Lemma 7.4.6]{pan2022fontaine}. Here we only point out that the inequality $c((R_{F}^{\ps, \chi})_{\kp}) \ge  c(\widehat{(R_{F}^{\ps, \chi})_{\kp}})$ follows from Proposition \ref{connected dim}.
		
		For the second assertion, it is a direct consequence of the first one. See \cite[Remark 3.2.8]{Zhang:2024ab}.
	\end{proof}
	
	\begin{cor}\label{structure of ps}
		For any irreducible component $C$ of $ \Spec R_{F}^{\ps, \chi}$, $C$ is either of dimension at least $1+2[F:\Q]$ or at most $2$.
	\end{cor}
	
	\begin{proof}
		If $C$ is not contained in the reducible locus, then its irreducible locus is a non-empty open subset; by prime avoidance and cutting down inside $C$, it contains a one-dimensional irreducible prime. Proposition \ref{irreducible point} then gives $\dim C\ge 1+2[F:\Q]$. Thus, if $C$ is of dimension less than $1+2[F:\Q]$, it is contained in the reducible locus $\Spec R_{F}^{\operatorname{ps, red}}$. There exists a natural surjection $\cO[[G_{F, \Sigma}^{\operatorname{ab}}(p)]] \twoheadrightarrow R_{F}^{\operatorname{ps, red}}$. As $F$ is abelian, Leopoldt's conjecture gives $\dim R_{F}^{\operatorname{ps, red}}\le 2$, and hence $\dim C\le 2$.
	\end{proof}
	
    Assume that there exists a one-dimensional irreducible pseudo-deformation of $T$ with determinant $\chi$. Let $I_F$ be the intersection of all minimal primes of $R_{F}^{\ps, \chi} $ of dimension at least $2[F:\Q]+1 $, and write $ R_{F,0}^{\ps, \chi} := R_{F}^{\ps, \chi}/I_F$. Then from our definition, we can conclude the following properties of $ R_{F,0}^{\ps, \chi}$ directly.
    
    \begin{prop}\label{def of 0}
    	1) $ R_{F,0}^{\ps, \chi}$ is reduced. Each irreducible component of $ R_{F,0}^{\ps, \chi}$ is of Krull dimension at least $1+2[F:\Q]$.
    	
    	2) For any irreducible prime $\kp$ of $R_{F}^{\ps, \chi} $, the natural surjection $ R_{F}^{\ps, \chi} \twoheadrightarrow R_{F}^{\ps, \chi}/\kp$ factors through $ R_{F,0}^{\ps, \chi}$.
    \end{prop}
    
    Now let $F_1/F$ be a totally real extension and suppose $F_1$ is also abelian. Let $\Sigma_1$ be the set of primes of $F_1$ lying above $\Sigma$. Let $ R_{F_1}^{\ps, \chi}$ be the universal pseudo-deformation ring parametrizing all the pseudo-deformations of $T|_{G_{F_1, \Sigma_1}}$ with determinant $\chi|_{G_{F_1, \Sigma_1}}$. By the universal property, there is a natural map $ R_{F_1}^{\ps, \chi} \to R_{F}^{\ps, \chi}$.
    
    \begin{prop}\label{finiteness of ps}
      The natural map $ R_{F_1}^{\ps, \chi} \to R_{F}^{\ps, \chi}$ is finite.
    \end{prop}
	
	\begin{proof}
		Write $ \km_1$ for the maximal ideal of $R_{F_1}^{\ps, \chi}$. We only need to show that $R_{F}^{\ps, \chi}/ (\km_1)$ is an Artinian ring. If not, we can find a one-dimensional prime $\kq$ of $R_F^{\ps, \chi}/ (\km_1)$. Consider the Galois representation $\rho(\kq) : G_{F, \Sigma} \to \GL_2(k(\kq))$ in the sense of Definition \ref{construction} (here we fix $\sigma_0, \tau_0 \in G_{F, \Sigma}$).
		
		  We claim that $\rho(\kq)|_{G_{F_1, \Sigma_1}}$ is either upper triangular or lower triangular. From our construction, we have $y(\alpha, \beta)=0$, $\alpha, \beta \in G_{F_1, \Sigma_1}$. If it is not lower triangular, i.e. $y(\alpha, \tau_0) \ne 0$ for some $\alpha \in G_{F_1, \Sigma_1}$, then $0=y(\alpha, \beta)y(\sigma_0, \tau_0)=y(\alpha, \tau_0)y(\sigma_0, \beta)$ (by \cite[2.1.4]{pan2022fontaine}), which implies $ y(\sigma_0, \beta)=0$. Then by Definition \ref{construction}, $\rho(\kq)|_{G_{F_1, \Sigma_1}}$ is 
		  upper triangular.
		  
		  We assume $\rho(\kq)|_{G_{F_1, \Sigma_1}}$ is upper triangular (the other case is similar). Note that $\bar{\chi}$ is totally odd. For any complex conjugation $c \in G_{F_1, \Sigma_1}$, by our definition, we have $ \rho(\kq)(c)=\begin{pmatrix}
				1 & *\\
				~& -1
			\end{pmatrix}$ (as $a|_{G_{F_1, \Sigma_1}}=\mathbf{1}, d|_{G_{F_1, \Sigma_1}}=\bar{\chi}|_{G_{F_1, \Sigma_1}}$ for functions $a, d$ in Definition \ref{a,d}). For any $x \in G_{F, \Sigma}$, we have $\rho(\kq)(xcx^{-1}) = \rho(\kq)(c')$, where $c, c'$ are complex conjugations of $G_{F_1, \Sigma_1} $. Then by an easy calculation, $\rho(\kq)(x)$ is upper triangular. Thus the pseudo-representation defined by $\kq$ has the form $\psi_1+\psi_2$, where $\psi_1$ lifts $\mathbf{1}$ and $\psi_1|_{G_{F_1,\Sigma_1}}=\mathbf{1}$. Since $\varpi\in\kq$, the character $\psi_1$ is valued in a field of characteristic $p$; it factors through the finite group $\Gal(F_1/F)$ and reduces to $\mathbf{1}$. Its prime-to-$p$ part injects under reduction, while a field of characteristic $p$ has no non-trivial $p$-power roots of unity; hence $\psi_1=\mathbf{1}$. It follows that $\psi_2=\bar\chi$, so $\kq$ is zero-dimensional, a contradiction.
	\end{proof}
	
	\begin{cor}\label{lower}
		Assume that $\bar{\chi}$ can be extended to a character of $G_\Q$. Then we have $\dim R_{F}^{\ps, \chi} \ge 1+2[F:\Q]$. In other words, $\Spec R_{F, 0}^{\ps, \chi}$ is not empty.
	\end{cor}
	
	\begin{proof}
		Note that $\bar{\chi} $ can be extended to a character of $G_{\Q}$. By the arguments in \cite[Lemma 7.3]{Fakhruddin2022}, the pseudo-deformation space $\Spec R_\Q^{\ps, \chi'}$ (of the pseudo-representation $1+\bar{\chi}$ with determinant $\chi'$) has a one-dimensional irreducible point associated to some cusp eigenform $f$ of tame level $N_0$ (the Artin conductor of the representation $1 \oplus \bar{\chi}$), where $\chi'$ is the corresponding determinant lifting $\bar{\chi}$. Then by Proposition \ref{irreducible point}, we have $\dim R_\Q^{\ps, \chi'} \ge 3$. Consider the natural map $ R_{F}^{\ps, \chi'} \to R_\Q^{\ps, \chi'}$. By Proposition \ref{finiteness of ps}, we have $\dim R_{F}^{\ps, \chi'} \ge \dim R_\Q^{\ps, \chi'} >2$, and hence $\dim R_{F}^{\ps, \chi'} \ge 1+2[F:\Q]$ by Corollary \ref{structure of ps}. Using 2) of Proposition \ref{det}, we get $\dim R_{F}^{\ps, \chi} \ge 1+2[F:\Q]$. In particular, $\Spec R_{F, 0}^{\ps, \chi}$ is not empty.
	\end{proof}
	
	As above, define $R_{F_1,0}^{\ps, \chi}:=R_{F_1}^{\ps, \chi}/I_{F_1}$ to be the union of all irreducible components of $R_{F_1}^{\ps, \chi}$ of dimension at least $1+2[F_1: \Q]$.
	
	\begin{cor}\label{finite for 0}
		The composite map $ R_{F_1}^{\ps, \chi} \to R_{F}^{\ps, \chi} \twoheadrightarrow R_{F,0}^{\ps, \chi}$ factors through $ R_{F_1,0}^{\ps, \chi}$. In other words,  there exists a well-defined finite map $ R_{F_1,0}^{\ps, \chi} \to R_{F,0}^{\ps, \chi}$ induced by the universal property.
	\end{cor}
	
	\begin{proof}
			For the first assertion, let $I:= \operatorname{ker} \{R_{F_1}^{\ps, \chi} \to R_{F,0}^{\ps, \chi} \}$. Since $R_{F,0}^{\ps,\chi}$ is reduced, $I$ is radical. If $\kP$ is a minimal prime of $I$, then Proposition \ref{finiteness of ps} gives $\dim R_{F_1}^{\ps,\chi}/\kP\ge 1+2[F:\Q]>2$. Choose a minimal prime $\kq$ of $R_{F_1}^{\ps,\chi}$ contained in $\kP$. Corollary \ref{structure of ps} implies that $\kq$ is one of the high-dimensional minimal primes defining $I_{F_1}$; hence $I_{F_1}\subset\kq\subset\kP$. Intersecting over all such $\kP$ gives $I_{F_1}\subset I$. The second assertion follows from Proposition \ref{finiteness of ps}.
	\end{proof}
	
	\section{Local-global compatibility}\label{sec 3}
	
	In this section, we recall some local-global compatibility results and deduce some corollaries. Our references are \cite{Pa_k_nas_2013}, \cite{Paskunas_2021}, \cite{Pa_k_nas_2021} and \cite{pan2022fontaine}.
	
	Throughout this section, we fix an odd prime $p$, and a totally real field $F$ in which $p$ splits completely. We denote by $K$  a $p$-adic local field with its ring of integers $\cO$, a uniformizer $\varpi$ and its residue field $\F$.
	
	\subsection{Completed cohomology and Hecke algebra}\label{sec for T}
	In this subsection, we recall some definitions about the completed cohomology for quaternionic forms and the $p$-adic big Hecke algebra. Our reference is \cite[Section 3]{pan2022fontaine}.
	
	Let $D$ be a quaternion algebra with centre $F$ such that $D$ is ramified at all infinite places of $F$ and unramified at all places lying above $p$. 
	
	Let $A$ be a topological $\cO$-algebra and $U = \prod_{v \nmid \infty} U_v$ be an open compact subgroup of $(D \otimes_F \mathbb{A}_F^\infty)^\times$ such that $U_v \subseteq \operatorname{GL}_2(\cO_{F_v})$ for any $v \mid p$. We write $U^p= \prod_{v \nmid p} U_v$ (the tame level) and $U_p =\prod_{v \mid p} U_v$. Let $\psi: (\mathbb{A}_F^\infty)^\times/F_+^\times \to A^\times$ be a continuous character, where $F_+$ is the set of totally positive elements in $F$.
	
    Let $S_{\psi}(U,A)$ be the space of continuous functions:
    \[f:D^{\times}\setminus \DAi\to A\]
    such that for any $g\in \DAi,u\in U,z\in \AFi$, we have
    \begin{itemize}
    	\item $f(gu)=f(g)$,
    	\item $f(gz)=\psi(z)f(g)$.
    \end{itemize}
    
     Let $\Sigma$ be a finite set of places of $F$ containing all those lying above $p$ and places $v$ where either $D$ is ramified or $U_v$ is not a maximal open subgroup. For any $v \notin \Sigma$, we define the Hecke operator $T_v \in \operatorname{End}(S_{\psi}(U,A))$ to be the double coset action $[U_v \begin{pmatrix}
    	\varpi_v & ~\\
    	~ & 1
    \end{pmatrix}U_v]$. Precisely, if we write $U_v \begin{pmatrix}
    	\varpi_v & ~\\
    	~ & 1
    \end{pmatrix}U_v  = \coprod_i \gamma_i U_v $, then $$
    (T_v \cdot f) (g) = \sum_i f(g\gamma_i), ~~f \in S_{\psi}(U,A). $$
    
    We define the Hecke algebra $\mathbb{T}_{\psi} ^\Sigma (U,A) \subseteq \operatorname{End}(S_{\psi}(U,A))$ as the $A$-subalgebra generated  by all Hecke operators $T_v, v \notin \Sigma$. In the case that $A= \mathcal{O}/ \varpi^n$, $ U$ is sufficiently small (in the sense of \cite[page 1046]{pan2022fontaine}) and $\psi \mid_{U \cap (\mathbb{A}_F^\infty)^\times} $ is trivial modulo $ \varpi^n$, we can check that the Hecke algebra $\mathbb{T}_{\psi}^\Sigma (U, \mathcal{O}/ \varpi^n)$ is independent of $\Sigma$, and we drop $\Sigma$ afterwards.
    
    For a tame level $U^p$, we write $S_{\psi}(U^p,A):=\varinjlim_{U_p} S_\psi(U^pU_p,A) $ with discrete topology, where $U_p=\prod_{v|p}U_v$ ranges over all open compact subgroups $U_v$ of $\GL_2(F_v)$.

    \begin{defn} \label{Completed coh}
    	The \textit{completed cohomology} of level $U^p$ is defined to be
    	\[S_\psi(U^p):=\Hom_\cO(K/\cO,S_{\psi}(U^p,K/\cO))\]
    	equipped with $\varpi$-adic topology, and the \textit{completed homology} is defined to be
    	\[M_\psi(U^p):=S_\psi(U^p,K/\cO)^\vee=\Hom_\cO^{\cont}(S_{\psi}(U^p,K/\cO),K/\cO)\]
    	equipped with compact-open topology.
    \end{defn}
    
    The following proposition is clear from the definition.
    
    \begin{prop}\label{3.1.2}
    	1) $ S_\psi(U^p)$ is $\varpi$-torsion free.
    	
    	2) $ S_\psi(U^p) \cong \Hom^{\mathrm{cont}}_{\cO}(M_\psi(U^p),\cO)$.
    \end{prop}
    
   	\begin{defn} \label{big}
   	Let $U^p$ be a tame level and $\psi: (\mathbb{A}_F^\infty)^\times/ (U^p \cap (\mathbb{A}_F^\infty)^\times)F_+^\times \to \mathcal{O}^\times$ be a continuous character. Define the \textit{big Hecke algebra} $$
   	\mathbb{T}_{\psi}(U^p) = \varprojlim_{(n, U_p) \in \mathcal{I}} \mathbb{T}_{\psi}(U^pU_p, \mathcal{O}/\varpi^n), $$
   	where $\mathcal{I}$ is the set of pairs $(n, U_p)$ with $U_p \subseteq K_p :=\prod_{v \mid p} \operatorname{GL}_2(\cO_{F_v})$ and $n$ a positive integer such that $ \psi \mid_{U_p \cap \mathcal{O}_{F,p}^\times} \equiv 1~ \operatorname{mod}~\varpi^n$.
   	
   \end{defn}
    
    The following result is classical.
    
    \begin{prop}\label{pseudo of T}
    	1) $\mathbb{T}_{\psi}(U^p) $ is a semi-local ring. Write $\km_1,\cdots,\km_r$ for all the maximal ideals of $\T_\psi(U^p)$. Then we have $\T_\psi(U^p)\cong \T_\psi(U^p)_{\km_1}\times\cdots\times \T_\psi(U^p)_{\km_r}$ and each $\T_\psi(U^p)_{\km_i}$ is $\km_i$-adically complete and separated. 
    	
    	2) For each maximal ideal $\km_i$, there exists a two-dimensional pseudo-representation $T_{\km_i}: G_{F, \Sigma} \to \T_\psi(U^p)_{\km_i}$ sending $\Frob_v$ to $T_v$ for any $v \notin \Sigma$ with determinant $\psi\varepsilon^{-1}$.
    \end{prop}
    
    \subsection{Pašk\=unas theory}\label{pas sec}
    
    In this subsection, we summarize some basic definitions and results from Pašk\=unas' theory for the $p$-adic analytic group $D_p^\times:=\prod_{v|p} \GL_2(F_v)$. For precise definitions, see \cite{Pa_k_nas_2013}, \cite{Pa_k_nas_2021}, or the summary for $\GL_2(\Q_p)$ in \cite[Appendix A]{Zhang:2024aa}. We retain the notation and terminology used there.
    
    Let $G=\prod_{v|p} \GL_2(F_v) \cong \prod_{v|p}\GL_2(\Q_p)$ and $Z(G)\simeq \prod_{v|p} F_v^\times$ be its centre. Let $\zeta: Z(G) \to \cO^\times $ be a continuous character.
    
    Let $\operatorname{Mod}^{\mathrm{sm}}_{G, \zeta}(\cO)$ be the category of smooth 
    representations of $G$ on $\mathcal{O}$-torsion modules with central character $\zeta$, i.e. the centre $Z$ acts via $\zeta$. Let $\mathrm{Mod}^{\mathrm{l\,fin}}_{G, \zeta}(\cO)$ be the full subcategory of $\mathrm{Mod}_{G, \zeta}^{\mathrm{sm}}(\cO)$ consisting of representations locally of finite length and $\mathrm{Mod}^{\mathrm{l\, adm}}_{G, \zeta} (\cO)$ the full subcategory consisting of locally admissible representations. By \cite[Lemma 3.4.1]{pan2022fontaine}, $\mathrm{Mod}^{\mathrm{l\,fin}}_{G, \zeta}(\cO)$ and $\mathrm{Mod}^{\mathrm{l\, adm}}_{G, \zeta} (\cO)$ are actually the same.
    
    Let $H \subset G$ be an open compact subgroup and $\cO[[H]]$ be the completed group algebra of $H$ over $\cO$. Let $\mathrm{Mod}_{G, \zeta}^{\mathrm{pro\,aug}}(\cO)$ be the category of profinite linearly topological $\cO[[H]]$-modules with action of $\cO[G]$ and on which the centre acts through $\zeta^{-1}$ such that both actions agree on $\cO[H]$, with morphisms given by $G$-equivariant continuous homomorphisms of topological
    $\cO[[H]]$-modules. This category is independent of the subgroup $H$. 
    
    For an object $M$ of $\operatorname{Mod}^{\mathrm{sm}}_{G, \zeta}(\cO)$, the Pontryagin dual $M \mapsto M^\vee$ induces an anti-equivalence between categories $\operatorname{Mod}^{\mathrm{sm}}_{G, \zeta}(\cO)$ and $\mathrm{Mod}_{G, \zeta}^{\mathrm{pro\,aug}}(\cO)$. Write $\kC_{G,\zeta}(\cO)$ for the full subcategory of $\mathrm{Mod}_{G, \zeta}^{\mathrm{pro\,aug}}(\cO)$ with objects isomorphic to $\pi^\vee$ for some object $\pi$ in $\mathrm{Mod}^{\mathrm{l\, adm}}_{G,\zeta}(\cO)$.
    
    Let $\mathrm{Irr}_{G,\zeta}$ be the set of irreducible representations in $\mathrm{Mod}_{G,\zeta}^{\mathrm{sm}}(\cO)$. Recall that a \textit{block} is an equivalence class of the relation $\sim$, where $\pi \sim \pi'$ if there exists $\pi_1, \cdots, \pi_n \in \Irr_{G, \zeta}$ such that $\pi \cong \pi_1$, $\pi' \cong \pi_n$, and for $1 \leq i \leq n-1$, $\pi_i \cong \pi_{i+1}$ or $\Ext^1_{G}(\pi_i, \pi_{i+1}) \neq 0$ or $\Ext^1_{G}(\pi_{i+1}, \pi_i) \neq 0$.
    Under this equivalence relation, we have a decomposition of the category $\mathrm{Mod}^{\mathrm{l\, adm}}_{G, \zeta} (\cO)$:
    \begin{eqnarray*}
    	\mathrm{Mod}^{\mathrm{l\, adm}}_{G,\zeta}(\cO)\cong \prod_{\mathfrak{B}\in\mathrm{Irr}_{G,\zeta}/\sim}\mathrm{Mod}^{\mathrm{l\, adm}}_{G,\zeta}(\cO)^\kB,
    \end{eqnarray*}
    where $\mathrm{Mod}^{\mathrm{l\, adm}}_{G,\zeta}(\cO)^\kB$ is the full subcategory of $\mathrm{Mod}^{\mathrm{l\, adm}}_{G,\zeta}(\cO)$ consisting of representations with all irreducible subquotients in the block $\kB$. Write $\kC_{G,\zeta}(\cO)^\kB$ for the full subcategory of $\kC_{G,\zeta}(\cO)$ anti-equivalent to $\mathrm{Mod}^{\mathrm{l\, adm}}_{G,\zeta}(\cO)^\kB$ for the block $\kB$ via the Pontryagin dual. Then we also have a natural decomposition \[\kC_{G,\zeta}(\cO)\cong\prod_{\mathfrak{B}\in\mathrm{Irr}_{G,\zeta}/\sim}\kC_{G,\zeta}(\cO)^\kB.\]
    
    In our case, the blocks containing an absolutely irreducible representation have been well-understood. See the list in \cite[Appendix A.2]{Zhang:2024aa} (containing the case $p=2$) for the case $G= \GL_2(\Q_p)$ and combine \cite[Lemma 3.4.7]{pan2022fontaine}. For each $v|p$ and a block $\kB$, we can attach a semi-simple two-dimensional representation $(\bar{\rho}_{\kB})_v$ over $\F$ with determinant $\varepsilon\zeta|_{F_v} \modd \varpi$. Here we view $ \zeta|_{F_v}$ as a character of $ G_{F_v}$ by local class field theory. 
    
    Now fix a block $\kB$, and write $\pi_{\kB}=\bigoplus_{\pi\in\kB_i}\pi$, where $\kB_i$ is the set of isomorphism classes of elements of $\kB$. This is a finite set by the discussion in the previous paragraph. Let $\pi_\kB\hookrightarrow J_{\kB}$ be an injective envelope of $\pi_{\kB}$ in $\mathrm{Mod}^{\mathrm{l\, adm}}_{G,\zeta}(\cO)$. Then its Pontryagin dual $P_{\kB}:=J_{\kB}^\vee$ is a projective envelope of $\pi^\vee\cong\bigoplus_{\pi\in\kB_i}\pi^\vee$ in $\kC_{G,\zeta}(\cO)$, called a \textit{projective generator} of $\kB$. Write
    $E_{\kB}:=\End_{\kC_{G,\zeta}(\cO)}(P_{\kB})\cong \End_G(J_{\kB})^{\operatorname{op}}.$
    This is a pseudo-compact ring (not necessarily commutative). Let $Z_{\kB}$ be the centre of $ E_{\kB}$. 
    
    \begin{thm} \label{Paskunas theory for G}
    	Write $R_v^{\ps}$ for the universal pseudo-deformation ring of $(\bar{\rho}_{\kB})_{v}$ with determinant $\varepsilon\zeta|_{F_v}$ (via class field theory). 
    	
    	1) There exists a natural finite map $\widehat{\bigotimes}_{v|p} R^{\ps}_{v}\to E_{\kB} $. In particular, $E_{\kB} $ is Noetherian.
    	
    	2) $ E_{\kB}$ is a finitely generated module over $Z_{\kB}$.
    \end{thm}
    
    \begin{proof}
    	See \cite[Corollary 3.4.8]{pan2022fontaine} and \cite[Proposition 3.3.2]{Zhang:2024aa}.
    \end{proof}
    
     The functor sending an object $N$ of the category $ \kC_{G,\zeta}(\cO)$ to  $\Hom_{\kC_{G,\zeta}(\cO)}(P_\kB, N)$ induces an equivalence between $ \kC_{G,\zeta}(\cO)^\kB$ and the category of right pseudo-compact $ E_{\kB}$-modules. The inverse functor is given by $\mm \mapsto \mm \widehat{\otimes}_{E_{\kB}} P_{\kB}$.
  
    Let $\Pi$ be a unitary $K$-Banach space representation of $G$ and $\Theta$ an open bounded $G$-invariant lattice in $\Pi$. By \cite[Lemma 4.4]{Pa_k_nas_2013}, its Schikhof dual
    $ \Theta^d:=\Hom_{\cO}^{\cont}(\Theta, \cO)$ is an object of $\mathrm{Mod}_{G, \zeta}^{\mathrm{pro\,aug}}(\cO)$. 
    
    
    \begin{rem}
    	Fix a tame level $U^p$ and a character $\psi$ as in the previous subsection. Take $\zeta=\psi$. We can conclude the following facts:
    	
    	1) $S_\psi(U^p, K/\cO)$ is a smooth admissible $\cO$-representation of $D_p^\times$.
    	
    	2) The completed homology $M_\psi(U^p)$ is an object in $ \kC_{G,\zeta}(\cO)$. Here we view $\psi$ as a character of $Z(G)$.
    	
    	3) Write $S_\psi(U^p)_K :=S_\psi(U^p) \otimes_\cO K$. Then $S_\psi(U^p)_K$ is a $K$-Banach space representation with open bounded $G$-invariant $\cO$-lattice $S_\psi(U^p) $.
    \end{rem}
    
    \subsection{Local-global compatibility}\label{sec lgc}
    In this subsection, we collect some local-global compatibility results. Our main reference is \cite[Section 3.5]{pan2022fontaine}.
    
    We fix a tame level $U^p$ and character $\psi$. Recall that $F$ is a totally real field in which $p$ splits completely. Let $\Sigma$ be a finite set of places of $F$ containing all those lying above $p$ and places $v$ where either $D$ is ramified or $U_v$ is not a maximal open subgroup. Under this setting, we can define a big Hecke algebra $\T_{\psi}(U^p)$. Let $\km$ be a maximal ideal of $\T_{\psi}(U^p)$. Then by Proposition \ref{pseudo of T}, we get a two-dimensional pseudo-representation $T_\km : G_{F,\Sigma} \to \T_{\psi}(U^p)_\km$ with determinant $\psi\varepsilon^{-1}$ and its residual representation $\bar{\rho}_\km : G_{F,\Sigma} \to \GL_2(\F)$ in the sense of 
    Definition \ref{construction}. 
    
    For each $v|p$, put $\bar{\rho}_{\km, v} := (\bar{\rho}_\km|_{G_{F_v}})^{\mathrm{ss}}$. We can define a block $\kB_v$ of $\GL_2(F_v)$ corresponding to $\bar{\rho}_{\km, v}\otimes \omega$ (see the list in \cite[Appendix A.2]{Zhang:2024aa}) for each $v|p$. Then by \cite[Lemma 3.4.7]{pan2022fontaine}, $\kB_\km := \otimes_{v|p} \kB_v$ defines a block of $D_p^\times=\prod_{v|p} \GL_2(F_v)$. 
    
    Let $P_{\kB_v}$ be a projective generator of $\kB_v$ for each $v|p$. Write $E_{\kB_v}:= \End_{\kC_{\GL_2(F_v),\psi}(\cO)}(P_{\kB_v})$.   Then by \cite[Lemma 3.4.5]{pan2022fontaine}, $P_{\kB_\km} := \widehat{\otimes}_{v|p} P_{\kB_v}$ defines a projective generator of the block $\kB_\km$. Also, by \cite[Proposition 3.3.2]{Zhang:2024aa}, we have $ E_{\kB_\km}:=\End_{\kC_{D_p^\times,\psi}(\cO)}(P_{\kB_\km}) \cong  \widehat{\otimes}_{v|p} E_{\kB_v}$.
    
     Write $$\mm_v := \Hom_{\kC_{\GL_2(F_v),\psi}(\cO)}(P_{\kB_v}, M_\psi(U^p)_\km),~~~ v|p,$$
    $$\mm:=\Hom_{\kC_{D_p^\times,\psi}(\cO)}(P_{\kB_\km}, M_\psi(U^p)_\km).$$
     
     Write $R_v^\ps$ for the universal pseudo-deformation ring of the  pseudo-representation $\tr \bar{\rho}_{\km, v}$ with determinant $\psi\varepsilon^{-1}$, and write $R_{F,p}^\ps := \widehat{\otimes}_{v|p} R_v^\ps$. Then by the universal property, there is a natural map $ R_{F,p}^\ps \to \T_{\psi}(U^p)_\km$ induced by $T_\km$.
     
     We can define two actions of $R_{F,p}^\ps $ on $\mm$. 
     \begin{enumerate}
     	\item \textbf{the Galois side} $\tau_{\Gal}$: which comes from the action of $\T_\psi(U^p)_\km$ on $M_\psi(U^p)_\km$ and the natural map $ R_{F,p}^\ps \to \T_{\psi}(U^p)_\km$.
     	\item \textbf{the automorphic side} $\tau_{\mathrm{Aut}}$: which comes from the action of $R_{F,p}^\ps $ (after twisting $\varepsilon^{-1}$ ) on $P_{\kB_\km}$ via $E_{\kB_\km}$ due to the map $ R_{F,p}^\ps \to E_{\kB_\km}$ (see Theorem \ref{Paskunas theory for G}).
     \end{enumerate}

     We first recall the main local-global compatibility results proved in \cite{pan2022fontaine} and \cite{Zhang:2024aa}.
     
     \begin{thm}[Local-global compatibility]\label{lgc}
     	Keep notations as above.
     	\begin{enumerate}
     	\item We have $M_\psi(U^p)_\km \in \kC_{\GL_2(F_v),\psi}(\cO)^{\kB_v}$, and hence  $M_\psi(U^p)_\km \in \kC_{D_p^\times,\psi}(\cO)^{\kB_\km}$.
     	
     	\item  Both actions $\tau_{\Gal},\tau_{\mathrm{Aut}}$ of $R_v^\ps $ on $\mm_v$ are the same. Hence, $\tau_{\Gal} = \tau_{\mathrm{Aut}}$.
     \end{enumerate}
     \end{thm}

       \begin{proof}
       	These assertions are proved in \cite[Theorem 3.5.3]{pan2022fontaine} and \cite[Theorem 3.3.3]{Zhang:2024aa}. 
       \end{proof}

       The following results are direct consequences of the previous theorem.
       
       \begin{prop}\label{finiteness from lgc}
       	Keep notations as above.
       	\begin{enumerate}
       	\item  $\mm$ is a faithful, finitely generated $\T_\psi(U^p)_\km$-module. 
       	\item  The natural map $ R_{F,p}^\ps \to \T_{\psi}(U^p)_\km$ is finite.
       \end{enumerate}
       	
       \end{prop}
  
  \begin{proof}
  	See \cite[Corollary 3.5.8]{pan2022fontaine} and \cite[Corollary 3.3.7]{Zhang:2024aa}.
  \end{proof}
  
    \subsection{A generalization of Pašk\=unas' result}\label{sec pas}
    In this subsection, we prove a generalization result of \cite[Corollary 5.13]{Paskunas_2021}. The proof of Proposition \ref{paskunas} is based on an argument communicated to the author by Vytautas Paškūnas, whom the author thanks for explaining this argument.
    
    Fix a place $v|p$ of $F$. Let $\bar{T}: G_{F_v} \to \mathbb{F}$ be a two-dimensional local pseudo-representation with determinant $\bar{\chi}$. Let \(\chi:G_{F_v}\to\mathcal O^\times\) be a crystalline character
    lifting \(\bar\chi\), and let
    \(\zeta:F_v^\times\to\mathcal O^\times\) be the character corresponding
    to \(\chi\varepsilon^{-1}\) via local class field theory. Let $R_v^{\ps, \chi}$ be the local pseudo-deformation ring of $\bar{T}$ with determinant $\chi$.
    
    Note that the pseudo-representation $\bar{T}$ defines a block $\mathfrak{B}_v$ of $\GL_2(F_v)$. To ease our notations, we put $\kC(\cO)^{\mathfrak{B}_v}:=\kC_{\GL_2(F_v),\zeta}(\cO)^{\mathfrak{B}_v} $ (as defined in Section \ref{pas sec}).

    \begin{prop}\label{paskunas}
    	Let $M$ be an $\cO$-torsion free object of $\kC(\cO)^{\mathfrak{B}_v}$. Assume that we are given a continuous action of $R_v^{\ps, \chi}$ on $M$, which commutes with the action of $\GL_2(F_v)$, such that the following hold:
    	\begin{enumerate}
    		\item the action of $R_v^{\ps, \chi}$ on $M$ is the same as the one induced by the isomorphism $R_v^{\ps, \chi} \cong Z_{\kB_v}$ (see \cite[Theorem A.2.2]{Zhang:2024aa}), where $ Z_{\kB_v}$ is the centre of $E_{\mathfrak{B}_v}$.
    		\item $M$ is a finitely generated $\cO[[\GL_2(\cO_{F_v})]]$-module.
    	\end{enumerate}
    	Then the action of $R_v^{\ps, \chi}$ on $M$ is not faithful.
    \end{prop}
    
    \begin{proof}
    	Suppose, for contradiction, that the action of \(R_v^{\mathrm{ps},\chi}\) on \(M\) is faithful.
    	
    	By Proposition \ref{density of crys}, we may choose fixed integers
    	\(
    	a>b
    	\)
    	such that there are infinitely many points in \( \Spec R_v^{\ps, \chi}[\frac{1}{p}]\) whose associated
    	representations are irreducible, crystalline of Hodge--Tate weights \(\{a,b\}\).
    	Let
		\(
		x:R_v^{\ps, \chi}\longrightarrow \overline{\mathbb Q}_p
		\)
		be such a closed point with kernel $\km_x$ and residue field $k(x)$, and let $T_x$ be the corresponding irreducible local pseudo-representation.
    	
    	Set
    	\(
    	\Pi(M):=\operatorname{Hom}_{\mathcal O}^{\mathrm{cont}}(M,K)
    	\). Since \(M\) is
    	\(\mathcal O\)-torsion free, \(\Pi(M)\) is a unitary \(K\)-Banach space
    	representation of \(\GL_2(F_v)\).
    	
    	By \cite[Theorem 7.1]{Pa_k_nas_2021}, applied at the point \(x\), 
    	we have
    	\(
		\Pi(M)[\mathfrak m_x]\simeq \Pi_x^{\oplus n_x}
		\)
		for some integer \(n_x>0\), where \(\Pi_x\) is an absolutely irreducible non-ordinary
		Banach representation satisfying
    	\(
    	\operatorname{tr}\check{\mathbf V}(\Pi_x)=T_x
    	\), where $\check{\mathbf V} $ is Colmez's Montreal functor.

		Let $\rho_x$ be the absolutely irreducible representation with trace $T_x$. By the Brauer--Nesbitt theorem and \cite[Theorem 1.1]{Colmez_2014}, we have $\check{\mathbf V}(\Pi_x)\simeq\rho_x$. Put
		\(
		\Theta_{a,b}:=\operatorname{Sym}^{a-b-1}(\cO^2)\otimes \det^{-a}
		\). Since $F_v\simeq\Q_p$, \cite[Theorems 0.20 and 0.21]{Colmez_2010} gives $\Pi_x^{\operatorname{alg}}\simeq \pi_{\operatorname{sm},x}\otimes_{k(x)}(\Theta_{a,b}\otimes_{\cO}k(x))$, after taking account of our convention that the Hodge--Tate weight of $\varepsilon$ is $-1$, where $\pi_{\operatorname{sm},x}$ is attached to $\operatorname{WD}(\rho_x)$ by Colmez's modified local Langlands correspondence. Since $\rho_x$ is crystalline, $\operatorname{WD}(\rho_x)$ has trivial inertia and monodromy operator $N=0$. Hence $\pi_{\operatorname{sm},x}^{\GL_2(\cO_{F_v})}\neq 0$, also when this modified smooth representation is reducible. It follows that
		\(
		\operatorname{Hom}_{\GL_2(\cO_{F_v})}(\Theta_{a,b},\Pi_x)\neq 0
		\), and further
    	\(
    	\operatorname{Hom}_{\GL_2(\cO_{F_v})}(\Theta_{a,b},\Pi(M)[\mathfrak m_x])\neq 0
    	\).
    	
    	Define
    	\(
    	M(\Theta_{a,b})
    	:=
    	\operatorname{Hom}_{\cO[[\GL_2(\cO_{F_v})]]}
    	\bigl(M,(\Theta_{a,b})^{d}\bigr)^{d},
    	\)
    	where
    	\(
    	(-)^d
    	\)
    	denotes the Schikhof dual. As \(M\) is finitely generated over
    	\(\mathcal O[[\GL_2(\cO_{F_v})]]\), the \(\mathcal O\)-module \(M(\Theta_{a,b})\) is finitely generated. In particular,  \(M(\Theta_{a,b})[1/\varpi]\) is a
    	finite-dimensional \(K\)-vector space.
    	
    	Note that the non-vanishing
    	\(
    	\operatorname{Hom}(\Theta_{a,b},\Pi(M)[\mathfrak m_x])\neq 0
    	\)
    	implies that the point \(x\) lies in the support of the
    	\(R_v^{\ps, \chi}[1/\varpi]\)-module
    	\(
    	M(\Theta_{a,b})[1/\varpi]
    	\). Varying $x$, we know that there are infinitely many distinct absolutely irreducible regular crystalline
    	points \(x\) with fixed Hodge--Tate weights \(\{a,b\}\) lying in the support of
    	\(M(\Theta_{a,b})[1/\varpi]\). However, the action of \(R_v^{\ps, \chi}[1/\varpi]\) on this finite-dimensional vector space factors
    	through a finite-dimensional \(K\)-algebra. Consequently,
    	\(
    	\operatorname{Supp}_{R_v^{\ps, \chi}[1/\varpi]}\bigl(M(\Theta_{a,b})[1/\varpi]\bigr)
    	\)
    	is a finite set of closed points. We get a contradiction.
   
    \end{proof}
    
    \begin{remark}\label{3.4.2}
    	The preceding result can be viewed as a
    	generalization of \cite[Corollary 5.13]{Paskunas_2021}. The proof
    	given there uses the explicit description of the endomorphism ring
    	\(E_{\mathfrak B_v}=\operatorname{End}_{\mathfrak C(\mathcal O)}(P_{\mathfrak B_v})\),
    	which was obtained in \cite{Pa_k_nas_2013}. This explicit description is unknown in the \textit{exceptional} case, namely when
     \(\bar T=\eta+\eta\omega
    	\)
    	for some character \(\eta\) if $p=3$. Consequently, the exceptional case is excluded
    	in the statement of \cite{Paskunas_2021}.
    	
    	The argument above is different. It uses \cite[Theorem 7.1]{Pa_k_nas_2021} rather than  the explicit structure of
    	\(E_{\mathfrak B_v}\). Thus the proof is insensitive to
    	the exceptional block. In particular, it applies also when
    	\(\bar T=\eta+\eta\omega\) if \(p=3\).
    	
    	For $p>2$, the conclusion of Proposition \ref{paskunas} also holds when
    	$\chi$ is merely continuous. To see this, choose $a\in\mathbf Z$
    	such that $\nu:=\bar\chi\omega^{-a}$ is unramified and put
    	$\chi_0=\widetilde\nu\varepsilon^a$, a crystalline lift of
    	$\bar\chi$. Let $\eta\equiv1\pmod\varpi$ be the unique
    	continuous character satisfying $\eta^2=\chi_0\chi^{-1}$.
    	Via local class field theory, form the compact representation
    	\(
    	M^\eta=M\otimes_{\mathcal O}(\eta^{-1}\circ\det).
    	\)
    	Its smooth-side central character is
    	$\zeta\eta^2=\chi_0\varepsilon^{-1}$.
    	Twisting gives an isomorphism
    	\(
    	R_v^{\mathrm{ps},\chi_0}\xrightarrow{\sim}
    	R_v^{\mathrm{ps},\chi} \) given by 
    	\(T_0(g)\longmapsto\eta(g)T(g),
    	\)
    	compatible with the canonical block-centre actions.
    	The module $M^\eta$ is still $\mathcal O$-torsion free and
    	finitely generated over $\mathcal O[[\mathrm{GL}_2(\mathcal O_{F_v})]]$:
    	the latter assertion follows by transporting the action through
    	the continuous group-algebra automorphism
    	$[g]\mapsto\eta(\det g)^{-1}[g]$.
    	Faithfulness is preserved by the displayed isomorphism of
    	pseudo-deformation rings. Thus a faithful action on $M$ would
    	give a faithful action on $M^\eta$, contrary to the crystalline
    	case of Proposition \ref{paskunas}. No de Rham condition on $\eta$ is
    	needed for this argument.
    \end{remark}
    
   A direct consequence of Proposition \ref{paskunas} is the equidimensionality of the big Hecke algebra when $F=\Q$. More precisely, let $D$ be a quaternion algebra with centre $\Q$ which is ramified at the unique infinite place and unramified at $p$.  
   Let $\psi: (\mathbb{A}_\Q^\infty)^\times/\Q_+^\times \to \cO^\times$ be a continuous character. Assume that the local character \(\psi|_{\mathbb Q_p^\times}\) is de Rham. Then we can define the big Hecke algebra $\T_\psi(U^p)$ for some tame level $U^p$ (see Section \ref{sec for T}).
   
   \begin{thm}\label{dim of T}
   	Each irreducible component of the big Hecke algebra $\T_\psi(U^p)$ is of Krull dimension $3$ and of characteristic $0$.
   \end{thm}
   
   \begin{proof}
   	By Proposition \ref{pseudo of T}, it suffices to consider a non-zero
   	localization $T:=\mathbb T_\psi(U^p)_{\mathfrak m}$.
   	Put $M=M_\psi(U^p)_{\mathfrak m}$,
   	\(
   	R^{\mathrm{Gal}}=R_p^{\mathrm{ps},\psi\varepsilon^{-1}} \)
   	and 
   	\( R^{\mathrm{can}}=R_p^{\mathrm{ps},\psi\varepsilon},
   	\)
   	where the latter ring lifts
   	$\operatorname{tr}(\bar\rho_{\mathfrak m,p}\otimes\omega)$.
   	
   	The module $M$ is $\mathcal O$-torsion free and finitely generated
   	over $\mathcal O[[\mathrm{GL}_2(\mathbf Z_p)]]$, by admissibility
   	and the duality for completed cohomology. By Theorem \ref{lgc}, the
   	$R^{\mathrm{Gal}}$-action corresponds, under the twisting
   	identification of Section \ref{sec lgc}, to the canonical
   	$R^{\mathrm{can}}$-action. The continuous-character version of
   	Proposition \ref{paskunas} explained in Remark \ref{3.4.2} therefore gives
   	\(
   	\operatorname{Ann}_{R^{\mathrm{Gal}}}(M)\ne0.
   	\)
   	Since $T$ acts faithfully on $M$, we have
   	\(
   	\ker(R^{\mathrm{Gal}}\to T)
   	=\operatorname{Ann}_{R^{\mathrm{Gal}}}(M)\ne0.
   	\)
   	By Propositions \ref{local pseudo} and Proposition \ref{finiteness from lgc}, $R^{\mathrm{Gal}}$ is a
   	four-dimensional domain and $T$ is finite over it. Consequently
   	$\dim T\le3$. On the other hand,
   	\cite[Theorem 3.6.1]{pan2022fontaine} and
   	\cite[Corollary 3.3.10]{Zhang:2024aa} show that every irreducible
   	component of $T$ has characteristic zero and dimension at least
   	$3$. This proves the assertion.

   \end{proof}
  
    \subsection{The setting of modular curves}\label{sec for modular curve}
    In this subsection, we prove results similar to those in Section \ref{sec lgc} in the setting of modular curves. Our main reference is \cite{Pan_2022}.
    
    Recall that \(p\) is an odd prime, and $K$ is $p$-adic local field with
    ring of integers \(\mathcal O\), uniformizer \(\varpi\), and residue field
    \(\mathbb F\). Write
    \(
    G:=\mathrm{GL}_2(\mathbb Q_p)
    \). Let \(K^p= \prod_{l \ne p} K_l \subset \mathrm{GL}_2(\mathbb A_f^p)\) be a tame level. For an open compact
    subgroup \(K_p\subset G\), we write
    \[
    Y_{K^pK_p}(\mathbb C)
    :=
    \mathrm{GL}_2(\mathbb Q)
    \backslash
    \bigl(\mathfrak H^\pm\times \mathrm{GL}_2(\mathbb A_f)/(K^pK_p)\bigr)
    \]
    for the corresponding modular curve, where $\mathfrak H^\pm $ is the union of the usual upper and lower half-planes.
    
    For \(n\geq 1\), define
    \[
    \widetilde H^i(K^p,\mathbb Z/p^n)
    :=
    \varinjlim_{K_p\subset G}
    H^i\bigl(Y_{K^pK_p}(\mathbb C),\mathbb Z/p^n\bigr),
    \]
    where the direct limit is taken over open compact subgroups \(K_p\subset G\). The
    completed cohomology of tame level \(K^p\) is defined as
    \[
    \widetilde H^i(K^p,\mathbb Z_p)
    :=
    \varprojlim_n \widetilde H^i(K^p,\mathbb Z/p^n).
    \]

    Let \(\Sigma\) be a finite set of primes containing \(p\) and all primes $l$ for which $K_l$ is not maximal.  We write $S_l$ and $T_l$ for the Hecke operators corresponding to the double cosets of $K_l$ in $\GL_2(\Q_l)$ represented by $\begin{pmatrix}
    		l & 0\\
    		0& l
    	\end{pmatrix} $ and $\begin{pmatrix}
    		l & 0\\
    		0& 1
    	\end{pmatrix} $, respectively. For each compact open subgroup
    \(K_p\subset G\), let
    \[
    \mathbb T(K^pK_p)
    \subset
    \operatorname{End}_{\mathbb Z_p}
    H^1\bigl(Y_{K^pK_p}(\mathbb C),\mathbb Z_p\bigr)
    \]
    be the $\mathbb{Z}_p$-subalgebra generated by Hecke operators $T_l$ and $S_l^{\pm 1}$ for any $l \notin \Sigma$. Define the big Hecke algebra
    of tame level \(K^p\) by
    \[
    \mathbb T(K^p):=\varprojlim_{K_p\subset G}\mathbb T(K^pK_p).
    \]

    Let
    \(
    \mathfrak m\subset \mathbb T(K^p)\otimes_{\mathbb Z_p}\mathcal O
    \)
    be a maximal ideal. Enlarging \(\mathcal O\) if necessary, we assume that the
    residue field of \(\mathfrak m\) is \(\mathbb F\). Then it gives a two-dimensional semi-simple representation $\bar{\rho}_{\km, p} : G_{\Q_p} \to \GL_2(\F)$. Let
    \(
    \zeta':\mathbb A_f^\times/\operatorname{det}(K^p)\mathbb Q_{+}^\times\longrightarrow \mathcal O^\times
    \)
    be a continuous character such that 
    \(
    \zeta:=\zeta'|_{\mathbb Q_p^\times}
    \) is congruent to $\operatorname{det}(\bar{\rho}_{\km, p})\omega^{-1}$ modulo $\varpi$ via local class field theory.
    For \(A\in\{K,\mathcal O,K/\mathcal O\}\), define
    \(
    \widetilde H^1(K^p,A)_{\zeta',\mathfrak m}
    \)
    to be the \(\mathfrak m\)-localization of \(\widetilde H^1(K^p,A)\) on which the
    centre \(\mathbb A_f^\times\) acts via \(\zeta'\). Let
    \(
    \mathbb T
    \)
    denote the image of \((\mathbb T(K^p)\otimes_{\mathbb Z_p}\mathcal O)_{\mathfrak m}\)
    inside
    \(
    \operatorname{End}_{\mathcal O}
    \bigl(\widetilde H^1(K^p,K/\mathcal O)_{\zeta',\mathfrak m}\bigr).
    \)
    By construction, \(\mathbb T\) acts faithfully on
    \(\widetilde H^1(K^p,K/\mathcal O)_{\zeta',\mathfrak m}\).
    
    Write $\chi= \zeta \varepsilon$. Let
    \(
    R_{p}^{\mathrm{ps},\chi}
    \)
    be the universal local pseudo-deformation ring of
    \(\operatorname{tr}\bar\rho_{\mathfrak m,p}\) with determinant \(\chi\). Then by the universal property, we get a natural map $R_{p}^{\mathrm{ps},\chi} \to \mathbb{T} $, and it gives an action $\tau^p$ of $R_{p}^{\mathrm{ps},\chi}$ on \(\widetilde H^1(K^p,K/\mathcal O)_{\zeta',\mathfrak m}\).
    
    On the other hand, the residual representation \(\bar\rho_{\mathfrak m,p}\) determines
    a block
    \(
    \mathfrak B_{\mathfrak m}
    \)
    of the category
    \(
    \operatorname{Mod}^{\mathrm{l.adm}}_{G,\zeta}(\mathcal O)
    \). By \cite[Theorem 1.2]{Pa_k_nas_2021}, any object in the block \(\mathfrak B_{\mathfrak m}\) carries another natural action of \(R_p^{\mathrm{ps},\chi}\).
    
    The following result is an analogue of Theorem \ref{lgc}.
    
    \begin{thm}
    	\label{thm:pan-local-global-modular-curves}
    	With the notation above, one has
    	\(
    	\widetilde H^1(K^p,K/\mathcal O)_{\zeta',\mathfrak m}
    	\in
    	\operatorname{Mod}^{\mathrm{l.adm}}_{G,\zeta}(\mathcal O)^{\mathfrak B_{\mathfrak m}}
    	\), and it gives an action $\tau_p$ of \(R_p^{\mathrm{ps},\chi}\) on \(\widetilde H^1(K^p,K/\mathcal O)_{\zeta',\mathfrak m}\). Moreover, we have $\tau_p=\tau^p$.
    \end{thm}
    
    \begin{proof}
    	See \cite[Theorem 6.3.5]{Pan_2022}.
    \end{proof}
    
    \begin{prop}\label{nonzero modular kernel}
    	Assume that $\zeta$ corresponds to a crystalline character via local
    	class field theory, that $K^p\GL_2(\Z_p)$ is sufficiently small, and that
    	$\km$ is non-Eisenstein. Then we have
    	\(
    	\operatorname{ker}\{R_p^{\mathrm{ps},\chi}\to\mathbb T\}
    	\ne 0.
    	\)
    \end{prop}
    
    \begin{proof}
    	Put $H=\widetilde H^1(K^p,K/\cO)_{\zeta',\mathfrak m}$. The
    	non-Eisenstein assumption gives an isomorphism
    	\(
    	M:=\widetilde H^1(K^p,\cO)_{\zeta',\mathfrak m}^{d}
    	\cong H^\vee
    	\)
    	(see \cite[Lemma 5.3.8]{emerton11} and
    	\cite[(37), page 1097]{Paskunas_2021}). In particular, $M$ is
    	$\cO$-torsion free. Theorem \ref{thm:pan-local-global-modular-curves}
    	places $M$ in $\kC_{\GL_2(\Q_p),\zeta}(\cO)^{\mathfrak B_\km}$
    	and identifies the action of $R_p^{\mathrm{ps},\chi}$ through
    	$\mathbb T$ with the action induced by the centre of the block.
    	Moreover, the admissibility of $H$ implies that $M$ is finitely
    	generated over $\cO[[\GL_2(\Z_p)]]$.
    	
    	Thus Proposition \ref{paskunas} applies to $M$ and gives
    	$\operatorname{Ann}_{R_p^{\mathrm{ps},\chi}}(M)\ne0$. By definition,
    	$\mathbb T$ acts faithfully on $H$, and hence on $M$. Consequently,
    	\(
    	\operatorname{ker}\{R_p^{\mathrm{ps},\chi}\to\mathbb T\}
    	=\operatorname{Ann}_{R_p^{\mathrm{ps},\chi}}(M)\ne0.
    	\)
    \end{proof}
    
    \section{Pro-modularity}\label{sec 4}
    In this section, we recall some modularity or automorphy lifting theorems.
    
    Throughout this section, we fix an odd prime $p$, and $F$ an abelian totally real field in which $p$ is unramified. Let $\Sigma$ be a finite set of finite places containing $\Sigma_p$, the set of finite places $v$ of $F$ lying above $p$. 
    
    We denote by $K$  a $p$-adic local field with its ring of integers $\cO$, a uniformizer $\varpi$ and its residue field $\F$. Let $\bar{\chi}: G_{F, \Sigma} \to \F^\times$ be a continuous, totally odd character, and suppose that $\bar{\chi} $ can be extended to a character of $G_{\Q}$.
    
    \subsection{Fontaine-Mazur conjecture in the residually reducible case}\label{sec deformation ring}
    In this subsection, we recall some progress on the Fontaine-Mazur conjecture in the residually reducible case.
    
    Let $\chi: G_{F, \Sigma} \to \cO^\times$ be a continuous, totally odd character lifting $ \bar{\chi}$ such that $\chi$ is de Rham for any $v|p$. Let $\Lambda_F:=\widehat{\bigotimes}_{v|p}\cO[[\cO_{F_v}^\times(p)]] $ be the Iwasawa algebra for $F$, where $ \cO_{F_v}^\times(p)$ is the $p$-adic completion of $ \cO_{F_v}^\times$.
    
    \textbf{Case 1}: ordinary Fontaine-Mazur conjecture when $ \bar{\chi}|_{G_{F_v}} \ne \mathbf{1}$ for any $v|p$. 
    
    Let $R^{\ps,\ord}$ be the universal ordinary pseudo-deformation ring of the pseudo-representation $1+\bar{\chi}$ for $G_{F, \Sigma}$ with determinant $\chi$. Here we require only that each local restriction be a sum of two characters, without specifying which character occurs on a stable line. Write $T^{\operatorname{univ}}: G_{F, \Sigma} \to R^{\ps,\ord}$ for the universal pseudo-deformation, and for each $v|p$, write $T^{\operatorname{univ}}|_{G_{F_v}}=\psi^{\operatorname{univ}}_{v,1}+\psi^{\operatorname{univ}}_{v,2}$, where $\psi^{\operatorname{univ}}_{v,1},\psi^{\operatorname{univ}}_{v,2}:G_{F_v}\to (R^{\ps,\ord})^\times$ are liftings of $\mathbf{1},\bar{\chi}|_{G_{F_v}}$, respectively.
    For each $v|p$, the character $ \psi^{\operatorname{univ}}_{v,1}$ induces a homomorphism $\cO[[\cO_{F_v}^\times(p)]]\to R^{\ps,\ord} $ via the local class field theory, and hence we get a natural map $\Lambda_F \to R^{\ps,\ord} $ by taking the completed tensor product.
    
    \begin{thm}[ordinary case 1] \label{ord1} Keep notations and assumptions as above.
    		\begin{enumerate}
    			\item $R^{\ps,\ord}$ is a finite $\Lambda_F$-algebra.
    			\item Let $\kp$ be any maximal ideal of $R^{\ps,\ord}[\frac{1}{p}]$. Let $\rho(\kp)$ be the Galois representation $G_{F,\Sigma}\to\GL_2(k(\kp))$ corresponding to $\kp$ (in the sense of Definition \ref{construction}). Assume 
    			\begin{itemize}
    				\item $\rho(\kp)$ is absolutely irreducible.
    				\item For any $v|p$, $\rho(\kp)|_{G_{F_v}}\cong\begin{pmatrix}\psi_{v,1} & *\\ 0 & \psi_{v,2}\end{pmatrix}$ such that $\psi_{v,1}$ is de Rham and has strictly less Hodge-Tate number than $\psi_{v,2}$ for any embedding $F_v\hookrightarrow \overbar{\Q_p}$.
    			\end{itemize}
    			Then $\rho(\kp)$ comes from a twist of a Hilbert modular form.
    		\end{enumerate}
    \end{thm}
    
    \begin{proof}
    	This is \cite[Theorem 5.1.1]{pan2022fontaine}.
    \end{proof}
    
    \textbf{Case 2}: ordinary Fontaine-Mazur conjecture when $ \bar{\chi}|_{G_{F_v}} = \mathbf{1}$ for any $v|p$.
    
    Let $R^{\ps,\ord}_1$ be the universal ordinary pseudo-deformation ring which pro-represents the functor from $\cOf$ to the category of sets sending $R$ to the set of tuples $(T;\{\psi_{v,1}\}_{v|p})$, where
    \begin{itemize}
    	\item $T:G_{F,\Sigma}\to R$ is a pseudo-representation lifting $1+\bar\chi$ with determinant $\chi$.
    	\item For any $v|p$, $\psi_{v,1}:G_{F_v}\to R^\times$ is a lifting of the trivial character, and we have
    	$T|_{G_{F_v}}=\psi_{v,1}+\psi_{v,2}$ (hence $\psi_{v,2}=\chi|_{G_{F_v}}(\psi_{v,1})^{-1}$) such that $T$ is $\psi_{v,1}$-ordinary in the sense of \cite[Definition 5.3.1]{pan2022fontaine}.
    \end{itemize}
    
    Write $(T_1^{\operatorname{univ}}, \{\psi^{\operatorname{univ}, \mathbf{1}}_{v,1}\}_{v|p}) $ for the universal object. The character $ \psi^{\operatorname{univ}, \mathbf{1}}_{v,1}$ induces a homomorphism $\cO[[\cO_{F_v}^\times(p)]]\to R^{\ps,\ord}_1 $ via the local class field theory, and further we get a natural map $\Lambda_F \to R^{\ps,\ord}_1 $ by taking the completed tensor product.

    In Theorem \ref{ord2}, $\psi_{v,1}$ always denotes the specified ordinary character.
    
     \begin{thm}[ordinary case 2] \label{ord2} Keep notations and assumptions as above.
    	\begin{enumerate}
    		\item $R^{\ps,\ord}_1$ is a finite $\Lambda_F$-algebra.
    		\item Let $\kp$ be any maximal ideal of $R^{\ps,\ord}_1[\frac{1}{p}]$. Let $\rho(\kp)$ be the Galois representation $G_{F,\Sigma}\to\GL_2(k(\kp))$ corresponding to $\kp$ (in the sense of Definition \ref{construction}). Assume 
    		\begin{itemize}
    			\item $\rho(\kp)$ is absolutely irreducible.
    			\item For any $v|p$, $\rho(\kp)|_{G_{F_v}}\cong\begin{pmatrix}\psi_{v,1} & *\\ 0 & \psi_{v,2}\end{pmatrix}$ such that $\psi_{v,1}$ is de Rham and has strictly less Hodge-Tate number than $\psi_{v,2}$ for any embedding $F_v\hookrightarrow \overbar{\Q_p}$.
    		\end{itemize}
    		Then $\rho(\kp)$ comes from a twist of a Hilbert modular form.
    	\end{enumerate}
    \end{thm}
    
    \begin{proof}
    	This is \cite[Theorem 6.1.1]{pan2022fontaine}.
    \end{proof}
    
    \textbf{Case 3}: non-ordinary Fontaine-Mazur conjecture.
    
    In this case, we suppose that $p$ splits completely in $F$. 
    
    \begin{thm} [non-ordinary case] \label{non-ord} Keep notations and assumptions as above.
     Let
    	$\rho:G_{F, \Sigma} \to \GL_2(\cO)$
    	be a continuous irreducible representation with the following properties
    	\begin{itemize}
    		\item Let $\bar{\rho}$ be the reduction of $\rho$ modulo $\varpi$. We assume its semi-simplification has the form $\bar{\chi}_1\oplus\bar{\chi}_2$ with $\bar{\chi}_1\bar{\chi}_2^{-1}=\bar{\chi}$.
    		\item $\rho|_{G_{F_v}}$ is irreducible and de Rham of distinct Hodge-Tate weights for any $v|p$. 
    	\end{itemize}
    	Then $\rho$ comes from a twist of a Hilbert modular form.
    \end{thm}
    
    \begin{proof}
    	See \cite[Theorem 7.1.1]{pan2022fontaine} and \cite[Theorem 5.1.1]{Zhang:2024aa}.
    \end{proof}

    We record a dimension comparison for later use.

    \begin{lem}\label{lem:ordinary-character-quotients}
    Suppose that $\bar\chi|_{G_{F_v}}=\mathbf{1}$ for every $v|p$. Let $R=R_F^{\ps,\chi}$, let $P$ be the completed tensor product of the local pseudo-deformation rings (lifting the pseudo-representation $\mathbf{1}+\mathbf{1}$) with determinant $\chi|_{G_{F_v}}$, and let $S=\widehat\otimes_{v|p}S_v$, where $S_v$ parametrizes characters of $G_{F_v}$ lifting $\mathbf{1}$. The map $P\to S$ is induced by $\psi_{v,1}\mapsto\psi_{v,1}+\chi\psi_{v,1}^{-1}$. For a proper ideal $I\subset R$, put $A_I=(R/I)\otimes_P S$ and $B_I=R_1^{\ps,\ord}/IR_1^{\ps,\ord}$. Then every prime quotient of $A_I$ is an $R$-algebra quotient of $B_I$, and $\dim A_I=\dim B_I$.
    \end{lem}

    \begin{proof}
    The map $P\to S$ is finite, as in the proof of \cite[Lemma 7.3.1]{pan2022fontaine}; thus $A_I$ is a complete Noetherian local ring. Imposing the $\psi_{v,1}$-ordinary identities gives a natural surjection $h:A_I\twoheadrightarrow B_I$.

    Let $\pi:A_I\twoheadrightarrow D=A_I/\mathfrak q$ be a prime quotient. Over an algebraic closure of $\operatorname{Frac}(D)$, realize its pseudo-representation by a semisimple representation $\rho$. Each local restriction is reducible. Its stable-line character is one of the $D$-valued characters $\psi_{v,1}$ and $\chi\psi_{v,1}^{-1}$, both lifting $\mathbf{1}$. Exchanging them where necessary defines an $R$-algebra involution $w$ of $A_I$. The matrix calculation in \cite[Lemma 5.3.2]{pan2022fontaine} gives the ordinary identities over $\operatorname{Frac}(D)$, hence over $D$. Thus the universal property gives $f:B_I\to D$ with $fh=\pi w$. Since $w^2=1$, we have $\pi=fhw$, so $f$ is surjective.

    The involution $w$ may depend on $\mathfrak q$; factorization through the fixed map $h$ is not asserted. Now $h$ gives $\dim B_I\le\dim A_I$, and a prime quotient of $A_I$ of maximal dimension gives the reverse inequality. This argument also applies in characteristic $p$.
    \end{proof}
    
    \subsection{Potential pro-modularity}\label{potential pro-modular}
    In this subsection, we recall the potential pro-modularity argument in \cite{Zhang:2024aa} and \cite{Zhang:2024ab} and prove a new case.
    
    Let $F$ be an abelian totally real field in which $p$ splits completely. Let $\chi : G_{F, \Sigma} \to \cO^\times$ be a totally odd, de Rham character. Assume that $\chi$ can be extended to a character of $G_\Q$.
    
    By Grunwald-Wang's theorem (see \cite[Theorem 5, page 80]{artin}), we can find a finite extension $F^1/F$ satisfying the following assumptions:
    	\begin{itemize}
    	\item $F^1$ is a totally real abelian extension of $\mathbb{Q}$ of even degree.
    	\item $p$ splits completely in $F^1$.
    	\item  For any $v \in \Sigma^1 \setminus \Sigma^1_p$, we have $p| \operatorname{Nm}(v)-1$.
    	\item $[F^1:\mathbb{Q}]>7|\Sigma^1 \setminus \Sigma^1_p|+6$.
    	\item  $\chi$ is unramified outside $ \Sigma^1_p$, and $\chi(\operatorname{Frob}_v) \equiv 1~\operatorname{mod}~\varpi$ for any $v \in \Sigma^1 \setminus \Sigma^1_p$,
    \end{itemize}
    where $ \Sigma^1$ (resp. $\Sigma^1_p $) is the finite set of finite places of $F^1$ lying above places in $\Sigma$ (resp. $\Sigma_p $).
    
     Let $D$ be a quaternion algebra over $F^1$ ramified exactly at all infinite places. We fix an isomorphism between $(D \otimes_{F^1}\A_{F^1}^\infty )^\times$ and $\GL_2(\A_{F^1}^\infty)$. We write $\psi:=\chi\varepsilon$ and view it as a character of $(\A_{F^1}^\infty )^\times/(F^1)^\times_{+}$ via global class field theory. Define a tame level $U^p_0=\prod_{v\nmid p}U_{0,v}$ as follows: $U_{0,v}=\GL_2(\cO_{F^1_v})$ if $v\notin \Sigma^1$ and 
    \[U_{0,v}=\mathrm{Iw}_v:=\{g\in\GL_2(\cO_{F^1_v}),g\equiv \begin{pmatrix}*&*\\0&*\end{pmatrix}\mod \varpi_v\}\]
    otherwise. For $v\in\Sigma^1\setminus\Sigma^1_p$, let $\Delta_v$ be the maximal $p$-power quotient of $k(v)^\times$ and let $U_{1,v}$ be the kernel of the map $U_{0,v}\to\Delta_v$ induced by $\begin{pmatrix}a&b\\c&d\end{pmatrix}\mapsto\frac{a}{d}\mod\varpi_v$; for $v\notin\Sigma^1$, put $U_{1,v}=U_{0,v}$. Write $U^p_1=\prod_{v\nmid p}U_{1,v}$ and define $\T:=\T_\psi(U^p_1)$. By the proof of Corollary \ref{lower}, we know that $1+\bar{\chi}|_{G_{F^1, \Sigma^1}}$ defines a maximal ideal $\km$ of $\T $.
    
    For any tuple of characters $\xi_v:k(v)^\times\to\cO^\times$ of $p$-power order, $v\in\Sigma^1\setminus\Sigma^1_p$, let $\xi$ be the resulting character of $U^p_0$. Since $U^p_1\subseteq\ker(\xi)$, the natural inclusion of $\xi$-twisted forms of level $U^p_0$ into untwisted forms of level $U^p_1$ induces a surjection $\T_\km\twoheadrightarrow\T_{\psi,\xi}(U^p_0)_{\km_\xi}$, where $\km_\xi$ denotes the corresponding residual maximal ideal.
    
    Let $R_{F^1}^{\ps, \chi}$ be the universal pseudo-deformation ring of the pseudo-representation $1+\bar{\chi}|_{G_{F^1, \Sigma^1}}$ with determinant $\chi|_{G_{F^1, \Sigma^1}}$. We say that a prime of $R_{F^1}^{\ps, \chi}$ is \textit{pro-modular}, if it comes from $\Spec \T_\km$. Note that by Corollary \ref{lower}, we have $\dim R_{F^1}^{\ps, \chi} \ge 2[F^1:\Q]+1 $. Recall that $R_{F^1,0}^{\ps, \chi}= R_{F^1}^{\ps, \chi}/I_{F^1}$, where $ I_{F^1}$ is the intersection of all minimal primes of $R_{F^1}^{\ps, \chi} $ of dimension at least $2[F^1:\Q]+1 $. We say that a prime of $R_{F^1,0}^{\ps, \chi}$ is \textit{pro-modular}, if it comes from some pro-modular prime of $R_{F^1}^{\ps, \chi}$.
    
    \begin{prop} \label{potential R=T}
    	For each irreducible component $C$ of $R_{F^1}^{\ps, \chi} $ of dimension at least $1+2[F^1:\Q] $, its generic point is pro-modular. Equivalently, every prime of $R_{F^1,0}^{\ps, \chi}$ is pro-modular.
    \end{prop}
    
    \begin{proof}
    	For the case $\bar{\chi}|_{G_{F^1_v}} \ne \mathbf{1}$ for any $v|p$,  it is proved in \cite[Theorem 5.5.1]{Zhang:2024aa}.
    	
    	Now we assume $\bar{\chi}|_{G_{F^1_v}} = \mathbf{1}$ for any $v|p$. We sketch the main steps here.
    	
    	We first claim that if $\mathfrak{r}$ is an irreducible one-dimensional pro-modular prime  of $R_{F^1}^{\ps, \chi} $, then every irreducible component of $R_{F^1}^{\ps, \chi} $ containing $\mathfrak{r}$ is pro-modular.
    	
    	The proof is similar to the proof of \cite[Proposition 4.5.1]{Zhang:2024ab}. Let $Z$ be the set of all irreducible components of $\Spec R_{F^1}^{\ps, \chi}$ containing $\mathfrak{r}$. Write $Z= Z_1 \amalg Z_2$, where $Z_1$ is the subset of $Z$ consisting of all pro-modular irreducible components. We show that if $Z_1$ is not empty, then $Z_2$ is empty.
    	
    	If not, by 1) of Proposition \ref{irreducible point}, there exist an irreducible component $Y \in Z_2$ and a pro-modular prime $\mathfrak{r}' \in Y$ of dimension at least $2[F^1: \mathbb{Q}]-1$. Then arguing as in the proof of \cite[Proposition 4.5.1]{Zhang:2024ab}, we may find a prime $\mathfrak{r}'\subset \mathfrak{r}''$ satisfying the following conditions:
    	\begin{itemize}
    		\item $\varpi \in \mathfrak{r}''$.
    		\item  The dimension of $\mathfrak{r}'' $ is at least $ 2[F^1: \mathbb{Q}]-7|\Sigma^1\setminus\Sigma^1_p|-2 \ge [F^1: \mathbb{Q}]+2$.
	    		\item Every irreducible one-dimensional prime $\mathfrak{t}\supset\mathfrak{r}''$ gives a two-dimensional Galois representation $\rho(\mathfrak{t}) : G_{F^1, \Sigma^1} \to \operatorname{GL}_2(k(\mathfrak{t}))$ satisfying (2) of \cite[Definition 4.1.1]{pan2022fontaine}. Since $\bar{\chi}|_{G_{F^1_v}} = \mathbf{1}$ for every $v|p$, its third condition follows from \cite[Lemma 4.1.6 (2)]{pan2022fontaine}.
    	\end{itemize}
    	
    	Now choose an irreducible one-dimensional prime containing $\mathfrak{r}''$, and we know that it is actually a nice prime in the sense of \cite[Definition 4.1.3]{Zhang:2024aa}. Therefore, from our assumption on $F^1$, \cite[Proposition 3.1.5]{Zhang:2024ab} and \cite[Theorem 4.1.4]{Zhang:2024aa} imply that $Y$ is pro-modular, which is contrary to our assumption. Hence, we only need to show that $Z_1$ is not empty.
    	
	    	Note that $\mathfrak{r}$ is pro-modular. Using \cite[Corollary 3.3.10]{Zhang:2024aa}, we know that there exists a pro-modular prime $\mathfrak{r}_1$ contained in $\mathfrak{r}$ of dimension at least $1+2[F^1: \mathbb{Q}]$. Arguing as above, we can find a nice prime containing $\mathfrak{r}_1$. Then \cite[Proposition 3.1.5]{Zhang:2024ab} and \cite[Theorem 4.1.4]{Zhang:2024aa} imply that at least one of the irreducible components containing $\mathfrak{r}$ is pro-modular. Hence, $Z_1$ is not empty.
    	
    	Next we show that the irreducible component $C$ contains an irreducible one-dimensional pro-modular prime. 
    	
    	We use some arguments in \cite[Section 7.3]{pan2022fontaine}. Let $R^{\ps}_v $ be the universal pseudo-deformation ring of the pseudo-representation $\mathbf{1}+\mathbf{1}$ for $G_{F^1_v}$ with determinant $\chi|_{G_{F^1_v}}$, and by \cite[Corollary 9.13]{Pa_k_nas_2013}, $R^{\ps}_v $ is a power series ring of relative dimension $3$ over $\cO$. Let $R_{v, \mathbf{1}}^{\ps,\ord}$ be the universal deformation ring of the one-dimensional representation $\mathbf{1}$ (trivial character) of $G_{F^1_v}$, and it is a power series ring of relative dimension $2$ over $\cO$. For each character $\psi_{v,1}$ lifting $\mathbf{1}$, $\psi_{v,1}+\psi_{v,1}^{-1}\chi$ defines a pseudo-deformation of $\mathbf{1}+\mathbf{1}$, which induces a finite map $\phi: R_v^\ps \to R_{v, \mathbf{1}}^{\ps,\ord}$. Consider the surjective map:
    	\[\varphi: R^{\ps}_v\hat{\otimes}_{\cO} R^{\ps,\ord}_{v,\mathbf{1}}\to R^{\ps,\ord}_{v,\mathbf{1}}:~a\otimes b\mapsto \phi(a)b.\]
    	  By \cite[Lemma 10.106.4]{stacks-project}, the kernel of $\varphi$ can be generated by three elements.
    	
    	Let $R^{\ps,\ord}_{p,\mathbf{1}}$ (resp. $R_{p, F^1}^\ps$) be the completed tensor product of the rings $R^{\ps,\ord}_{v,\mathbf{1}}$ (resp. $R_v^\ps$), $v|p$, over $\cO$. Write $\kP$ for the generic point of $C$, and let $C^{\ord,\mathbf{1}}$ be its pullback to $\Spec R^{\ps,\ord}_1$. Here $R^{\ps,\ord}_1$ is the ordinary ring of \cite[6.1.1]{pan2022fontaine}, whose data include the characters $\psi_{v,1}$ with respect to which $T$ is ordinary; by \cite[Lemma 5.3.2]{pan2022fontaine}, these are the characters on the stable lines at irreducible points.
    	
    	Put $A_C=(R_{F^1}^{\ps,\chi}/\kP)\otimes_{R_{p,F^1}^{\ps}}R_{p,\mathbf{1}}^{\ps,\ord}$. We have
    	\[
    	A_C\cong\big((R_{F^1}^{\ps,\chi}/\kP)\widehat\otimes_{\cO}R_{p,\mathbf{1}}^{\ps,\ord}\big)\otimes_{R_{p,F^1}^{\ps}\widehat\otimes_{\cO}R_{p,\mathbf{1}}^{\ps,\ord}}R_{p,\mathbf{1}}^{\ps,\ord}.
    	\]
    	Each local character ring introduces two parameters, while the graph map has kernel generated by three elements. Thus $\dim A_C\ge\dim C-[F^1:\Q]$. By Lemma \ref{lem:ordinary-character-quotients}, $\dim C^{\ord,\mathbf{1}}=\dim A_C\ge1+[F^1:\Q]$. Choose an irreducible component $C_1^{\ord,\mathbf{1}}$ of largest dimension.
    	
    	Note that in our case, the Iwasawa algebra $\Lambda_{F^1}$ is regular of dimension $1+[F^1:\Q] $. Then by Theorem \ref{ord2} and the inequality $\dim C_1^{\ord,\mathbf{1}} \ge 1+[F^1:\Q]$, we get a finite surjective morphism $C_1^{\ord,\mathbf{1}} \to \Spec \Lambda_{F^1} $. As Leopoldt's conjecture holds for $F^1$, we know that the dimension of the reducible locus of $ R_{F^1}^{\ps, \chi}$ is at most $2 ~( < 1+[F^1: \Q])$. Hence, modular points are Zariski dense in $C_1^{\ord,\mathbf{1}} $. Furthermore, we can find an irreducible one-dimensional pro-modular prime in $C_1^{\ord,\mathbf{1}} $, and hence in $C$.
    	
    	Consequently, we obtain our result.
    	
    \end{proof}
    
    Note that by Corollary \ref{lower}, we have $\dim R_{F^1}^{\ps, \chi} \ge 2[F^1:\Q]+1 $. Consider the natural composite map $f: R_{F^1}^{\ps, \chi} \twoheadrightarrow \T_\km \twoheadrightarrow (\T_\km)^{\operatorname{red}}$. By \cite[Corollary 3.3.10]{Zhang:2024aa}, every minimal prime of $(\T_\km)^{\operatorname{red}}$ is of dimension at least $1+2[F^1:\Q]$. Thus, the map $f$ factors through $R_{F^1,0}^{\ps, \chi}$, and we get a surjection $R_{F^1,0}^{\ps, \chi} \twoheadrightarrow (\T_\km)^{\operatorname{red}}$.
   
   \begin{prop}\label{R=T}
   	The surjection $R_{F^1,0}^{\ps, \chi} \twoheadrightarrow (\T_\km)^{\operatorname{red}}$ is an isomorphism.
   \end{prop}
   
   \begin{proof}
   	Note that by Proposition \ref{def of 0}, both $R_{F^1,0}^{\ps, \chi} $ and $(\T_\km)^{\operatorname{red}}$ are reduced. Thus, we only need to show that each prime ideal of $R_{F^1,0}^{\ps, \chi}$ is pro-modular, and it follows directly from Proposition \ref{potential R=T}.
   \end{proof}
   
  \begin{remark}
  	\label{rem:twist_pro_modularity}
  	It is enough to assume that the residual character \(\bar\chi\), rather than the chosen lift \(\chi\), extends to \(G_{\mathbb Q}\). We can use the following arguments.
  	
  	Assume that $\bar{\chi}$ can be extended to a character of $G_\Q$, and let \(\chi\) be an arbitrary de Rham lift of \(\bar\chi\). Since \(F^1/\mathbb Q\) is abelian and totally real, Leopoldt's conjecture holds for \(F^1\). Then we may write $\chi= \psi \varepsilon^k$, where $\psi$ is a finite order character. Write \(
  	 \chi_{0}:=\tilde{\bar{\psi}}\varepsilon^k \), where $\tilde{\bar{\psi}} $ is the Teichmüller lifting of $\bar{\psi}$. Note that $\chi_{0} $ is a character of $G_\Q$ and it ramifies at the same set of places away from $p$ as $\bar{\chi}$. Since \(p\) is odd, after enlarging \(\mathcal O\) if necessary, there exists a finite order character
  	 \(
  	 \varphi:G_F\to \mathcal O^\times
  	 \)
  	 lifting $\mathbf{1}$ such that \(\varphi^2=\chi_{0}\chi^{-1}\). By Proposition \ref{det}, we get an isomorphism $R_{F^1}^{\ps, \chi} \cong R_{F^1}^{\ps, \chi_0}$.
  	 On the automorphic side, twisting \(\varphi\) induces an isomorphism between the corresponding big Hecke algebras, after possibly replacing the tame level by a sufficiently small one. See \cite[3.2.2]{pan2022fontaine} for example. Consequently, pro-modularity is preserved under this twist. Therefore Proposition \ref{potential R=T} and Proposition \ref{R=T} for the determinant \(\chi_0|_{G_{F^1}}\) imply the same assertions for the original determinant \(\chi|_{G_{F^1}}\). 
  \end{remark}
   

    \begin{remark}\label{rem:potential-weight-one}
    	A Hilbert analogue of Pan's classicality theorem
    	\cite[Theorem 6.4.7]{Pan_2022}, applicable to the pro-modular
    	representations supplied by Proposition \ref{potential R=T},
    	would imply the full residually reducible weight-one Fontaine--Mazur
    	conjecture over $\Q$ for every odd prime $p$; compare
    	\cite[Remark 6.4.10]{Pan_2022}. More precisely, the conclusion of
    	Theorem \ref{thm:irregular-FM-Eisenstein} would hold under
    	hypotheses \textup{(1)--(4)} alone, without \textup{(5)}.
    	 It suffices to consider the remaining case of Remark \ref{rem:irregular-FM-Eis-p3},
    	in which $\rho|_{G_{\Q_p}}$ is absolutely irreducible. After a finite-order twist,
    	take $F=\Q$ and $\chi=\det\rho$ above. Since $p$ splits completely in $F^1$,
    	the local restriction is unchanged, so $\rho|_{G_{F^1}}$ remains
    	irreducible. Proposition \ref{def of 0}(2) places its prime in
    	$\Spec R_{F^1,0}^{\ps,\chi}$, so Proposition \ref{potential R=T}
    	makes it pro-modular. The proposed Hilbert classicality theorem
    	would then give parallel-weight-one automorphy over $F^1$,
    	and soluble descent would give a classical weight-one eigenform
    	over $\Q$, hence finite image.
    \end{remark}

    \section{Zariski density of modular points}\label{sec 5}
    
   In this section, we study the geometry of modular points in Eisenstein pseudo-deformation spaces. We begin in Section \ref{sec for emerton} by recalling the infinite fern arguments over \(\mathbb Q\). In Section \ref{sec:lower}, we establish sharp lower bounds for the dimensions of irreducible components of global pseudo-deformation rings. The main difficulty here is that, unlike usual deformation rings, pseudo-deformation rings cannot in general be studied directly through the tangent spaces of Galois cohomology. We overcome this difficulty by combining a connectedness-dimension argument with a detailed analysis of the reducible locus.
   
	   In Section \ref{F=Q}, we combine these dimension bounds with Newton--Thorne's vanishing theorem for adjoint Bloch--Kato Selmer groups and Pan's ordinary finiteness theorems to prove the Zariski density of modular points over \(\mathbb Q\). The use of Newton--Thorne removes the global residual hypothesis required by Kisin's earlier vanishing theorem. As an arithmetic application, Section \ref{irr fm} applies Pan's weight-one classicality theorem to deduce an irregular Fontaine-Mazur theorem in the residually reducible case.
   
   We then turn to certain totally real fields in Section \ref{F ne Q}. The same adjoint Selmer-group input applies in this setting, while local-global compatibility for completed cohomology and Paškūnas' theory provide the required lower bound for the dimension of the big Hecke algebra. This leads to a big \(R=\mathbb T\) theorem for the union of the high-dimensional irreducible components.
   
   Finally, Section \ref{q1} considers a more refined version of the density problem. Instead of retaining only the residual pseudo-representation, we fix a non-split residual extension and pass to its usual deformation ring. We formulate questions concerning the existence and density of modular points on its irreducible components, with the ramification set kept fixed throughout.
    
    For notations, let $p \ge 3$ be an odd prime. We denote by $K$ a $p$-adic local field with its ring of integers $\cO$, a uniformizer $\varpi$ and its residue field $\F$. 
    
    \subsection{Emerton's setting}\label{sec for emerton}
    In this subsection, we briefly recall the construction of the $p$-adic big Hecke algebra. Our main reference here is \cite{Emerton2011}.
    
    We fix a positive integer $N \ge 1$ such that $p$ does not divide $N$. Write $\Sigma$ for the set of primes of $\Q$ dividing $Np$.
    
    For an integer $k \ge 1 $, write $\mathcal{M}_k(N)$ for the space of modular forms of weight $k$ and level $\Gamma_1(N)$ (or level $N$ for simplicity). Recall that the map $\begin{pmatrix}
    	a & b\\
    	c& d
    \end{pmatrix}  \mapsto d~ \modd N$ induces an isomorphism $ \Gamma_0(N)/\Gamma_1(N) \cong (\Z/N\Z)^\times$. If $d \in (\Z/N\Z)^\times$, this map induces an automorphism of $\mathcal{M}_k(N)$ by $\langle d \rangle$, the \textit{diamond operator}.
    
    Let $l$ be a prime not dividing $N$. We define the automorphism $S_l$ of $\mathcal{M}_k(N)$ via the formula $S_l = \langle l \rangle l^{k-2}$. We also define the endomorphism $T_l$ of  $\mathcal{M}_k(N)$ via the formula 
    $$ (T_l f)(\tau)= \sum_{n=0}^{\infty} c_{nl}(f)q^n + \sum_{n=0}^{\infty} l c_n(S_lf)q^{nl},~~ \tau \in \mathcal{H},~~ q=\exp(2\pi i \tau),$$
    where $c_n(f)$ is the $n$-th Fourier coefficient of the modular form $f$.
    
    The \textit{Hecke algebra} $\T_k(N)$ (or just $\T_k$) is defined to be the $\Z$-subalgebra of $\End(\mathcal{M}_k(N))$ generated by the Hecke operators $lS_l$ and $T_l$ as $l$ ranges over primes not dividing $N$.
    
    Let $\T_{\le k}^{(p)} (N)$ (or just $\T_{\le k}^{(p)}$) be the $\Z$-algebra of endomorphisms of $\oplus_{i=1}^k \mathcal{M}_i(N)$ generated by the operators $lS_l$ and $T_l$ as $l$ ranges over primes not dividing $Np$. From the definition, we see that if $k' > k$, then there is a surjection $\T_{\le k'}^{(p)} \twoheadrightarrow \T_{\le k}^{(p)}$. The $p$\textit{-adic Hecke algebra} $\T(N)$ (or just $\T$) is defined to be the projective limit
    $$ \T:= \varprojlim_k \Z_p \otimes_\Z \T_{\le k}^{(p)}. $$
    
    \begin{thm}\label{decomposition}
    	The ring $\T$ is a $p$-adically complete, Noetherian $\Z_p$-algebra, and is the product of finitely many noetherian local $\Z_p$-algebras.
    \end{thm}
    
    \begin{proof}
    	See \cite[Theorem 2.7]{Emerton2011}.
    \end{proof}
    
    \begin{conj}\cite[Conjecture 2.9]{Emerton2011}\label{conj}
    	The ring $\T$ is equidimensional of Krull dimension $4$.
    \end{conj}
    
    Using the theory of eigencurves, we can obtain the lower bound of $\dim \T$.
    
    \begin{thm}[Infinite fern]\label{lower bound}
    	The Krull dimension of each irreducible component of $\Spec \T$ is at least $4$. Moreover, if $C$ is an irreducible component of $\Spec \T$ of dimension $4$, then modular points in $C$ are Zariski dense.
    \end{thm}
    
    \begin{proof}
    	See \cite[Theorem 2.27 \& Corollary 2.28]{Emerton2011}.
    \end{proof}
    
    If we replace $\Z_p$ by $\cO$ in all definitions above, we can define a $p$-adic Hecke algebra $\T'$ over $\cO$. It is clear that we have $\T' = \T \otimes_{\Z_p} \cO$. From now on, we assume that our Hecke algebra $\T$ is defined over $\cO$. 
    
    By Theorem \ref{decomposition}, we know that $\T$ is semi-local. If we assume $\km_1, \cdots, \km_r$ are all maximal ideals of $\T$, combining \cite[24.C]{matsumura80}, we have
    $$ \T \cong \T_{\km_1} \times \cdots \times \T_{\km_r},$$
    where each $ \T_{\km_i}$ is $\km_i$-adically complete.
    
    Let $\km$ be a maximal ideal of $\T$. Here we use arithmetic Frobenius $\Frob_l^{\mathrm{arith}}=\Frob_l^{-1}$. There exists a pseudo-representation $T_\km: G_{\Q, \Sigma} \to \T_\km$ such that $T_\km(\Frob_l^{\mathrm{arith}})=T_l$ and $\det(T_\km)(\Frob_l^{\mathrm{arith}})=lS_l$ (see \cite[Lemma 4.1.1]{Deo2023}). Let $R_{\Q}^\ps$ be the universal pseudo-deformation ring of $(T_\km \modd \km)$ for $G_{\Q, \Sigma}$. Then by the universal property, we get a surjection $R_{\Q}^\ps \twoheadrightarrow \T_\km$.
    
    \subsection{The lower bound of pseudo-deformation rings}\label{sec:lower}
    In this subsection, we prove a lower bound of the Krull dimension of the global pseudo-deformation ring when $F=\mathbb{Q}$.
    
    Let $\Sigma$ be a finite set of finite places of $\Q$ containing $p$. Let $\bar{T}: G_{\Q, \Sigma} \to \F$ be a two-dimensional pseudo-representation (not necessarily reducible) satisfying Assumption \ref{odd ass}. Write $\bar{\chi}:= \operatorname{det} \bar{T}$. Then there exists a universal pseudo-deformation ring $R_\Q^\ps$ of the pseudo-representation $\bar{T}$. Let $\chi: G_{\Q, \Sigma} \to \cO^\times$ be a continuous character lifting $\bar{\chi}$. Let $ R_\Q^{\ps, \chi}$ be the universal pseudo-deformation ring of the pseudo-representation $\bar{T}$ with determinant $\chi$.
    
    To study the lower bound of $\dim R_\Q^{\ps, \chi}$, the following result is crucial in this subsection.
    
    \begin{prop}\label{Gal coh}
    	Let $\bar\psi$ be the character of $G_{\Q, \Sigma}$ unramified at $p$ such that $\bar{\chi} = \bar\psi\omega^{s_0-1}$ for some integer $1 \leq s_0 \leq p-1$. Let $\varepsilon$ be the $p$-adic cyclotomic character.
    	\begin{enumerate}
    		\item There exist infinitely many integers $s > 2$ such that $s \equiv s_0 \mod{p-1}$ and $$\dim_K(H^1(G_{\Q, \Sigma}, \Psi\varepsilon^{s-1})) = \dim_K(H^1(G_{\Q, \Sigma}, \Psi^{-1}\varepsilon^{1-s})) = 1$$ for any lift $\Psi : G_{\Q, \Sigma} \to K^{\times}$ of $\bar\psi$ unramified at $p$.
    		\item Let $\Psi : G_{\Q, \Sigma} \to K^{\times}$ be a lift of $\bar\psi$ which is unramified at $p$. Let $s$ be an integer satisfying the conclusion of the first part. Let $y \in H^1(G_{\Q, \Sigma}, \Psi\varepsilon^{s-1})$ be a non-zero element and $\rho_y : G_{\Q, \Sigma} \to \GL_2(K)$ be the representation corresponding to $y$. Then $\dim_K(H^1(G_{\Q, \Sigma},\operatorname{ad}(\rho_y)))=3$.
    	\end{enumerate}
    \end{prop}
    
    \begin{proof}
    	See \cite[Proposition 2.4.4]{Deo2023}.
    \end{proof}
    
    Write $\mathbb{S}$ for the set of all positive integers $s$ satisfying Proposition \ref{Gal coh} (1). Then there exists an infinite subset $\mathbb{S}' \subset \mathbb{S}$ such that $s_1 \equiv s_2 \pmod{2(p-1)}$ for any $s_1, s_2 \in \mathbb{S}'$.
    
    Let $R_\mathbf{1}:=\cO[[G_{\Q, \Sigma}^{\operatorname{ab}}(p)]]$ be the universal deformation ring of the one-dimensional trivial character $\mathbf{1} :  G_{\Q, \Sigma} \to \F^\times$. We may write $R_\mathbf{1}$ in the form $\cO[[X, X_1, \cdots, X_t]]/((1+X_1)^{p^{r_1}}-1, \cdots, (1+X_t)^{p^{r_t}}-1)$, where $X$ defines the cyclotomic line and the variables $X_i$ parametrize the finite torsion part. This ring is $\cO$-flat and equidimensional of dimension $2$. Thus each minimal prime is determined entirely by the torsion part. By global class field theory (the Kronecker-Weber theorem), after fixing such a minimal prime and specializing the remaining cyclotomic variable $X$ to $0$, the universal character gives a character lifting $\mathbf{1}$ which is unramified at $p$; over the component itself, however, the universal character still varies along the cyclotomic line. Enlarging $\cO$ if necessary, for every minimal prime $\kp$ of $R_\mathbf{1}$, the ring $R_\mathbf{1}/\kp$ is isomorphic to a power series ring of relative dimension $1$ over $\cO$.
    
    \begin{prop}\label{at least 3}
    	Each irreducible component of $\Spec R_\Q^{\ps, \chi}$ is of dimension at least $3$.
    \end{prop}
    
    \begin{proof}
    	If $\bar{T}$ is irreducible in the sense of Definition \ref{construction}, then we can construct an absolutely irreducible Galois representation $\bar{\rho}: G_{\Q, \Sigma} \to \GL_2(\F)$ satisfying $\bar{T}=\operatorname{tr}(\bar{\rho})$. Then by \cite[Theorem 2.4.1, page 44]{Berger_2013}, $R_\Q^{\ps, \chi}$ is isomorphic to the universal deformation ring of $\bar{\rho}$ with fixed determinant $\chi$. In this case, our result follows from classical Galois deformation theory and the global characteristic formula. See \cite[Proposition 6.1.4]{Allen_2019}.
    	
    	Now we focus on the reducible case. After some finite twist, we may assume $\bar{T}=1+\bar{\chi}$. Let $\chi': G_{\Q, \Sigma} \to \cO^\times$ be a character lifting $\bar{\chi}$ of the form $\Psi\varepsilon^{s-1}$, where $\Psi: G_{\Q, \Sigma} \to \cO^{\times}$ is a character lifting $\bar{\psi}$ (defined as in Proposition \ref{Gal coh}) unramified at $p$ and $s \in \mathbb{S}'$. Using Proposition \ref{det}, we only need to prove our claim for the pseudo-deformation ring $R_\Q^{\ps, \chi'}$. To simplify our notations, we may identify $\chi$ as $\chi'$ in the following.
    	
    	Let $R_\Q^{\ps, \operatorname{red}}$ be the reducible locus of $R_\Q^{\ps, \chi}$. By Corollary \ref{structure of ps}, if there exists an irreducible component
    	\(C\) of $\Spec R_\Q^{\ps, \chi}$ of dimension at most \(2\),
    	then \(C\) is contained in $\Spec R_\Q^{\ps, \operatorname{red}}$. Let
    	\(\mathfrak p_0\) be the generic point of \(C\). Thus we have
    	surjections
    	$
    	R_{\mathbf 1}
    	\twoheadrightarrow R^{\mathrm{ps},\mathrm{red}}_{\mathbb Q}
    	\twoheadrightarrow R_\Q^{\ps, \chi}/\mathfrak p_0
    	$.
    	
    	We first exclude the case \(\dim R_\Q^{\ps, \chi}/\mathfrak p_0=1\). Choose a minimal prime \(\mathfrak p_0'\) of \(R_{\mathbf 1}\) through which the above composite map factors, and put $A:=R_\Q^{\ps,\chi}/\mathfrak p_0$ and $B:=R_{\mathbf 1}/\mathfrak p_0'$. Let $q:B\twoheadrightarrow A$ be the induced map, let $\mu_B$ be the universal character over $B$, and put $\mu_A:=q\circ\mu_B$. The pseudo-representation $T_B:=\mu_B+\chi\mu_B^{-1}$ induces a map $g:R_\Q^{\ps,\chi}\to B$ whose composite with $q$ is the quotient map to $A$. Since $B$ is a domain, $\ker g$ is a prime contained in the minimal prime $\mathfrak p_0$; hence $\ker g=\mathfrak p_0$. Thus $g$ induces an injection $i:A\hookrightarrow B$ such that $q\circ i=\operatorname{id}_A$, and
    	\[
    	\mu_B+\chi\mu_B^{-1}=i(\mu_A)+\chi i(\mu_A)^{-1}.
    	\]
	    	Applying Brauer--Nesbitt over the fraction field of $B$ to the preceding equality, and then using that $\bar T=\mathbf{1}+\bar\chi$ is multiplicity-free to distinguish the two residual summands, we obtain $\mu_B=i(\mu_A)$. The ring $B$, being a quotient of $R_{\mathbf 1}$, is topologically generated over $\cO$ by the values of $\mu_B$. Hence $i$ is also surjective, so $A\cong B$, contradicting $\dim A=1$ and $\dim B=2$.
    	
    	Now we know that \(C\) is a two-dimensional component contained in the reducible
    	locus. As $\dim R_\mathbf{1}=2$, we may view \(C\) as an irreducible component of
	    	\(\operatorname{Spec} R_{\mathbf 1}\). Let \(\Phi\) be the unramified-at-\(p\) character obtained from the torsion part of the corresponding minimal prime by specializing its free cyclotomic variable to $0$.
    	
    	Choose \(s'\in \mathbb S'\). Since \(s\equiv s'\pmod {2(p-1)}\), we may
    	write
    	\(
    	\lambda=\frac{s-s'}{2(p-1)}\in \mathbb Z
    	\). Then the characters
    	\[
    	\chi_1:=\Phi\,\varepsilon^{\lambda(p-1)},\qquad
    	\chi_2:=(\Psi\Phi^{-1})\varepsilon^{s-1-\lambda(p-1)}
    	\]
    	are crystalline characters lifting \(1\) and \(\bar\chi\), respectively,
    	and satisfy \(\chi_1\chi_2=\chi\). Hence they define a reducible
    	pseudo-representation
    	\(
    	\chi_1+\chi_2
    	\)
    	with determinant \(\chi\), which defines a point in $C$. Let \(\mathfrak p_{s'}\) be the corresponding
    	one-dimensional prime of \(R_\Q^{\ps, \chi}\), and we have
    	\(
    	\mathfrak p_{s'}\supset \mathfrak p_0
    	\).
    	
    	Consider the character
    	\(
    	\chi_2/\chi_1
    	=
    	\Psi\Phi^{-2}\varepsilon^{s'-1}
    	\). Note that \(\Psi\Phi^{-2}\) is a character lifting $\bar{\psi}$ unramified at \(p\). By Proposition \ref{Gal coh}, we have
    	\(
    	\dim_K H^1(G_{\mathbb Q,\Sigma},\chi_2/\chi_1)
    	=
    	\dim_K H^1(G_{\mathbb Q,\Sigma},\chi_1/\chi_2)
    	=
    	1
    	\).
    	Choose a non-zero element
    	\(
    	y\in H^1(G_{\mathbb Q,\Sigma},\chi_2/\chi_1)
    	\)
    	and let
    	\(
    	\rho_y:G_{\mathbb Q,\Sigma}\longrightarrow \mathrm{GL}_2(K)
    	\)
    	be the corresponding non-split extension of \(\chi_1\) by \(\chi_2\).
    	Let \(R_y\) be the universal deformation ring of \(\rho_y\) with fixed
    	determinant \(\chi\). By the global characteristic formula and
    	Proposition \ref{Gal coh}, the ring \(R_y\) is a power series ring of dimension \(2\)
    	over its residue field \(K\).
    	
    	By \cite[Proposition 2.2.1]{pan2022fontaine}, the ring $\widehat{(R_\Q^{\ps, \chi})_{\mathfrak p_{s'}}}$ is the universal pseudo-deformation ring of the pseudo-representation $\chi_1+\chi_2 $ with fixed determinant $\chi$. Combining \cite[Corollary 1.4.4 (2)]{kisin2009fontaine} and Proposition \ref{Gal coh}, we obtain a natural surjection
    	\(
    	\widehat{(R_\Q^{\ps, \chi})_{\mathfrak p_{s'}}}
    	\twoheadrightarrow R_y
    	\). In particular, we have
    	\(
    	\dim (R_\Q^{\ps, \chi})_{\mathfrak p_{s'}}
    	=
    	\dim \widehat{(R_\Q^{\ps, \chi})_{\mathfrak p_{s'}}}
    	\geq \dim R_y
    	=
    	2
    	\).
    	
    	Let \(I_\mathbb{Q}\) be the intersection of all minimal primes \(\mathfrak q\) of
    	\(R_\Q^{\ps, \chi}\) such that \(\dim R_\Q^{\ps, \chi}/\mathfrak q\geq 3\). The preceding inequality
    	implies that, for every \(s'\in \mathbb S'\), the prime
    	\(\mathfrak p_{s'}\) contains one of these minimal primes. Hence
    	\(
    	I_\Q\subset \mathfrak p_{s'}\) for all 
    	\(s'\in \mathbb S'
    	\) and
    	\(
    	I_\Q\subset \cap_{s'}\mathfrak p_{s'}
    	\).
    	
    	On the other hand, since the set \(\mathbb S'\) is infinite, the set of points
    	\(\{\mathfrak p_{s'}\}_{s'\in \mathbb S'}\) is Zariski dense in \(C\).
    	Thus
    	\(
    	\cap_{s'}\mathfrak p_{s'}
    	=
    	\mathfrak p_0
    	\). 
    	It follows that
    	\(
    	I_\Q\subset \mathfrak p_0
    	\).
    	Since \(R_\Q^{\ps, \chi}\) is Noetherian, \(I_\Q\) is the intersection of finitely many
    	minimal primes of \(R_\Q^{\ps, \chi}\) of dimension at least \(3\). By \cite[Lemma 10.15.1]{stacks-project},
    	\(\mathfrak p_0\) contains one such minimal prime. This contradicts the
    	fact that \(\mathfrak p_0\) itself is a minimal prime of \(R_\Q^{\ps, \chi}\) with
    	\(\dim R_\Q^{\ps, \chi}/\mathfrak p_0=2\), and we obtain our result.
    	
    	 \end{proof}
    
      Now we turn to the universal pseudo-deformation ring $R_\Q^\ps$. Recall that by Proposition \ref{det}, we have an isomorphism $ R_\Q^\ps \cong R_\Q^{\ps, \chi}\widehat{\otimes}_\cO R_\mathbf{1}$.
    
    \begin{cor}\label{at least 4}
	    The following statements hold.
    	\begin{enumerate}
	    	\item Suppose each irreducible component of $\Spec R_\Q^{\ps, \chi}$ is of characteristic $0$. Then each irreducible component of $\Spec R_\Q^\ps$ is of characteristic $0$ and of dimension at least $4$.
	    	\item Let $\kp$ be a minimal prime of $R_\Q^\ps$ of characteristic $0$. Then there exist a minimal prime $\kp_\chi$ of $R_\Q^{\ps, \chi} $ and a minimal prime $\kq_\chi$ of $R_\mathbf{1}$ such that $R_\Q^\ps/\kp \cong (R_\Q^{\ps, \chi}/\kp_\chi) \widehat{\otimes}_\cO (R_\mathbf{1}/\kq_\chi)$. Moreover, $\kp_\chi$ and $\kq_\chi$ are the inverse images of $\kp$ in $R_\Q^{\ps, \chi} $ and $R_\mathbf{1}$, respectively.
    	\end{enumerate}
    \end{cor}
    
    \begin{proof}
	    We first prove the second claim. Let $\kp$ be a minimal prime of $R_\Q^\ps$ of characteristic $0$. Let $\kp_\chi$ and $\kq_\chi$ be the inverse images of $\kp$ in $R_\Q^{\ps, \chi} $ and $R_\mathbf{1}$, respectively. As $R_\mathbf{1}$ is $\cO$-flat, by \cite[Lemma 4.2.4]{pan2022fontaine}, the natural map $ R_\Q^{\ps, \chi} \to R_\Q^{\ps, \chi}\widehat{\otimes}_\cO R_\mathbf{1}$ satisfies the going-down property. Hence, $\kp_\chi$ is a minimal prime of $R_\Q^{\ps, \chi}$. Since $\kp$ is of characteristic $0$, so is $\kp_\chi$.
    	
    	Let $\kq_0$ be a minimal prime of $R_\mathbf{1}$ contained in $\kq_\chi$. As both $ \kp_\chi$ and $\kq_0$ are of characteristic $0$ (hence define $\cO$-flat domains), by \cite[Lemma 4.2.4]{pan2022fontaine}, we get natural maps $R_\Q^\ps \twoheadrightarrow (R_\Q^{\ps, \chi}/\kp_\chi) \widehat{\otimes}_\cO (R_\mathbf{1}/\kq_0) \to R_\Q^\ps/\kp$. As we explained previously, $R_\mathbf{1}/\kq_0$ is isomorphic to a power series ring over $\cO$. Thus $(R_\Q^{\ps, \chi}/\kp_\chi) \widehat{\otimes}_\cO (R_\mathbf{1}/\kq_0)$ is an $\cO$-flat integral domain. As $\kp$ is minimal, we can conclude that the kernel of the first arrow is also $\kp$, and hence the second arrow is an isomorphism. To finish the proof of the second claim, we need to show $\kq_0=\kq_\chi$. As $ R_\Q^{\ps, \chi}/\kp_\chi$ is $\cO$-flat, by \cite[Lemma 4.2.4]{pan2022fontaine}, the natural map $R_\mathbf{1}/\kq_0 \to R_\Q^\ps/\kp $ satisfies the going-down property. The contraction $\kq_\chi/\kq_0$ of the zero prime is therefore a minimal prime of the domain $R_\mathbf{1}/\kq_0$, so $\kq_0=\kq_\chi$.

	    Now suppose that every irreducible component of $\Spec R_\Q^{\ps, \chi}$ is of characteristic $0$, and let $\kp$ be any minimal prime of $R_\Q^\ps$. The same going-down argument shows that its inverse image $\kp_\chi$ in $R_\Q^{\ps, \chi}$ is minimal, and hence is of characteristic $0$. It follows that $\kp$ is of characteristic $0$, so the second claim applies. Since $R_\mathbf{1}/\kq_\chi$ is a power series ring of relative dimension $1$ over $\cO$, the dimension assertion follows from Proposition \ref{at least 3}.
    \end{proof}
    
    \begin{rem}
    	We do not know whether $R_\Q^{\ps, \chi}$ and $R_\Q^\ps$ are $\cO$-flat rings in general.
    \end{rem}
    
    \subsection{The case $F =\mathbb{Q}$}\label{F=Q}
    In this subsection, we prove the Zariski density of modular points in the pseudo-deformation space when $F=\Q$.
    
    Keep notations as in the previous subsection. Let $\Sigma$ be a finite set of rational primes containing $p$. Let $\bar{\chi}: G_{\mathbb{Q}, \Sigma} \to \F^\times$ be a continuous odd character. Let $R_{\Q}^\ps$ be the universal pseudo-deformation ring of the pseudo-representation $1+\bar{\chi}$ for $G_{\Q, \Sigma}$. We say a prime of $R_{\Q}^\ps$ is a \textit{modular point} if the corresponding pseudo-deformation comes from an eigenform up to finite twist.
    
    As an analogue of the big $R=\T$ theorem in the residually irreducible case (see \cite[Theorem 2.1]{Nakamura2023} for example), we prove the following theorem in the residually reducible case.
    
    \begin{thm}\label{big R=T Q}
    	Keep notations as above. Assume that $\bar{\chi}|_{G_{\mathbb{Q}_p}} \ne \omega^{\pm 1}$.
    	Then  each irreducible component of $\Spec R_{\Q}^\ps$ is of characteristic $0$ and Krull dimension $4$. Furthermore, modular points are Zariski dense in  $\Spec R_{\Q}^\ps$.
    \end{thm}
    
    \begin{rem}
    	To prove Theorem \ref{big R=T Q}, we combine the strategy in \cite[Section 7]{pan2022fontaine} and \cite{Deo2023}. The main obstruction is the complexity of universal pseudo-deformation rings. In Deo's setting, under the assumption of the cyclicity of $H^1(G_{\Q, \Sigma}, \bar{\chi})$, he can use the classical Galois deformation rings instead and theory of \textbf{GMA} (Generalized Matrix Algebra). To show Theorem \ref{big R=T Q}, we have to deal with pseudo-deformation rings directly.
    \end{rem}
    
    Let $\chi: G_{\Q, \Sigma} \to \cO^\times$ be a character lifting $\bar{\chi}$ of the form $\Phi\varepsilon^{s-1}$, where $\Phi: G_{\Q, \Sigma} \to \cO^{\times}$ is a character lifting $\bar{\psi}$ (defined as in Proposition \ref{Gal coh}) unramified at $p$ and $s \in \mathbb{S}'$. Let $R_\Q^{\ps, \chi}$ be the universal pseudo-deformation ring of the pseudo-representation $1+\bar{\chi}$ for $G_{\Q, \Sigma}$ with determinant $\chi$. Let $R^{\ps,\ord}$ be the universal ordinary pseudo-deformation ring  defined as in Section \ref{sec deformation ring} (if $ \bar{\chi}|_{G_{\Q_p}} = \mathbf{1}$, we also write $R^{\ps,\ord}$ for  $R_1^{\ps,\ord}$ to ease our notations). 
    

    

    \begin{lem}\label{ordinary locus} Let $C_\chi$ be an irreducible component of $\Spec R_\Q^{\ps, \chi}$. We denote by $ C^\ord$ the pullback of $C_\chi$ to $\Spec R^{\ps, \ord}$ via the natural map $R_\Q^{\ps, \chi} \to R^{\ps, \ord} $.
    	
    	1) The space $C^\ord$ has dimension $2$; in particular, it has an irreducible component of dimension $2$.
    	
    	2) The generic point of $C_\chi$ is of characteristic $0$ and of dimension $3$. 
    \end{lem}
    
    \begin{proof}
    	Let $R_p^\ps$ be the local pseudo-deformation ring of $1+\bar{\chi}|_{G_{\mathbb{Q}_p}}$ with determinant $\chi|_{G_{\Q_p}}$.
    	
    	We first consider the case $\bar{\chi}|_{G_{\Q_p}}=\mathbf{1}$. Theorem \ref{ord2} gives $\dim C^\ord\le\dim\Lambda_\Q=2$. Lemma \ref{lem:ordinary-character-quotients} and the graph calculation in Proposition \ref{potential R=T} give $\dim C^\ord\ge\dim C_\chi-1$. Together with Proposition \ref{at least 3}, this implies $\dim C^\ord=2$ and $\dim C_\chi=3$. Choose a two-dimensional irreducible component $C'$ of $C^\ord$. Its finite map to $\Spec\Lambda_\Q$ is surjective: the corresponding ring map is injective, since both rings have dimension $2$ and $\Lambda_\Q$ is a domain. Thus $C'$ contains a characteristic-zero point, and so $C_\chi$ is of characteristic $0$.
    	
    	Now assume $\bar{\chi}|_{G_{\Q_p}} \ne \mathbf{1}$. Let $\Spec R_p^{\ps, \ord}$ be the ordinary locus of $\Spec R_p^\ps$. By \cite[Corollary B.20]{Pa_k_nas_2013}, the kernel of $ R_p^{\ps} \twoheadrightarrow R_p^{\ps, \ord}$ is principal. Then by Proposition \ref{at least 3}, every irreducible component of $C^\ord$ is of dimension at least $2$. Then we can use the same proof as in the previous paragraph, just replacing Theorem \ref{ord2} by Theorem \ref{ord1}.

    \end{proof}
    
    Using Lemma \ref{ordinary locus} and Corollary \ref{at least 4}, we can conclude the first part of Theorem \ref{big R=T Q}.
    
    Now we prove the second part of the main theorem. Let $f$ be a normalized cuspidal newform such that $\bar\rho_f^{\operatorname{ss}}=1\oplus\bar\chi$ and $\rho_f$ is unramified outside $\Sigma$, and let $N_f$ denote the prime-to-$p$ part of its exact level. For $\ell\ne p$, local-global compatibility gives $v_\ell(N_f)=a_\ell(\rho_f)\le\operatorname{Sw}_\ell(1\oplus\bar\chi)+2$, since the Swan conductor is unchanged under reduction and the tame contribution is at most $2$. This bounds the Artin conductor of $\rho_{f}$  .
    
    For each $\ell\in\Sigma\setminus\{p\}$, choose $n_\ell\ge\operatorname{Sw}_\ell(1\oplus\bar\chi)+2$ and put $N:=\prod_{\ell\in\Sigma\setminus\{p\}}\ell^{n_\ell}$. Then $N_f\mid N$, and passage to tame level $N$ by oldform constructions does not change the associated Galois representation. The same bound applies to the product of the tame conductors of the two characters defining the Eisenstein series below. By the proof of Corollary \ref{lower}, $1\oplus\bar\chi$ defines a maximal ideal $\km$ of $\mathbb T(N)$ (after base change to $\cO$), and there is a natural surjection $R_\Q^\ps\twoheadrightarrow\mathbb T(N)_\km$.
    
    Let $C$ be an irreducible component of $\Spec R_\Q^\ps$. Combining Theorem \ref{lower bound}, we only need to show that $C$ defines an irreducible component of the big Hecke algebra $\mathbb{T}(N)_\km$.
    
     By the second part of Corollary \ref{at least 4}, there is a character $\Psi: G_{\Q,\Sigma}\to\cO^\times$ lifting $\mathbf{1}$ and unramified at $p$ such that some component of $\Spec R_\Q^{\ps,\Psi\chi}$ is contained in $C$. By Lemma \ref{ordinary locus}, its ordinary pullback has a two-dimensional component $C_1^\ord$. We regard this space as finite over its image in $C$, and use the same notation for a point and its image when no confusion is possible. Its finite map to $\Spec\Lambda_\Q$ is surjective by equality of dimensions, using Theorem \ref{ord1} or \ref{ord2} as appropriate.
    
    Let $\Spec R_\Q^{\ps, \operatorname{red}} $ be the reducible locus of $\Spec R_\Q^{\ps, \Psi\chi}$. Recall that there is a natural surjection $R_\mathbf{1} \twoheadrightarrow R_\Q^{\ps, \operatorname{red}}$.
    
    
    \textbf{Case 1}: the image of $C_1^\ord$ is contained in $\Spec R_\Q^{\ps,\operatorname{red}}$.
    
    In this case, the global character lifting $\mathbf{1}$ restricts to one of the two local characters on $C_1^\ord$. Exchanging them if necessary changes the weight map by an automorphism of $\Lambda_\Q$. Applying going up to this finite surjective weight map, we can find a point $\kp_E$ in $C_1^\ord$ defining a reducible pseudo-representation $\chi_1+\chi_2$, where $\chi_1$ lifts $\mathbf{1}$ and is unramified at $p$.
    
    \begin{lem}\label{eisenstein series}
    	1) The representation $\chi_1 \oplus \chi_2$ defines an Eisenstein series.
    	
    	2) We have $\dim_K(H^1(G_{\Q, \Sigma}, \chi_1\chi_2^{-1})) = \dim_K(H^1(G_{\Q, \Sigma}, \chi_2\chi_1^{-1})) = 1$.
    \end{lem}
    
    \begin{proof}
    	The first claim follows from \cite[Lemma 2.3.10]{Deo2023}. The second one follows from Proposition \ref{Gal coh} and our choice of $s$.
    \end{proof}
    
    Consider the Galois representation $\rho_*: G_{\Q, \Sigma} \to \GL_2(K)$ such that for $\sigma \in G_{\Q, \Sigma}$, $\rho_*(\sigma)=\begin{pmatrix}
    	\chi_1(\sigma) & *\\
    	~& \chi_2(\sigma)
    \end{pmatrix}$, where $*$ is defined by a non-zero class in $H^1(G_{\Q, \Sigma}, \chi_1\chi_2^{-1})$, unique up to scalar. Let $R_*$ be the universal deformation ring of $\rho_*$. By \cite[Proposition 2.2.1]{pan2022fontaine}, the universal pseudo-deformation ring of the pseudo-representation of $\chi_1+\chi_2$ is $\widehat{(R_\Q^{\ps})_{\kp_E}}$. By the universal property, we have a natural map $f_E: \widehat{(R_\Q^{\ps})_{\kp_E}} \to R_* $.
    
   \begin{lemma}\label{sur}
   	 1) The map $f_E$ is a surjection.
   	 
   	 2) The ring $R_*$ is a power series ring over $K$ of dimension $3$.
   \end{lemma}
    
    \begin{proof}
    	The first claim follows from Lemma \ref{eisenstein series} and \cite[Corollary 1.4.4 (2)]{kisin2009fontaine}. For the second one, by Proposition \ref{Gal coh}, we have $\dim_K(H^1(G_{\Q,\Sigma},\operatorname{ad}(\rho_*)))=3$. Then by the global characteristic formula (see \cite[Lemma 9.7]{kisin2003overconvergent}), we have $ H^2(G_{\Q,\Sigma},\operatorname{ad}(\rho_*))=0$. The claim follows from classical Galois deformation theory (see \cite[Proposition 1.5.1, page 30]{Berger_2013}).
    \end{proof}
    
    Note that by the first part of Theorem \ref{big R=T Q}, we have $\dim \widehat{(R_\Q^{\ps})_{\kp_E}}=3$. Choose a minimal prime $\mathfrak{Q}$ of $\widehat{(R_\Q^{\ps})_{\kp_E}} $ of dimension $3$. Note that the reducible locus of $\Spec \widehat{(R_\Q^{\ps})_{\kp_E}}$ is of dimension at most $\dim (K[[G_{\Q, \Sigma}^{\operatorname{ab}}(p)]]) \widehat{\otimes}_K (K[[G_{\Q, \Sigma}^{\operatorname{ab}}(p)]])=2 < 3$. Then $\mathfrak{Q}$ defines an irreducible pseudo-deformation of the pseudo-representation $\chi_1+\chi_2$. 
    
    Let $\lambda$ be a one-dimensional irreducible point containing $\mathfrak{Q}$. Let $\cO_\lambda$ be a DVR obtained by localizing the normalization of $\widehat{(R_\Q^{\ps})_{\kp_E}}/\lambda$, and let $K_\lambda$ be its fraction field. By Definition \ref{construction}, we get an absolutely irreducible representation $\rho_\lambda: G_{\Q, \Sigma} \to \GL_2(K_\lambda)$. Take the reduction $\bar{\rho_\lambda}$ of a stable lattice in $\rho_\lambda$. By \cite[Lemma 2.13 \& Corollary 2.14]{Skinner_1999}, $\bar{\rho_\lambda} $ is a non-split Galois representation with semisimplification $\chi_1 \oplus \chi_2$. By Lemma \ref{eisenstein series}, $\rho_\lambda$ is a lifting of $\rho_*$ up to isomorphism. Hence, we get the following commutative diagram by the universal property.
    \[\begin{tikzcd}
    	\widehat{(R_\Q^{\ps})_{\kp_E}}  \arrow[r, two heads] \arrow[d, two heads]& R_* \arrow[d] \\
    	\widehat{(R_\Q^{\ps})_{\kp_E}}/\lambda \arrow[r, hook]& K_\lambda.
    \end{tikzcd}\]
    Then we have $\operatorname{ker} f_E \subset \lambda$.
    
    Note that all such $\lambda$ are Zariski dense in the irreducible component defined by $\mathfrak{Q}$. Varying $\lambda$, we have $ \operatorname{ker} f_E \subset \cap\lambda = \mathfrak{Q}$. By Lemma \ref{sur}, $\operatorname{ker} f_E $ is actually a minimal prime of dimension $3$, and hence $\operatorname{ker} f_E=\mathfrak{Q}$. This implies that the ring $\widehat{(R_\Q^{\ps})_{\kp_E}}$ has a unique minimal prime of dimension $3$, which is exactly $ \operatorname{ker} f_E $. Now consider the natural map $ (R_\Q^{\ps})_{\kp_E} \to \widehat{(R_\Q^{\ps})_{\kp_E}}$, which is faithfully flat by \cite[Lemma 10.97.3]{stacks-project}. By the going-down property (\cite[Theorem 5.D]{matsumura80}), we know that there is a unique minimal prime of $ (R_\Q^{\ps})_{\kp_E}$ of dimension $3$. 
    
    In conclusion, we have shown the following proposition in \textbf{Case 1}.
    
    \begin{prop}\label{pE}
    	There is a modular prime $\kp_E$ in $C$ such that $C$ is the unique irreducible component of $\Spec R_\Q^\ps$ containing $\kp_E$. In particular, $C$ is of Krull dimension $4$.
    \end{prop}
     
     \begin{proof}
     	Note that by Lemma \ref{eisenstein series}, $\kp_E$ defines an Eisenstein series. Then the result follows from our previous discussions.
     \end{proof}
    
    \begin{rem}
    	Another possible way to prove Proposition \ref{pE} is to use the theory of GMA (Generalized Matrix Algebra). One can see \cite{Deo2022} and \cite{bc09} for more details.
    \end{rem}
    
    Combining \cite[Lemma 2.3.10]{Deo2023}, we know that the natural surjection $R_\Q^\ps \twoheadrightarrow R_\Q^\ps/\kp_E$ factors through $\mathbb{T}(N)_\km$. Let $I$ be the kernel of the composite map $ R_\Q^\ps \twoheadrightarrow \mathbb{T}(N)_\km \twoheadrightarrow (\mathbb{T}(N)_\km)^{\operatorname{red}}$. By Theorem \ref{lower bound} and the first part of Theorem \ref{big R=T Q}, the ideal $I$ is an intersection of some minimal primes of $R_\Q^\ps$. As $I \subset \kp_E$, by \cite[Lemma 10.15.1]{stacks-project}, $\kp_E$ is contained in some irreducible component of $\Spec \mathbb{T}(N)_\km$ of dimension $4$, which is $C$ by Proposition \ref{pE}. This shows that $C$ is an irreducible component of $\Spec \mathbb{T}(N)_\km$.

    \textbf{Case 2}: the image of $C_1^\ord$ is not contained in $\Spec R_\Q^{\ps,\operatorname{red}}$.
    
    Let $\kp^\ord$ be the generic point of $C_1^\ord$. By Definition \ref{construction}, we can define a Galois representation $\rho(\kp^\ord)$ satisfying $\rho(\kp^\ord)|_{G_{\Q_p}} \cong \begin{pmatrix}
    	\psi_{p,1} & *\\
    	~& \psi_{p,2}
    \end{pmatrix}$. Here $\psi_{p,1}$ is the stable-line character. If $\bar\chi|_{G_{\Q_p}}=\mathbf{1}$, it is the specified ordinary character on $R_1^{\ps,\ord}$; otherwise it may lift either $\mathbf{1}$ or $\bar\chi|_{G_{\Q_p}}$. Its ordinary identities hold on $C_1^\ord$, since they hold at its generic point.
    
    \textbf{Case 2.1}: $\bar{\chi}|_{G_{\Q_p}} \ne \mathbf{1}$ and $\psi_{p,1}$ is a character lifting $\mathbf{1}$.
    
    In this case, by Theorem \ref{ord1} and the going-up property, we may choose a one-dimensional prime \(\mathfrak p_f\) of $R_\Q^{\ps}$ containing \(\mathfrak p^{\mathrm{ord}}\) such that \((\psi_{p,1}\bmod \mathfrak p_f)|_{I_p}=\varepsilon^{a(p-1)}\), where \(a\) is a positive integer satisfying \(a>s/(2(p-1))\). We may further assume that the representation associated to \(\mathfrak p_f\) is irreducible, since the reducible locus in \(C_1^{\mathrm{ord}}\) has Krull dimension at most \(1\).
    
    Let \(\widetilde{\omega}\) denote the Teichmüller lift of the mod \(p\) cyclotomic character, and put
    \[
    \eta_a:=\varepsilon^{a(p-1)+1-s}\widetilde{\omega}^{s_0-1}.
    \]
    Since \(s\equiv s_0\pmod {p-1}\), we know that \(\eta_a\) is a character lifting $\mathbf{1}$. In particular, \(\eta_a\) factors through the cyclotomic \(\mathbb Z_p\)-extension of \(\mathbb Q\). Define \(\rho_f:=\rho(\mathfrak p_f)\otimes \eta_a\). Then
    \[
    \rho_f|_{G_{\mathbb Q_p}}\simeq \begin{pmatrix} \eta_1\varepsilon^{2a(p-1)+1-s}\widetilde{\omega}^{s_0-1} & * \\ 0 & \eta_2\widetilde{\omega}^{s_0-1} \end{pmatrix},
    \]
    where \(\eta_1\) and \(\eta_2\) are unramified characters of $G_{\Q_p}$. By our choice of \(a\), the two diagonal characters have distinct Hodge--Tate weights, and $2a(p-1)+1-s>1$. Hence, $\rho_f$ is irreducible and ordinary regular de Rham at $p$.
    
    By Theorem \ref{ord1}, \(\rho_f\) arises from a classical cuspidal eigenform of tame level dividing \(N\), which is a finite-order twist of $p$-power conductor of a $p$-ordinary form. Arguing as in the proof of \cite[Theorem 4.2.2]{Deo2023}, it defines a \(p\)-adic modular form of tame level dividing \(N\). Since \(\mathbb T(N)_{\mathfrak m}\) acts faithfully on the corresponding residual summand of these \(p\)-adic modular forms, \(\rho_f\) gives a homomorphism \(\mathbb T(N)_{\mathfrak m}\to \overline{\mathbb Q}_p\).
    
    We claim that this point lies on the component \(C\). Note that \(\eta_a\) is a character of $\Gal(\Q_\infty/\Q)$ lifting $\mathbf{1}$, where $\Q_\infty $ is the cyclotomic $\Z_p$-extension of $\Q$. Then we have $\bar{\rho}_f^{\operatorname{ss}}=1\oplus\bar{\chi}$. Recall that $R_\mathbf{1}:=\cO[[G_{\Q, \Sigma}^{\operatorname{ab}}(p)]]$ is of the form $\cO[[X, X_1, \cdots, X_t]]/((1+X_1)^{p^{r_1}}-1, \cdots, (1+X_t)^{p^{r_t}}-1)$, where the variable $X$ defines the cyclotomic line. Then our claim follows from the second part of Corollary \ref{at least 4}.
    
    
    We now record an adjoint Bloch-Kato Selmer group input that will also be used in Section \ref{F ne Q}. Let $L$ be a totally real field, let $S$ be a finite set of finite places of $L$ containing all those above $p$, and let $V$ be a finite-dimensional $K$-vector space equipped with a continuous $G_{L,S}$-action, assumed de Rham at each $v\mid p$. For $v\mid p$, set
    \[
    \begin{aligned}
    H^1_f(L_v,V)&:=\ker\left(H^1(L_v,V)\longrightarrow H^1(L_v,V\otimes_{\Q_p}B_{\operatorname{crys}})\right),\\
    H^1_g(L_v,V)&:=\ker\left(H^1(L_v,V)\longrightarrow H^1(L_v,V\otimes_{\Q_p}B_{\operatorname{dR}})\right).
    \end{aligned}
    \]
    For $v\nmid p$, set $H^1_f(L_v,V):=H^1_{\operatorname{ur}}(L_v,V)$ and $H^1_g(L_v,V):=H^1(L_v,V)$. We then define $H^1_f(L,V)$ and $H^1_{g,S}(L,V)$ by imposing the corresponding local conditions at every $v\in S$. Recall also that a Weil--Deligne representation is \textit{generic} if it admits no non-zero morphism to its Tate twist. By \cite[p.~1923]{NewtonThorne}, one has $H^1_f(L_v,\End(V))=H^1_g(L_v,\End(V))$ if the Frobenius-semisimplification of the Weil-Deligne representation associated to $V|_{G_{L_v}}$ is generic. Also see \cite[Section 1]{Allen_2016} for more details.
    
    \begin{thm}[Newton--Thorne]\label{thm:Newton-Thorne-5.4}
    	Let $L$ be a totally real number field, and let $\pi$ be a regular algebraic cuspidal automorphic representation of $\GL_2(\mathbb A_L)$. Let $\iota:\overline{\Q}_p\xrightarrow{\sim}\mathbb C$ be an isomorphism, and let $r_{\pi,\iota}:G_L\to\GL_2(\overline{\Q}_p)$ be the associated Galois representation. Suppose that one of the following conditions holds:
    	\begin{enumerate}
    		\item $\pi$ does not have CM;
    		\item $\pi$ has CM by a quadratic extension $L'/L$, and $L'\not\subset L(\mu_{p^\infty})$.
    	\end{enumerate}
    	Then $H^1_f(L,\operatorname{ad}r_{\pi,\iota})=0$.
    \end{thm}
    
    \begin{proof}
    	See \cite[Theorem 5.4]{NewtonThorne}.
    \end{proof}
    
    \begin{cor}\label{cor:ordinary-adjoint-vanishing}
	    	Assume that $p$ splits completely in the totally real field $L$, and let $S$ be a finite set of finite places of $L$ containing all places above $p$. Let $\rho:G_{L,S}\to\GL_2(K)$ be an irreducible representation which, up to a character twist, is associated to a regular algebraic cuspidal automorphic representation $\pi$ of $\GL_2(\mathbb A_L)$. If $\rho|_{G_{L_v}}$ is ordinary for every $v\mid p$, then
    	\(
    	H^1_f(L,\operatorname{ad}\rho)=H^1_{g,S}(L,\operatorname{ad}\rho)=0
    	\).
    	In particular, $H^1_{g,S}(L,\operatorname{ad}^0\rho)=0$.
    \end{cor}
    
    \begin{proof}
    	By the global Whittaker expansion for cuspidal automorphic forms, $\pi$ is globally generic, and hence each local component $\pi_v$ is generic; see \cite[Theorem 11.3.6 and the discussion on p.~297]{Getz_2024}. Note that every cuspidal automorphic representation of $\GL_2$ is essentially self-dual, so $\pi$ is polarizable. Local-global compatibility therefore gives $\iota\operatorname{WD}(r_{\pi,\iota}|_{G_{L_v}})^{F\text{-ss}}\cong\operatorname{rec}_{L_v}^T(\pi_v)$ for every finite place $v$ by \cite[Theorem 2.1.2(3)]{Allen_2016}. It follows from \cite[Lemma 1.1.3]{Allen_2016} and \cite[p.~1923]{NewtonThorne} that $H^1_f(L_v,\operatorname{ad}\rho)=H^1_g(L_v,\operatorname{ad}\rho)$ for every $v\in S$. Thus $H^1_f(L,\operatorname{ad}\rho)=H^1_{g,S}(L,\operatorname{ad}\rho)$.
    	
    	It remains to exclude the exceptional CM case in Theorem \ref{thm:Newton-Thorne-5.4}. Suppose that $\pi$ has CM by $L'/L$ with $L'\subset L(\mu_{p^\infty})$. Since $p$ splits completely in $L$, we have $L\cap\Q(\mu_{p^\infty})=\Q$, and the cyclotomic extension is totally ramified at every $v\mid p$. Thus $L'/L$ is ramified at every such $v$. Writing $r_{\pi,\iota}\cong\operatorname{Ind}_{G_{L'}}^{G_L}\theta$, its restriction at $v$ is $\operatorname{Ind}_{G_{L'_w}}^{G_{L_v}}\theta_w$ for the unique $w\mid v$. This induction is reducible by ordinarity, so $\theta_w$ extends to $G_{L_v}$. The two extensions differ by a quadratic character and hence have the same Hodge--Tate weights, contradicting regularity. Theorem \ref{thm:Newton-Thorne-5.4} now gives the vanishing. Finally, $\operatorname{ad}^0\rho$ is a direct summand of $\operatorname{ad}\rho$, and character twists do not change the adjoint representation.
    \end{proof}
    
    \begin{rem}\label{rem:Kisin-Newton-Thorne}
	    	For $L=\Q$, Kisin proved the vanishing of the geometric adjoint Bloch--Kato Selmer group for non-CM modular forms under, in the residually reducible case, the hypothesis that the prime-to-$p$ parts of the two residual diagonal characters are distinct; see \cite[Corollary 6.7(2) and Corollary 7.12]{Kisin_2004}. In the present setting, this amounts to excluding the case in which their ratio is a power of $\omega$. Kisin then used this vanishing to obtain the expected smoothness and dimension of the unrestricted deformation ring; see \cite[Theorem 8.2]{Kisin_2004}. Corollary \ref{cor:ordinary-adjoint-vanishing} removes this residual restriction at the ordinary regular automorphic points used here, and also treats the CM case: Newton--Thorne gives the vanishing of the Bloch--Kato Selmer group, local-global compatibility and genericity identify it with the geometric Selmer group, and ordinarity excludes the exceptional CM case. We will nevertheless use the cohomological argument of \cite[Theorem 8.2]{Kisin_2004}, in the form also explained in \cite[Lemma 1.3.5]{Allen_2016}, to pass from this vanishing to the expected smoothness and dimension of the unrestricted deformation ring.
    \end{rem}
    
    \begin{lem}\label{iso}
    	 We have an isomorphism $\widehat{(R_\Q^{\ps})_{\kp_f}} \cong R_f$, where $R_f $ is the universal deformation ring of $\rho_{f}$. Both of them are power series rings over $K$ of dimension $3$.
    \end{lem}
    
    \begin{proof}
    	Put $W=\operatorname{ad}\rho_f$. Corollary \ref{cor:ordinary-adjoint-vanishing} gives $H^1_{g,\Sigma}(\Q,W)=0$. Thus localization embeds $H^1(G_{\Q,\Sigma},W)$ into $H^1(\Q_p,W)/H^1_g(\Q_p,W)$, which has dimension $3$ by regularity and local genericity. The global Euler characteristic formula gives $h^1-h^2=3$, hence $h^1=3$ and $h^2=0$. This is the cohomological argument of \cite[Theorem 8.2]{Kisin_2004}; see also \cite[Lemma 1.3.5]{Allen_2016}. Thus $R_f\cong K[[x_1,x_2,x_3]]$, and the comparison isomorphism follows from \cite[Corollary 2.2.2]{pan2022fontaine}.
    \end{proof}
    
    \begin{prop}\label{pf}
    	There is a modular point $\kp_f$ in $C$ such that $C$ is the unique irreducible component of $\Spec R_\Q^{\ps}$ containing $\kp_f$. In particular, $C$ is of Krull dimension $4$.
    \end{prop}
    
    \begin{proof}
    	By Lemma \ref{iso}, we know that $\kp_f$ is actually a regular point. Then the result holds.
    \end{proof}
    
    Using an argument similar to that in \textbf{Case 1}, we conclude that $C$ is an irreducible component of $\Spec \mathbb{T}(N)_\km$ in this case.
    
    \textbf{Case 2.2}: $\bar{\chi}|_{G_{\Q_p}} \ne \mathbf{1}$ and $\psi_{p,1}$ is a character lifting $\bar{\chi}$.
    
    In this case, by Theorem \ref{ord1} and the going-up property, we may choose a one-dimensional prime \(\mathfrak p_f'\supset \mathfrak p^{\mathrm{ord}}\) such that \((\psi_{p,2}\bmod \mathfrak p_f')|_{I_p}=\mathbf{1}\). Hence
    \[
    \rho(\mathfrak p_f')|_{G_{\mathbb Q_p}}
    \simeq
    \begin{pmatrix}
    	\eta_2\varepsilon^{s-1} & * \\
    	0 & \eta_1
    \end{pmatrix},
    \]
    where \(\eta_1\) is a character lifting \(1\), and \(\eta_1\eta_2=\Psi\Phi\). Thus both \(\eta_1\) and \(\eta_2\) are crystalline characters, and \(s>2\). If \(\rho(\mathfrak p_f')\) is reducible, then it defines an Eisenstein series of level dividing \(N\), and we can use the same argument as in \textbf{Case 1}. If \(\rho(\mathfrak p_f')\) is absolutely irreducible, then by the ordinary Fontaine--Mazur theorem (Theorem \ref{ord1}) and the classical local-global compatibility result (see \cite[Lemma 5.1.2]{Deo2023}), \(\rho(\mathfrak p_f')\) comes from a classical eigenform of level dividing \(N\). Moreover, \(\rho(\mathfrak p_f')\) gives a homomorphism \(\mathbb T(N)_{\mathfrak m}\to \overline{\mathbb Q}_p\). We can then conclude by the same argument as in \textbf{Case 2.1}.
    
    \textbf{Case 2.3}: \(\bar\chi|_{G_{\mathbb Q_p}}=\mathbf{1}\).
    
    The reducible locus in $C_1^\ord$ has dimension at most $1$, so its image under the finite weight map contains only finitely many arithmetic primes of weight $\varepsilon^{a(p-1)}$. By surjectivity, we may choose $a>s/(2(p-1))$ and a one-dimensional point $\mathfrak p_f''$ on $C_1^\ord$ such that $(\psi_{p,1}\bmod\mathfrak p_f'')|_{I_p}=\varepsilon^{a(p-1)}$ and $\rho(\mathfrak p_f'')$ is absolutely irreducible.

    Put $\eta_a:=\varepsilon^{a(p-1)+1-s}$ and $\rho_f:=\rho(\mathfrak p_f'')\otimes\eta_a$. As $\psi_{p,1}$ is the specified stable-line character, $\rho_f|_{G_{\Q_p}}$ has an unramified quotient and stable-line character $\eta_1\varepsilon^{2a(p-1)+1-s}$ with $\eta_1$ unramified. The inequality $2a(p-1)+1-s>1$ makes $\rho_f$ ordinary and regular de Rham. Since $s\equiv1\pmod{p-1}$, the character $\eta_a$ lifts $\mathbf{1}$ and factors through the cyclotomic $\Z_p$-extension. Thus $\rho_f$ still defines a point on $C$ by Corollary \ref{at least 4} (2).

    Theorem \ref{ord2}, applied with determinant $\det\rho_f$, and the regularity argument of \textbf{Case 2.1} now provide a regular modular point on $C$ and complete the proof.
    
    In conclusion, every irreducible component of $\Spec R_\Q^\ps$ is also an irreducible component of $\Spec \mathbb{T}(N)_\km$. The second part of Theorem \ref{big R=T Q} now follows from Theorem \ref{lower bound}.
    
    \begin{rem}
    	If the residual pseudo-representation $\bar{T}= \bar{\chi_1}+\bar{\chi_2}$, where $\bar{\chi_1}\bar{\chi_2}^{-1}$ satisfies the conditions in Theorem \ref{big R=T Q}, then we can obtain similar results by using almost the same proof as in \textbf{Case 1} and \textbf{Case 2.1}.
    \end{rem}
    
    \subsection{Irregular Fontaine-Mazur conjecture in the residually reducible case}\label{irr fm}
    In this subsection, we discuss the Fontaine-Mazur conjecture in the irregular case under the residually reducible setup and give an analogous result to \cite[Theorem 1.0.5]{Pan_2022}.
    
    \begin{thm}[Irregular Fontaine--Mazur in the Eisenstein case]
    	\label{thm:irregular-FM-Eisenstein}
    	Let \(p>2\) be a prime, and let
    	\(
    	\rho:G_{\mathbb Q}\to \mathrm{GL}_2(\overline{\Q_p})
    	\)
    	be a continuous irreducible representation. Assume that:
    	
    	\begin{enumerate}
    		\item \(\rho\) is unramified outside finitely many primes;
    		
    		\item \(\rho\) is odd, i.e.
    		\(
    		\det\rho(c)=-1
    		\)
    		for a complex conjugation \(c\in G_{\mathbb Q}\);
    		
    		\item \(\rho|_{G_{\mathbb Q_p}}\) is Hodge--Tate of weights
    		\(
    		\{0,0\};
    		\)
    		
	    	\item $\bar{\rho}^{\operatorname{ss}} = \bar{\chi_1} \oplus \bar{\chi_2}$, where $\bar\chi_1\ne\bar\chi_2$; 
	    	
	    	\item one of the following conditions holds: \begin{itemize}
	    		\item $(\bar{\chi_1}/\bar{\chi_2})|_{G_{\Q_p}} \ne \omega^{\pm 1}$,
	    	    \item $\rho|_{G_{\Q_p}}$ is reducible.
	    	\end{itemize}
    	\end{enumerate}
    	
    	Then \(\rho\) is modular in the sense that there exists a classical cuspidal
    	eigenform \(f\) of weight \(1\) such that
    	\(
    	\rho\simeq \rho_f.
    	\)
    	In particular, \(\rho\) has finite image.
	    If \(p\geq 5\), hypothesis~\textup{(5)} is automatic; hence in this range the conclusion holds under hypotheses~\textup{(1)--(4)}.
    \end{thm}
    
    \begin{rem}
    	 In the residually irreducible setting, an
    		irregular Fontaine--Mazur theorem for two-dimensional representations of
    		\(G_{\mathbb Q}\) is known under the assumptions that
    		\(
    		\bar\rho|_{G_{\mathbb Q(\mu_p)}} 
			\) is irreducible (Taylor--Wiles hypothesis)
    		and that
    		\(
    		(\bar\rho|_{G_{\mathbb Q_p}})^{\mathrm{ss}}
    		\)
    		is either irreducible or a sum \(\eta_1\oplus \eta_2\) with
			\(\eta_1/\eta_2\neq 1,\omega^{\pm 1}\). See \cite[Theorem 1.0.5 \& Remark 1.0.6]{Pan_2022}. These are precisely the kinds of hypotheses
    		under which the known pro-modularity results can be invoked. See \cite[Theorem 1.2.3]{emerton11} or \cite[Theorem 2.1]{Nakamura2023}.
    		
	    		
	    		Berger--Klosin \cite{BK23} prove related $R=\mathbb T$ theorems for certain ordinary weight-one deformation problems under explicit Eisenstein congruence, class-group, and local hypotheses. For a non-quadratic residual character, their Hecke algebra is a weight-one specialization of a Hida Hecke algebra and generally parametrizes non-classical $p$-adic weight-one forms. For a quadratic residual character, after imposing a split condition at $p$ and further hypotheses, they identify a deformation quotient with a Hecke algebra of classical CM weight-one forms. These results yield modularity in several more specialized residually reducible settings, but do not imply Theorem~\ref{thm:irregular-FM-Eisenstein} in the present generality.
	    \end{rem}
    
    We first recall Pan's classicality theorem.
    
    Following \cite[Definition 6.1.2]{Pan_2022}, a representation $\rho:G_\Q\to\GL_2(\overline{\Q}_p)$ is \textit{pro-modular} if $\rho\cong\rho_\lambda$ for some $\lambda:\mathbb T(K^p)\to\overline{\Q}_p$ and tame level $K^p$, with $\mathbb T(K^p)$ as in Section \ref{sec for modular curve}.

    \begin{thm}[Classicality]
    	\label{thm:Pan-classicality}
    	Let \(p>2\) be a prime, and let
    	\(
    	\rho:G_{\mathbb Q}\to \mathrm{GL}_2(\overline{\Q_p})
    	\)
    	be a continuous irreducible representation.
    	Assume that:
    	\begin{enumerate}
    		\item \(\rho \cong \rho_\lambda\) is irreducible and pro-modular;
    		\item \(\rho_\lambda|_{G_{\mathbb Q_p}}\) is Hodge--Tate of
    		weights \(\{0,0\}\).
    	\end{enumerate}
    	Then $\rho$ is classical; that is,
    	\(M_1(K^p)[\mathfrak p_\lambda]\neq 0\) for some tame level $K^p$, where $M_1(K^p)$ is the space of classical weight $1$ modular forms of level $K^p$.
    
    \end{thm}
    
    \begin{proof}
    	See \cite[Theorem 6.4.7]{Pan_2022}.
    \end{proof}

    By \cite[Section 6.1.1]{Pan_2022}, $\rho_\lambda(\Frob_l)$ has characteristic polynomial $X^2-l^{-1}\lambda(T_l)X+l^{-1}\lambda(S_l)$, so the arithmetic normalization in Section \ref{sec for emerton} is $\rho_\lambda^\vee(1)$. Poincar\'e duality induces the Hecke involution $\jmath(T_l)=S_l^{-1}T_l$, $\jmath(S_l)=S_l^{-1}$ on the cuspidal spectrum, with $\rho_{\lambda\circ\jmath}\cong\rho_\lambda^\vee(1)$. We use this identification, also on its closure and in Appendix \ref{A}, to match eigensystems and maximal ideals without twisting the given Galois representation.

	    \begin{lem}\label{lem:weight-zero-local-induction}
	    Let \(p>2\), and let \(V:G_{\Q_p}\to\GL_2(\overline{\Q}_p)\) be a continuous absolutely irreducible representation which is Hodge--Tate of weights \(\{0,0\}\). Then there are a quadratic extension \(L/\Q_p\) and a continuous character \(\theta:G_L\to\overline{\Q}_p^\times\) such that
	    \(
	    V\cong\operatorname{Ind}_{G_L}^{G_{\Q_p}}\theta
	    \).
	    Suppose moreover that \(\bar V^{\operatorname{ss}}=\bar\alpha_1\oplus\bar\alpha_2\) is reducible. Then \(\bar\alpha_1/\bar\alpha_2\) is the quadratic character associated with \(L/\Q_p\). In particular, if \(\bar\alpha_1/\bar\alpha_2=\omega^{\pm1}\), then \(p=3\) and \(L=\Q_3(\mu_3)\).
	    \end{lem}

	    \begin{proof}
	    By \cite[Corollary to Theorem 11]{sen80}, \(V\) is potentially unramified. Thus \(V|_{W_{\Q_p}}\) is smooth and remains absolutely irreducible, since \(W_{\Q_p}\) is dense in \(G_{\Q_p}\). Transporting coefficients through an isomorphism $\overline{\Q}_p\cong\mathbb C$, we may apply \cite[34.1 Theorem]{Bushnell_2006}: as $p\ne2$, there are a quadratic extension \(L/\Q_p\) and a character \(\theta_W:W_L\to\overline{\Q}_p^\times\) such that
	    \(
	    V|_{W_{\Q_p}}\cong\operatorname{Ind}_{W_L}^{W_{\Q_p}}\theta_W
	    \).
	    Under this isomorphism, the \(\theta_W\)-isotypic line is stable under \(W_L\). Since \(W_L\) is dense in \(G_L\), this line is stable under \(G_L\), and the action on it defines a continuous character \(\theta:G_L\to\overline{\Q}_p^\times\) extending \(\theta_W\). By Frobenius reciprocity, we have \( \Hom_{G_{\Q_p}}(\operatorname{Ind}_{G_L}^{G_{\Q_p}}\theta, V) \cong \Hom_{G_L}(\theta, V|_{G_L})\), which implies $V \cong \operatorname{Ind}_{G_L}^{G_{\Q_p}}\theta$ since $V$ is absolutely irreducible.

	    It remains to analyze the reduction. Let \(\sigma\) be the non-trivial element of \(\operatorname{Gal}(L/\Q_p)\). Computing the semisimplified reduction using a lattice induced from a stable lattice for \(\theta\), we obtain
	    \(
	    \bar V^{\operatorname{ss}}\cong\left(\operatorname{Ind}_{G_L}^{G_{\Q_p}}\bar\theta\right)^{\operatorname{ss}}
	    \).
	    Since $[G_{\Q_p}:G_L]=2$ is prime to $p$, reducibility of \(\bar V^{\operatorname{ss}}\) implies \(\bar\theta^\sigma=\bar\theta\). Over an algebraic closure of the residue field, \(\bar\theta\) extends to a character \(\bar\eta\) of \(G_{\Q_p}\), and $\bar V^{\operatorname{ss}}\cong\bar\eta\oplus\bar\eta\bar\eta_{L/\Q_p}$, where $\bar\eta_{L/\Q_p}$ is the quadratic character of $L/\Q_p$. Finally, this character has order two, whereas \(\omega\) has order \(p-1\). Equality with \(\omega^{\pm1}\) therefore forces \(p=3\) and \(L=\Q_3(\mu_3)\).
	    \end{proof}
        
	       After twisting by a finite-order character, we may assume $\bar{\rho}^{\operatorname{ss}} = 1 \oplus \bar{\chi}$. Notice that modularity is invariant under finite-order twists. Let $\Sigma$ be a finite set of rational primes containing $p$ and those at which $\rho$ is ramified. Let $R_\Q^\ps$ be the universal pseudo-deformation ring of the pseudo-representation $1+\bar{\chi}$ (for the profinite group $G_{\Q, \Sigma}$).

       Let \(\bar V\) be the semisimple residual representation with trace \(1+\bar{\chi}\).
       For each prime \(\ell\in \Sigma\), \(\ell\neq p\), put
       \(
       s_\ell:=\operatorname{Sw}_\ell(\bar V)
       \), where $ \operatorname{Sw}_\ell(\bar V)$ is the corresponding Swan conductor.
       We choose an integer
       \(
       n_\ell:=2+s_\ell
       \)
       and take an open compact subgroup
       \(
       K_\ell\subset \mathrm{GL}_2(\mathbb Z_\ell)
       \)
       contained in \(K_1(\ell^{n_\ell})\). Indeed, for every $p$-adic lift $V$ of $\bar V$, the wild part of the Artin conductor is unchanged under reduction and the tame contribution is at most $2$; hence $a_\ell(V)\le \operatorname{Sw}_\ell(\bar V)+2=n_\ell$.
       Here for \(n\geq 1\), we write
      \[
      K_1(\ell^n)
      :=
      \left\{
      \begin{pmatrix}
      	a & b\\
      	c & d
      \end{pmatrix}
      \in \mathrm{GL}_2(\mathbb Z_\ell)
      \ \middle|\
      c\equiv 0 \pmod{\ell^n},\quad d\equiv 1 \pmod{\ell^n}
      \right\}.
      \] For \(\ell\notin \Sigma\cup\{p\}\), we put
      \(
      K_\ell=\mathrm{GL}_2(\mathbb Z_\ell)
      \). Define $K^p:= \prod_{l \ne p} K_l$ as a tame level.
      
      Let $\mathbb{T}(K^p)$ be the big Hecke algebra of tame level $K^p$. By the proof of Corollary \ref{lower}, we know that the pseudo-representation $1+\bar{\chi}$ defines a maximal ideal $\km$ of $ \mathbb{T}(K^p) \otimes_{\Z_p} \cO$. Write $\mathbb{T}_\km := (\mathbb{T}(K^p) \otimes_{\Z_p} \cO)_\km$, and we get a natural surjection $R_\Q^\ps \twoheadrightarrow \mathbb{T}_\km$.
      
      We first assume $(\bar{\chi_1}/\bar{\chi_2})|_{G_{\Q_p}} \ne \omega^{\pm 1}$. Note that the reducible locus of $\Spec R_\Q^\ps$ is of dimension at most $3$. Then Theorem \ref{big R=T Q} implies that cuspidal modular points are Zariski dense in $\Spec R_\Q^\ps$. For such a point $\kp \in \Spec R_\Q^\ps$, by our choice of the tame level $K^p$ and the classical local--global compatibility result at $l\ne p$, we know that the natural surjection $R_\Q^\ps \twoheadrightarrow  R_\Q^\ps /\kp$ factors through $\mathbb{T}_\km$. Moreover, the Zariski density implies that the kernel of the surjection $R_\Q^\ps \twoheadrightarrow \mathbb{T}_\km$ is nilpotent. Equivalently,  $\Spec \mathbb{T}_\km$ and $\Spec R_\Q^\ps$ have the same underlying topological space.
      
      Let $\kp_\rho$ be the prime of $\Spec R_\Q^\ps$ corresponding to the pseudo-representation $\operatorname{tr} (\rho)$. Viewing it as a point of $\Spec \mathbb{T}_\km$, we get natural maps $\mathbb{T}(K^p) \to \mathbb{T}(K^p) \otimes_{\Z_p} \cO \to \mathbb{T}_\km \twoheadrightarrow \mathbb{T}_\km/\kp_\rho$. In particular, $\rho$ is pro-modular. Then by Theorem \ref{thm:Pan-classicality}, we conclude our result.
      
	       It remains to consider the case in which $\rho|_{G_{\Q_p}}$ is reducible and $(\bar\chi_1/\bar\chi_2)|_{G_{\Q_p}}=\omega^{\pm1}$; in particular, $\bar\chi|_{G_{\Q_p}} \ne \mathbf{1}$. By Theorem \ref{thm:Pan-classicality}, we only need to show that $\rho$ is pro-modular. Applying the construction in \cite[Section 2.5]{Zhang:2024aa} to the fixed-determinant pseudo-deformation point defined by $\operatorname{tr}\rho$, and using \cite[Corollary 2.14]{Skinner_1999}, after enlarging the coefficient field if necessary we may choose a stable lattice $\Lambda$ whose reduction is a non-split representation of the form
	       \(\bar\rho\cong\begin{pmatrix}1&*\\0&\bar\chi\end{pmatrix}
	       \).
	       Choose a $G_{\Q_p}$-stable line $L_p$ in the representation space of $\rho|_{G_{\Q_p}}$. Then $L_p\cap\Lambda$ is a saturated rank-one submodule of $\Lambda$, hence a direct summand, and its reduction $\bar L_p$ is a $G_{\Q_p}$-stable line in $\bar\rho|_{G_{\Q_p}}$. Since $\bar\chi|_{G_{\Q_p}}\ne\mathbf{1}$, the character $\bar\psi_{p,1}$ of $G_{\Q_p}$ on $\bar L_p$ is either $\mathbf{1}$ or $\bar\chi|_{G_{\Q_p}}$. Let $\psi_{\rho,1}$ be the character of $G_{\Q_p}$ on $L_p$. We let $R^\ord$ denote the ordinary deformation ring of $(\bar\rho,\bar L_p)$ with fixed determinant $\det\rho$.

	       Let $R_{\bar\psi_{p,1}}^{\ps,\ord}$ parametrize pairs $(T,\psi_{p,1})$, where $T$ lifts $\operatorname{tr}\bar\rho$ with determinant $\det\rho$, the character $\psi_{p,1}$ lifts $\bar\psi_{p,1}$, and $T|_{G_{\Q_p}}=\psi_{p,1}+\psi_{p,2}$ with $\psi_{p,2}=(\det\rho)|_{G_{\Q_p}}\psi_{p,1}^{-1}$ and $T$ $\psi_{p,1}$-ordinary in the sense of \cite[Definition 5.3.1]{pan2022fontaine}. Let $R^{\ps,\ord}$ be defined as in Section \ref{F=Q}, now with determinant $\det\rho$. Since the local residual characters are distinct, $\psi_{p,1}$ is already distinguished on this ring. Thus $R_{\bar\psi_{p,1}}^{\ps,\ord}$ is its quotient by the specified ordinary relations. By \cite[Lemma 5.3.2]{pan2022fontaine}, $(\rho,L_p)$ defines points $\kp_\rho\in\Spec R^\ord$ and $\kp_\rho^\ps\in\Spec R_{\bar\psi_{p,1}}^{\ps,\ord}$.

	       There is an irreducible component of $\Spec R^\ord$ through $\kp_\rho$ of dimension at least $2$. Indeed, with the chosen non-split lattice and ordinary line, $R^\ord$ is an instance of $R_b^{\operatorname{univ}}$ in \cite[Section 2.5]{Zhang:2024aa}. Moreover, $H^0(G_{\Q,\Sigma},\ad^0\rho(1))=0$: a non-zero element would induce an isomorphism $\rho\cong\rho(1)$, contradicting $\varepsilon^2\ne\mathbf{1}$. Thus \cite[Corollary 2.5.7]{Zhang:2024aa} applies at the characteristic-zero point $\kp_\rho$. The specified ordinary conditions identify the same subfunctor of deformations of $\rho$: the distinct local residual characters determine the chosen stable direct summand, and the trace relations impose its stability. Consequently $\widehat{(R_{\bar\psi_{p,1}}^{\ps,\ord})_{\kp_\rho^\ps}}\cong\widehat{(R^\ord)_{\kp_\rho}}$ by \cite[Corollary 2.2.2 and the proof of Lemma 7.4.7]{pan2022fontaine}. Hence some component $C^\ord$ through $\kp_\rho^\ps$ has dimension at least $2$.
       
	       Let $\kp^\ord$ be the generic point of $C^\ord$, and choose a minimal prime $\kq_0\subset\kp^\ord$ of $R_\Q^\ps$, defining a component $C_0$. Since $\kp^\ord\subset\kp_\rho^\ps$ and $\kp_\rho^\ps$ is of characteristic $0$, so is $\kq_0$. By Theorem \ref{ord1}, $C^\ord$ is finite over $\Lambda_\Q$; its dimension is therefore exactly $2$, and this finite map is surjective by equality of dimensions. The same holds for the weight map defined by the specified stable character $\psi_{p,1}$, since exchanging the characters changes the weight coordinate by an automorphism. The reducible locus is a proper closed subset, as $\rho$ is irreducible. We may therefore choose an arithmetic point $\kp_f'\in C^\ord$ outside it, with $\psi_{p,1}\bmod\kp_f'$ of Hodge--Tate weight $-b$ for some sufficiently large $b>0$, and with trivial finite-order part on $1+p\Z_p$. Since $\det\rho$ has finite order by global class field theory, the quotient character has weight $b$. These ordered weights ensure that $\rho(\kp_f')$ is regular de Rham, as in \textbf{Cases 2.1 and 2.2}.

	       Put $\eta_b=\varepsilon^b\widetilde\omega^{-b}$, with $\widetilde\omega$ the Teichm\"uller lift of $\omega$. Then $\eta_b$ lifts $\mathbf{1}$ and varies only along the cyclotomic line. The twist $\rho_f=\rho(\kp_f')\otimes\eta_b$ has stable-line weight $-2b$, quotient weight $0$, and determinant a finite-order character times $\varepsilon^{2b}$. By Theorem \ref{ord1} and local-global compatibility, it comes from a classical cuspidal eigenform, allowing finite level at $p$. It is a finite-order twist of $p$-power conductor of an ordinary form, so its eigenvalues define a point $\kp_f$ of $\T(N)_{\km_N}$, with $N$ chosen as in Section \ref{F=Q} and $\km_N$ the corresponding residual maximal ideal; see \cite[Section 2.4 and the proof of Corollary 2.28]{Emerton2011}. This point still lies on $C_0$: Corollary \ref{at least 4}(2) applies to $\kq_0$, and a cyclotomic twist preserves both minimal primes in its tensor-product description.

	       The proof of Lemma \ref{iso} applies to $\rho_f$, so $\widehat{(R_\Q^\ps)_{\kp_f}}$ is a power series ring of dimension $3$. Thus $\kq_0$ is the unique minimal prime below $\kp_f$, and $\dim R_\Q^\ps/\kq_0=4$. Choose a minimal prime of $\T(N)_{\km_N}$ below $\kp_f$, and denote its contraction to $R_\Q^\ps$ by $\kq$. Then $\kq_0\subset\kq$, while Theorem \ref{lower bound} gives $\dim R_\Q^\ps/\kq\ge4$. Hence $\kq=\kq_0$, and the cuspidal modular points are Zariski dense in $C_0$ by Theorem \ref{lower bound}. By the choice of $K^p$ and the same density argument as above, $C_0$ is contained in $\Spec\mathbb{T}_\km$. Finally, $\kq_0\subset\kp^\ord\subset\kp_\rho^\ps$ shows that $\rho$ is pro-modular, as required.
       
       \begin{rem}\label{rem:irregular-FM-Eis-p3}
		       	1)
		       	By Lemma~\ref{lem:weight-zero-local-induction}, the only remaining case not covered by Theorem~\ref{thm:irregular-FM-Eisenstein} can occur when $p=3$: it is the case in which $\rho|_{G_{\Q_3}}$ is induced from a character of $G_{\Q_3(\mu_3)}$ and $(\bar\chi_1/\bar\chi_2)|_{G_{\Q_3}}=\omega$.
       	
       	2) Although our pro-modularity result does not use $p$-adic Langlands correspondence, the proof of Theorem \ref{thm:Pan-classicality} (hence Theorem \ref{thm:irregular-FM-Eisenstein}) still requires Pašk\=unas theory. See \cite[Section 6]{Pan_2022} for more details.
       	
       	3) The conclusion of Theorem~\ref{thm:irregular-FM-Eisenstein} is subject to a
       	strong rigidity phenomenon for classical modular forms of weight one. By the main theorem of \cite{DeligneSerre1974}, \(\rho_f\) has finite image. If \(\bar\rho_f^{\operatorname{ss}}\) is reducible with distinct constituents,
       	then, since \(p>2\), the projective image of \(\rho_f\) is dihedral. Consequently, \(\rho_f\) is induced from a finite-order character of an
       	imaginary quadratic field, and $\bar{\chi}:=\bar\chi_1/\bar\chi_2$ is an odd quadratic character (see \cite[Remark 3.4]{BK23}). Moreover, Theorem~\ref{thm:irregular-FM-Eisenstein} shows that if
       	\(\bar\chi\) is non-quadratic and satisfies the hypotheses of the
       	theorem, then no such representation \(\rho\) can exist. 
       \end{rem}
        
    \subsection{The case $F \ne \Q$}\label{F ne Q}
     In this subsection, we prove a big $R=\T$ theorem in the residually reducible case in the setting of completed cohomology of quaternionic forms. 
    
     Let $F$ be an abelian totally real field of even degree such that $p$ splits completely in $F$. Write $\Sigma$ for a finite set of finite places of $F$ containing $\Sigma_p$ (the set of places of $F$ lying above $p$). Let $\bar{\chi} : G_{F, \Sigma} \to \F^\times$ be a continuous, totally odd character such that it can be extended to a character of $G_\Q$.  Let $\chi : G_{F, \Sigma} \to \cO^\times$ be a continuous, de Rham character lifting $\bar{\chi}$. Let $\bar{T} = 1 + \bar{\chi} : G_{F, \Sigma} \to \F$ be a two-dimensional pseudo-representation. Let $R_F^{\ps, \chi}$ be the universal pseudo-deformation ring of $\bar{T}$ with determinant $\chi$. 
    
    
    Note that $F$ is a totally real field of even degree. Then there exists a quaternion algebra $D$ over $F$ ramified exactly at all infinite places. Fix an isomorphism between $(D \otimes_{F}\A_{F}^\infty )^\times$ and $\GL_2(\A_{F}^\infty)$. Write $\psi = \chi \varepsilon$ for a character of $(\A_F^\infty)^\times/F_{+}^\times $ via global class field theory. We consider tame levels $U^p=\prod_{v\nmid p}U_v$ with $U_v=\GL_2(\cO_{F_v})$ for $v\notin\Sigma$ and $\psi|_{U^p\cap(\A_F^\infty)^\times}=1$. Thus the choice at $v\in\Sigma\setminus\Sigma_p$ must be compatible with the fixed character $\chi$. As in Section \ref{sec for T}, we can define $\T_\psi(U^p)$. Whenever the residual eigensystem $\bar T$ occurs at this level, denote the corresponding maximal ideal by $\km$; then there is a natural surjection $R_F^{\ps,\chi}\twoheadrightarrow\T_\psi(U^p)_\km$. We say a point $\kp\in\Spec R_F^{\ps,\chi}$ is \textit{pro-modular} of tame level $U^p$ if it is also a point of $\Spec\T_\psi(U^p)_\km$.
    
    Recall that $R_{F, 0}^{\ps, \chi} = R_F^{\ps, \chi}/I_F$, where $I_F$ is the intersection of all minimal primes of dimension at least $1+2[F:\Q]$. 
    
    \begin{thm}\label{big R=T quaternion}
    	Keep notations and assumptions above. Further assume that for any $v\mid p$, $\bar{\chi}|_{G_{F_v}} \ne \omega^{\pm 1}$. Then there exist a tame level $U^p$ as above and a maximal ideal $\km$ corresponding to $\bar T$ such that $\Spec R_{F, 0}^{\ps, \chi}=\Spec \T_{\psi}(U^p)_\km^{\operatorname{red}}$. Moreover, $R_{F, 0}^{\ps, \chi} $ is equidimensional of dimension $1+2[F:\Q]$.
    \end{thm}
    
    To prove Theorem \ref{big R=T quaternion}, we use a strategy similar to that of Section \ref{F=Q}.
    
    The adjoint Bloch-Kato Selmer group input needed below was established in Corollary \ref{cor:ordinary-adjoint-vanishing}. In particular, it applies to ordinary regular automorphic points over $F$ without any residual genericity assumption at the places in $\Sigma$.
    
    Choose $U^p$ as above with $U_v\subseteq K_1(\varpi_v^{n_v})$ and $n_v\ge\operatorname{Sw}_v(1\oplus\bar\chi)+2$ for $v\in\Sigma\setminus\Sigma_p$, where $K_1(\varpi_v^{n_v})$ is defined by $c\equiv0$ and $d\equiv1\pmod{\varpi_v^{n_v}}$. Shrink these subgroups further if necessary to ensure compatibility with $\psi$. By the conductor bound used in Section \ref{F=Q}, local-global compatibility and local newform theory, this single level accommodates all regular cuspidal automorphic points of $R_F^{\ps,\chi}$. The points constructed below will also show that the residual maximal ideal $\km$ exists.
    
    Now we consider an irreducible component $C$ of $\Spec R_{F, 0}^{\ps, \chi}$, which is of dimension at least $1+2[F: \Q]$. We claim that $C$ is pro-modular of tame level $U^p$. As $\bar{\chi} $  can be extended to a character of $G_\Q$, we only need to consider the following two cases.
    
    \textbf{Case 1}: $ \bar{\chi}|_{G_{F_v}} \ne \mathbf{1}$ for any $v\mid p$.
    
    Write $C^\ord := C \cap \Spec R^{\ps, \ord}$, where $R^{\ps, \ord} $ is the universal ordinary pseudo-deformation ring defined in Section \ref{sec deformation ring}. For $v \mid p$, let $R_v^{\ps, \chi}$ be the local pseudo-deformation ring of $1+\bar{\chi}$ with determinant $\chi$ and let $R_v^{\ps, \ord}$ be its ordinary locus.
    By \cite[Corollary B.20]{Pa_k_nas_2013}, the kernel of $ R_v^{\ps, \chi} \twoheadrightarrow R_v^{\ps, \ord}$ is principal. Then by Krull principal ideal theorem, every irreducible component of $C^\ord$ is of dimension at least $1+2[F:\Q]-[F:\Q]=1+[F:\Q]$. Then by the same argument as in the proof of \cite[Corollary 7.2.3]{pan2022fontaine}, we can find an irreducible component $C_1^\ord \subset C^\ord$ of dimension $1+[F:\Q]$ such that the set of \textit{irreducible regular de Rham} (see \cite[Corollary 7.2.3]{pan2022fontaine}) primes are Zariski dense in $C_1^\ord$.
    
    Let $\kp$ be an irreducible regular de Rham prime in $C_1^\ord $. By the ordinary Fontaine-Mazur theorem (Theorem \ref{ord1}), the associated Galois representation $\rho(\kp)$ arises, up to twist, from a regular algebraic cuspidal automorphic representation of $\GL_2(\A_F)$ and is ordinary at each $v|p$. By our choice of $U^p$, its eigensystem occurs in $\T_\psi(U^p)$ and reduces to $\bar T$, giving the maximal ideal $\km$ above. Thus $\kp$ is pro-modular of tame level $U^p$. Corollary \ref{cor:ordinary-adjoint-vanishing} gives $H^1_f(F,\operatorname{ad}\rho(\kp))=H^1_{g,\Sigma}(F,\operatorname{ad}\rho(\kp))=0$, and hence $H^1_{g,\Sigma}(F,\operatorname{ad}^0\rho(\kp))=0$. In particular, no condition on $\bar\chi|_{G_{F_v}}$ at $v\nmid p$ is needed for this vanishing.
    
    Note that $\rho(\kp)$ is irreducible, so $\widehat{(R_F^{\ps,\chi})_{\kp}}$ is its fixed-determinant deformation ring. Put $W=\operatorname{ad}^0\rho(\kp)$. As in Lemma \ref{iso}, the local quotients $H^1(F_v,W)/H^1_g(F_v,W)$ have dimension $2$ for $v|p$, whereas the global Euler characteristic formula gives $h^1-h^2=2[F:\Q]$. The vanishing of $H^1_{g,\Sigma}(F,W)$ therefore gives $h^2=0$ and $h^1=2[F:\Q]$. Hence the completed local ring is regular of dimension $2[F:\Q]$, and $C$ is the unique component containing $\kp$, of dimension $1+2[F:\Q]$. By \cite[Corollary 3.3.10]{Zhang:2024aa}, a minimal prime of $\T_\psi(U^p)_\km$ contained in $\kp$ has dimension at least $1+2[F:\Q]$. Its contraction must equal the generic point of $C$, by uniqueness and dimension. Thus $C$ is pro-modular.
    
    \textbf{Case 2}: $ \bar{\chi}|_{G_{F_v}} = \mathbf{1}$ for any $v\mid p$.
    
    Let $C^{\ord,\mathbf{1}}$ be the pullback of $C$ to $\Spec R_1^{\ps,\ord}$. By Lemma \ref{lem:ordinary-character-quotients}, the graph calculation in Proposition \ref{potential R=T} and Theorem \ref{ord2}, we have $\dim C^{\ord,\mathbf{1}}=1+[F:\Q]$. Choose a component $C_1^{\ord,\mathbf{1}}$ of largest dimension. Its finite map to $\Spec\Lambda_F$ is surjective by equality of dimensions. As in Proposition \ref{potential R=T}, this gives an \textit{irreducible regular de Rham} point on $C_1^{\ord,\mathbf{1}}$, whose image $\kp'$ lies in $C$. We can then apply the same arguments as in \textbf{Case 1}.
    
    Now we can conclude Theorem \ref{big R=T quaternion}.
    
    \begin{rem}\label{remark for F}
    	1) The proof in this case is much simpler than that in the case $F=\Q$. One reason is that the ordinary locus has dimension $1+[F: \Q]>2$, whereas the reducible locus has dimension at most $2$ by Leopoldt's conjecture. Thus, Pan's finiteness results (Theorem \ref{ord1} and Theorem \ref{ord2}) allow us to find cuspidal automorphic points rather than those corresponding to Eisenstein series. Another reason is that we need not worry about the level of $\pi$ at each $v|p$ since we use completed cohomology here.
    	
    	2) The dimension argument above also proves the equidimensionality of $R_{F,0}^{\ps,\chi}$ when $F\ne\Q$ satisfies the other hypotheses without requiring $[F:\Q]$ to be even: the construction of a regular automorphic point on each component does not use the parity of $[F:\Q]$. However, we do not know whether the full ring $R_{F}^{\ps,\chi}$ is equidimensional, because an analogue of Proposition \ref{Gal coh} (or the vanishing of certain $\mu$-invariants) is not known even when $F$ is abelian. See the proof of \cite[Proposition 2.4.4]{Deo2023} for more details.
    	
    	3) It is also possible to prove the Zariski density of \textit{essentially modular} points (corresponding to regular algebraic cuspidal automorphic representations with level prime to $p$ up to twist) in $\Spec R_{F, 0}^{\ps, \chi}$ by using \cite[Theorem 5.9]{chenevier2011}. Although the result in \textit{loc. cit.} is in the residually irreducible setting, it is asserted in \cite[page 969]{chenevier2011} that we may get the same result in the residually reducible case by working with pseudo-representations instead of representations. The author has not checked the details.
    \end{rem}
    
     \subsection{Open question: Zariski density of modular points in the deformation space}\label{q1}  In this subsection, we discuss the relation between our previous results and Serre's modularity conjecture in the residually reducible case.
    
    Serre's modularity conjecture states that for an odd, irreducible, continuous Galois representation $\bar{\rho}: G_{\Q} \to \GL_2(\F)$ with a finite field $\F$ of characteristic $p$ (including the case $p=2$), it arises from a newform. It has been proved in \cite{Khare_2009a} and \cite{Khare_2009}. If $\bar{\rho}$ is reducible, then we say it is \textit{modular} if there exist a Galois representation $\rho_{f}: G_{\Q} \to \GL_2(\cO_f)$ associated to some newform $f$ and a suitable choice of lattice over $\cO_f$ such that $\bar{\rho} \cong \rho_{f} \modd \mathfrak{m}_{\cO_f}$.
    
    To state it more precisely, we give some notations first. Let $p \ge 3$ be an odd prime, and let $F$ be an abelian totally real field in which $p$ splits completely. Let $\Sigma$ be a finite set of finite places containing $\Sigma_p$, the set of places $v$ of $F$ above $p$. We denote by $K$  a $p$-adic local field with its ring of integers $\cO$, a uniformizer $\varpi$ and its residue field $\F$. Let $\bar{\chi}: G_{F, \Sigma} \to \F^\times$ be a continuous, totally odd character.
    
    
    
    Let $x=[\phi_x] \in H^1(F_\Sigma/F, \F(\bar{\chi}^{-1}))$ be a non-zero cohomology class. It is easy to see that $x$ defines a two-dimensional representation (up to equivalence)
    $$ \bar{\rho}_x : G_{F, \Sigma} \to \GL_2(\F), ~~ g \mapsto \begin{pmatrix}
    	1 & \phi_x(g)\bar{\chi}(g)\\
    	0& \bar{\chi}(g)
    \end{pmatrix}. $$
    For simplicity, we choose a complex conjugation $c \in G_{F, \Sigma}$ and an element $\sigma \in G_{F, \Sigma}$, and assume $\phi_x(c)=0, \phi_x(\sigma)=1$ for the definition of $ \bar{\rho}_x$. Then it is conjectured that the Galois representation $ \bar{\rho}_x$ arises from a cuspidal Hilbert modular eigenform $f$ (a classical newform when $F=\Q$). 
    
    
    Let $\chi : G_{F, \Sigma} \to \cO^\times$ be a de Rham lifting of $\bar{\chi}$. As $x \ne 0$, we can define the universal deformation ring $R_x^{\chi}$ of the non-split Galois representation $ \bar{\rho}_x$ with determinant $\chi$. Let $ R_{F}^{\ps, \chi}$ be the universal pseudo-deformation ring parametrizing all pseudo-deformations of $1+\bar{\chi}$ with determinant $\chi$. By the universal property, we have a natural map $R_{F}^{\ps, \chi} \to R_x^{\chi} $.
    
    Let $\kp$ be a one-dimensional prime of $R_x^{\chi}$ such that the Galois representation $\rho(\kp)$ associated to $\kp$ is irreducible. Here we do not require
    that $R_x^{\chi}/\kp$ is of characteristic $0$.
    
    \begin{prop}\label{dim of x} Assume that $\bar{\chi} $ can be extended to a character of $G_{\Q}$ and $\bar{\chi}|_{G_{F_v}} \ne \omega^{\pm 1}$ for any $v|p$.
    	Then for any irreducible component of $\Spec R_x^{\chi}$ containing $\kp$, it is of Krull dimension $1+2[F:\Q]$.
    \end{prop}
    
    \begin{proof}
    	Let $C$ be an irreducible component of $\Spec R_x^{\chi}$ containing $\kp$ with generic point $\kq$. Let $\kq'$ be the kernel of the composite map $R_{F}^{\ps, \chi} \to R_x^{\chi} \twoheadrightarrow R_x^{\chi}/\kq$. We show that $C$
    	is of Krull dimension	$1+2[F:\Q]$.
    	
    	By the global characteristic formula, we know that each irreducible component of $ R_x^{\chi}$ has Krull dimension at least $1+2[F:\Q]$. In particular, $\dim C \ge 1+2[F:\Q]$.
    	
    	Note that $\rho(\kp)$ is irreducible, which implies that the Galois representation $\rho(\kq)$ associated to $\kq$ is also irreducible. By \cite[Corollary 2.3.6 (2)]{pan2022fontaine}, we have $\dim R_x^{\chi}/\kq \le \dim R_{F}^{\ps, \chi}/\kq'$. The opposite bound $\dim R_{F}^{\ps,\chi}/\kq'\le 1+2[F:\Q]$ follows from Theorem \ref{big R=T quaternion} and Remark \ref{remark for F} 2) if $F\ne\Q$, and from Theorem \ref{equi dim 3} if $F=\Q$. This proves the result.
    \end{proof}
    
    Keep notations in the proof of Proposition \ref{dim of x}. The dimension equalities above show that $\kq'$ defines an irreducible component $C'$ of $\Spec R_F^{\ps,\chi}$, with $\dim C=\dim C'=1+2[F:\Q]$. When $[F:\Q]$ is even, Theorem \ref{big R=T quaternion} shows that $C'$ is pro-modular, without any residual genericity assumption at the tame places; Remark \ref{remark for F} 3) further suggests the Zariski density of modular points in $C'$. This does not directly settle the corresponding question for the component $C$ of the deformation space. We may therefore ask the following question in the general setting above.
    
    \begin{ques}\label{ques:5.6.2}
    	Let $C$ be an irreducible component of $\Spec R_x^{\chi}$. Assume that it contains an irreducible one-dimensional point, i.e. an irreducible lifting of $ \bar{\rho}_x$ (not necessarily of characteristic $0$). 
    	\begin{enumerate}
    		\item Can we find an irreducible modular point in the component $C$?
    		\item  If so, is the set of modular points in $C$ Zariski dense?
    	\end{enumerate}
    \end{ques}
    
    
    \begin{rem}
    	\label{rem:fixed-ramification-and-FKP}
    	\begin{enumerate}
    		\item
    		The existence of an irreducible lifting of \(\bar\rho_x\) with
    		ramification contained in \(\Sigma\) does not seem to be known in this
    		generality at present. 
    		
    		\item
    		Let us compare this with the related work \cite{Fakhruddin2022}. They study an analogue of Serre's modularity conjecture in the residually reducible case for a fixed Galois representation of $G_{\Q}$ (see \cite[Theorem B]{Fakhruddin2022}). Their proof uses an extension of
    		Ramakrishna's method to construct irreducible geometric \(p\)-adic lifts, and
    		then invokes modularity lifting theorems. However, their construction allows
    		the introduction of new ramified primes. 
    		This is different from the fixed-ramification problem considered here.
    		
    		\item
    		The fixed-\(\Sigma\) version of related questions has been studied in
    		 \cite{BK} when $F=\Q$. In their setting, one fixes a finite set of places
    		\(\Sigma\) and studies extension classes in the residual Bloch-Kato Selmer group
    		\(
    		H^1_\Sigma(\Q,\bar{\chi}^{-1})
    		\)
    		which are modular by Galois representations attached to automorphic forms
    		that are also unramified outside \(\Sigma\). Both their work and the discussion in
    		\cite{Fakhruddin2022} suggest that one should not expect a naive
    		fixed-level or minimal-level statement for all actual reducible residual
    		representations. In particular, unlike in the residually irreducible case of
    		Serre's modularity conjecture in \cite{Khare_2009a} and \cite{Khare_2009}, one should not expect the Serre weight and the Artin
    		conductor of \(\bar\rho_x\) alone to determine a modular lift. See \cite[Proposition 7.7]{Fakhruddin2022}.
    		
    		\item
    		Finally, in some special cases the answer to the analogue of
    		Question~\ref{ques:5.6.2} is known to be affirmative. For example, when
    		\(F=\mathbb Q\) and the relevant global Galois cohomology group has dimension
    		\(1\), Deo proves a big \(R=\T\) theorem for the universal
    		pseudo-deformation ring and the fixed-level big Hecke algebra in
    		\cite{Deo2023}. In such a situation the pseudo-deformation ring can be compared effectively with the deformation ring of a chosen non-split residual
    		representation, and modular points are Zariski dense. The general
    		fixed-\(\Sigma\) question, however, remains much less clear because the natural map $R_{F}^{\ps, \chi} \to R_x^{\chi} $ is not finite in general.
    	\end{enumerate}
    \end{rem}

    \section{Local-to-global finiteness and dimensions of global deformation rings}\label{sec 6}
    In this section, we give two approaches to local-to-global finiteness for pseudo-deformation rings and apply them to the dimensions of global deformation spaces.
    
    The first argument, given in Section \ref{sec finite 1}, is based on Pan’s ordinary finiteness theorems and is close in spirit to the method of Allen–Calegari. It is relatively short and applies in a flexible range of residually reducible settings. In particular, it already gives the finiteness input required for the proof of Mazur’s dimension conjecture in the Eisenstein case.
    
    The second argument, developed in Section \ref{sec finite 2}, uses
    potential pro-modularity together with local-global compatibility. It
    first establishes finiteness over an auxiliary totally real field and
    then descends it to the original field. This method provides an
    alternative to the ordinary argument and, unlike that argument, admits
    an analogue without residual reducibility, as developed in
    Appendix \ref{app:nonEis}.
    
    Either finiteness theorem is sufficient for the Eisenstein dimension argument of Section \ref{sec:MazurEis}. We use the second formulation there in order to keep the proof parallel to the non-Eisenstein argument in Appendix \ref{A}. We then apply the resulting dimension theorems to Emerton’s equidimensionality conjecture in Section \ref{emerton's conj}. 
    
    For notations, we fix an odd prime $p$, and $F$ an abelian totally real field in which $p$ is unramified. Let $\Sigma$ be a finite set of finite places containing $\Sigma_p$, the set of finite places $v$ of $F$ lying above $p$. 
    
    Throughout this section, we denote by $K$  a $p$-adic local field with its ring of integers $\cO$, a uniformizer $\varpi$ and its residue field $\F$. Let $\bar{\chi}: G_{F, \Sigma} \to \F^\times$ be a continuous, totally odd character, and suppose that $\bar{\chi} $ can be extended to a character of $G_{\Q}$. 
    
    \subsection{Local-to-global finiteness (I)}\label{sec finite 1}
    
    In this subsection, we prove a finiteness result as an analogue of \cite[Theorem 1]{Allen_2014} in the residually reducible case.
    
    Let $\chi: G_{F, \Sigma} \to \cO^\times$ be a continuous, totally odd character lifting $\bar{\chi}$ such that $\chi$ is de Rham for any $v \in \Sigma_p$. 
    
    Let $R_{F}^{\ps, \chi}$ be the universal pseudo-deformation ring of the pseudo-representation $1+\bar{\chi}$ for $ G_{F, \Sigma}$ with determinant $\chi$. Let $R_v^\ps$ be the local universal pseudo-deformation ring of $1+\bar{\chi}|_{G_{F_v}}$  for $ {G_{F_v}}$ with determinant $\chi|_{G_{F_v}}$. Let $R_{p, F}^\ps$ be the completed tensor product (over $\cO$) of all $R^{\ps}_v, v|p$ . Then by the universal property, we get a natural map $R_{p, F}^\ps \to R_{F}^{\ps, \chi}$.
    
    \begin{prop}\label{finiteness1}
    	The map $R_{p, F}^\ps \to R_{F}^{\ps, \chi}$ is finite.
    \end{prop}
    
    \begin{proof}
    	Let $\mathfrak{m}_p$ be the maximal ideal of $R_{p, F}^\ps$. We show that $ R_{F}^{\ps, \chi}/(\mathfrak{m}_p)$ is an Artinian local ring. Assume it is not true. Then we can find a prime $\kp$ of $ R_{F}^{\ps, \chi}/(\mathfrak{m}_p)$ of dimension one and we deduce a contradiction.
    	
    	 Note that $\kp$ is of characteristic $p$. By the universal property, the pseudo-representation $T(\kp)$ satisfies $T(\kp)|_{G_{F_v}} = 1+\bar{\chi}|_{G_{F_v}}$ for each $v|p$. Its local restrictions are therefore sums of characters. As $\bar{\chi}$ can be extended to a character of $G_{\Q}$, we can divide the proof into the following two cases.
    	
    	If $\bar{\chi}|_{G_{F_v}} \ne \mathbf{1}$ for any $v|p$, let $R^{\ps, \ord}$ be the universal ordinary pseudo-deformation ring as defined in Section \ref{sec deformation ring}. Then we get a surjection $R^{\ps, \ord} \twoheadrightarrow R_{F}^{\ps, \chi}/\kp$ as both are quotients of $ R_{F}^{\ps, \chi}$. Thus, by Theorem \ref{ord1}, the composite map $\Lambda_F \to R^{\ps, \ord} \to R_{F}^{\ps, \chi}/\kp$ is finite. From our construction, the kernel of the map $ \Lambda_F \to R_{F}^{\ps, \chi}/\kp$ is the maximal ideal of the Iwasawa algebra $\Lambda_F $, and the map $\Lambda_F/\mathfrak{m}_{\Lambda_F} \to R_{F}^{\ps, \chi}/\kp$ is finite. This is impossible as $1=\dim R_{F}^{\ps, \chi}/\kp \le \dim \Lambda_F/\mathfrak{m}_{\Lambda_F} =0$.
    	
    	If $\bar{\chi}|_{G_{F_v}}=\mathbf{1}$ for any $v|p$, put $D=R_F^{\ps,\chi}/\kp$. Over an algebraic closure of $\operatorname{Frac}(D)$, the representation realizing $T(\kp)$ is upper triangular at each $v|p$ with both diagonal characters equal to $\mathbf{1}$. Thus the matrix calculation in \cite[Lemma 5.3.2]{pan2022fontaine} gives the $\mathbf{1}$-ordinary identities over $D$. We obtain a surjection $R_1^{\ps,\ord}\twoheadrightarrow D$ with every specified ordinary character equal to $\mathbf{1}$. The map $\Lambda_F\to D$ therefore kills the maximal ideal of $\Lambda_F$, and Theorem \ref{ord2} again contradicts $\dim D=1$.
    	
    	Now we can conclude the proof by applying Nakayama's lemma.
    \end{proof}
    
    \begin{rem}\label{6.1.2}
    	
    	 In the proof of Proposition \ref{finiteness1}, we use a similar strategy as in \cite{Allen_2014}. Also, we can prove a finiteness result about the unramified pseudo-deformation ring as in \cite[Theorem 1]{Allen_2014} using the following arguments.
    	
    	Keep all the assumptions on $F$ and $\bar{\chi}$, and further suppose that $\bar{\chi}$ is unramified at $v|p$. Let $\chi$ be a finite order character lifting $\bar\chi$, unramified at each $v|p$, and let $R^{\ps,\operatorname{unr}}$ be the universal pseudo-deformation ring of $1+\bar\chi$ for $G_{F,\Sigma\setminus\Sigma_p}$ with determinant $\chi$. The same Artinian-fibre argument shows that $R^{\ps,\operatorname{unr}}$ is finite over $\cO$. Indeed, for any one-dimensional prime quotient $D$ of $R^{\ps,\operatorname{unr}}/\varpi$, after a finite integral extension of $D$ we can choose the unramified characters on the local stable lines. This gives a map from $R^{\ps,\ord}$ or $R_1^{\ps,\ord}$ whose image contains $D$, and the weight map kills the maximal ideal of $\Lambda_F$. Theorem \ref{ord1} or Theorem \ref{ord2} will make this image finite over $\F$, a contradiction. Such a finiteness result may be used to study the unramified Fontaine-Mazur conjecture in the residually reducible case.
    \end{rem}

    \subsection{Local-to-global finiteness (II)}\label{sec finite 2}
    In this subsection, we give another proof of local-to-global finiteness,
    using potential pro-modularity and local-global compatibility. This
    argument also provides the model for Appendix \ref{app:nonEis}.
    
    Throughout this subsection, we assume that $p$ splits completely in $F$.
    Let $\chi:G_{F,\Sigma}\to\cO^\times$ be a continuous, totally odd
    character lifting $\bar\chi$, de Rham at every $v\in\Sigma_p$.
    Recall that $\bar\chi$ extends to a character of $G_\Q$. By
    Remark \ref{rem:twist_pro_modularity}, the pro-modularity and big
    $R=\T$ statements of Section \ref{potential pro-modular} apply to
    this determinant. Fix an extension $F^1/F$ as constructed there,
    and write $\Sigma^1$ and $\Sigma_p^1$ for the places of $F^1$ above
    $\Sigma$ and $\Sigma_p$, respectively.
    
    Let $R_{F^1}^{\ps,\chi}$ be the universal pseudo-deformation ring of
    $1+\bar\chi|_{G_{F^1,\Sigma^1}}$ with determinant $\chi|_{G_{F^1}}$,
    and define $R_{F^1,0}^{\ps,\chi}$ as in Corollary \ref{lower}.
    Define $R_F^{\ps,\chi}$ and $R_{F,0}^{\ps,\chi}$ similarly.
    For $v\in\Sigma_p^1$ (resp. $v\in\Sigma_p$), let $R_v^\ps$ be the
    local pseudo-deformation ring of $1+\bar\chi|_{G_{F_v^1}}$
    (resp. $1+\bar\chi|_{G_{F_v}}$) with determinant $\chi$.
    Put
    \[
    R_{F^1,p}^\ps=\widehat{\bigotimes}_{v\in\Sigma_p^1}R_v^\ps,
    \qquad
    R_{F,p}^\ps=\widehat{\bigotimes}_{v\in\Sigma_p}R_v^\ps,
    \]
    where the completed tensor products are over $\cO$. Since $p$
    splits completely in $F^1$, the local identifications
    $F_{v'}^1\cong F_v\cong\Q_p$ for $v'|v$ induce a natural surjection
    $R_{F^1,p}^\ps\twoheadrightarrow R_{F,p}^\ps$.
    
    \begin{lem}\label{com dia2}
    	There is a commutative diagram
    	\[
    	\begin{tikzcd}
    		R_{F^1,p}^\ps \arrow[r] \arrow[d, two heads]
    		& R_{F^1,0}^{\ps,\chi} \arrow[d]\\
    		R_{F,p}^\ps \arrow[r]
    		& R_{F,0}^{\ps,\chi}.
    	\end{tikzcd}
    	\]
    \end{lem}
    
    \begin{proof}
    	The right vertical map is well-defined by Corollary
    	\ref{finite for 0}. The commutativity follows from the universal
    	property and the compatibility of local and global restriction.
    \end{proof}
    
    \begin{prop}\label{finiteness 2}
    	The natural map $R_{F,p}^\ps\to R_{F,0}^{\ps,\chi}$ is finite.
    	Moreover, when $F=\Q$, the natural map
    	$R_p^{\mathrm{ps},\chi}\to R_\Q^{\mathrm{ps},\chi}$ is finite.
    \end{prop}
    
    \begin{proof}
    	By Proposition \ref{R=T}, there is an isomorphism
    	$R_{F^1,0}^{\ps,\chi}\cong(\T_\km)^{\operatorname{red}}$ for the
    	big Hecke algebra in Section \ref{potential pro-modular}.
    	Proposition \ref{finiteness from lgc} therefore shows that
    	$R_{F^1,p}^\ps\to R_{F^1,0}^{\ps,\chi}$ is finite.
    	The map $R_{F^1,0}^{\ps,\chi}\to R_{F,0}^{\ps,\chi}$ is finite
    	by Corollary \ref{finite for 0}. Hence their composite is finite.
    	By Lemma \ref{com dia2}, this composite factors through
    	$R_{F,p}^\ps$, so $R_{F,p}^\ps\to R_{F,0}^{\ps,\chi}$ is finite
    	as well.
    	
    	When $F=\Q$, Proposition \ref{at least 3} gives
    	$R_{\Q,0}^{\ps,\chi}=(R_\Q^{\ps,\chi})^{\operatorname{red}}$.
    	The last assertion then follows from
    	Proposition \ref{finite-after-reduction}.
    \end{proof}
    
    \begin{rem}\label{rem on fin}
    	Let us compare the two proofs of finiteness.
    	Proposition \ref{finiteness1} is closer in spirit to
    	\cite{Allen_2014}: it uses Pan's ordinary finiteness theorems
    	(Theorems \ref{ord1} and Theorem \ref{ord2}) and applies without requiring
    	that $p$ split completely in $F$.
    	
    	In contrast, Proposition \ref{finiteness 2} uses potential
    	pro-modularity over an auxiliary field and local-global
    	compatibility, followed by descent. Although its conclusion
    	is already implied by Proposition \ref{finiteness1} in the present
    	residually reducible setting, this proof extends to the residually
    	irreducible setting in Appendix \ref{app:nonEis}.
    	
    	Both arguments compare the global ring with local
    	pseudo-deformation rings, rather than local framed deformation
    	rings as in \cite{Allen_2014}. Their smaller Krull dimensions
    	are essential for the dimension arguments over $\Q$ in
    	Section \ref{sec:MazurEis} and Appendix \ref{A}.
    \end{rem}
    
    \subsection{Mazur's dimension conjecture in the Eisenstein case}\label{sec:MazurEis}
      In this subsection, we prove the equidimensionality of the universal pseudo-deformation ring in the Eisenstein case for all odd primes.
      
       Let $\Sigma$ be a finite set of finite places of $\Q$ containing $p$. Let $\bar{T}: G_{\Q, \Sigma} \to \F$ be a two-dimensional pseudo-representation satisfying Assumption \ref{odd ass}. Write $\bar{\chi}:= \operatorname{det} \bar{T}$. Then there exists a universal pseudo-deformation ring $R_\Q^\ps$ of the pseudo-representation $\bar{T}$. Let $\chi: G_{\Q, \Sigma} \to \cO^\times$ be a continuous character lifting $\bar{\chi}$. Let $ R_\Q^{\ps, \chi}$ be the universal pseudo-deformation ring of the pseudo-representation $\bar{T}$ with determinant $\chi$.
      
      Assume that $\bar{T}$ is reducible. After some finite twist, we may write $\bar{T}=1+\bar{\chi}$. By Proposition \ref{det}, we may choose \(\chi\) to be crystalline at \(p\).
      
      \begin{thm}\label{equi dim 3}
      	Each irreducible component of $\Spec R_\Q^{\ps, \chi}$ is of Krull dimension $3$.
      \end{thm}
      
      \begin{proof}
      	Put $R=R_\Q^{\ps,\chi}$ and $P=R_p^{\ps,\chi}$. By Propositions \ref{local pseudo} and \ref{finiteness 2}, $P$ is an $\cO$-flat domain of dimension $4$ and $R$ is finite over $P$. In view of Proposition \ref{at least 3}, it suffices to exclude a minimal prime $\mathfrak b$ of $R$ with $\dim R/\mathfrak b=4$. For such a prime, the finite map $P\to R/\mathfrak b$ is injective, since both rings are domains of dimension $4$.

      	First suppose $\bar\chi|_{G_{\Q_p}}\ne\mathbf{1}$. By Lemma \ref{dim of v ord}, the local ordinary quotient $P^{\ord}=R_p^{\ps,\ord}$ is a power series ring of relative dimension $2$ over $\cO$. Going up gives a prime $\mathfrak b'\supset\mathfrak b$ of $R$ lying above $\ker(P\to P^{\ord})$. Thus $P^{\ord}\hookrightarrow R/\mathfrak b'$ is finite and $\dim R/\mathfrak b'=3$. On the other hand, $R/\mathfrak b'$ is a quotient of the global ordinary ring $R_\Q^{\ps,\ord}$, hence is finite over $\Lambda_\Q$ by Theorem \ref{ord1}. This contradicts $\dim\Lambda_\Q=2$.

      	Now suppose $\bar\chi|_{G_{\Q_p}}=\mathbf{1}$. Let $R_p^1\cong\cO[[X,Y]]$ be the deformation ring of the trivial character. Lemma \ref{lem:ordinary-character-quotients} and the graph calculation in Proposition \ref{potential R=T} give $\dim(R_{\Q,1}^{\ps,\ord}/\mathfrak bR_{\Q,1}^{\ps,\ord})=\dim((R/\mathfrak b)\otimes_P R_p^1)\ge4+2-3=3$. Theorem \ref{ord2} bounds this dimension by $\dim\Lambda_\Q=2$, again a contradiction.
      	
      \end{proof}
      
      \begin{rem}\label{6.3.2}
      	Both Proposition \ref{finiteness1} and Proposition \ref{finiteness 2} imply the finiteness statement required above. For compatibility with the proof in the residually irreducible case, we use Proposition \ref{finiteness 2}. The essential point is that the source is a local pseudo-deformation ring of Krull dimension $4$, rather than a framed local deformation ring of Krull dimension $7$ (see \cite[Theorem 1.5]{B_ckle_2023}). The finiteness theorem therefore reduces the desired global upper bound to the exclusion of a single possible excess-dimensional case.
      \end{rem}

      \begin{thm}\label{mazur thm eisenstein}
      	1) Each irreducible component of $\Spec R_\Q^{\ps, \chi}$ is of characteristic $0$.
      	
      	2) Each irreducible component of $\Spec R_\Q^\ps$ is of dimension $4$ and of characteristic $0$.
      \end{thm}
      
      \begin{proof}
      	We keep most notations of various pseudo-deformation rings as in the previous proof. Assume that $\kp$ is a minimal prime of $ \Spec R_\Q^{\ps, \chi}$ of characteristic $p$. 
      	
      	We first consider the case $\bar{\chi}|_{G_{\Q_p}} \ne \mathbf{1}$. We first claim that the ring $R_p^{\mathrm{ps},\chi}/\varpi$ is an integral domain. Let $R_p^\ps$ be the local pseudo-deformation ring of $1+\bar{\chi}|_{G_{\Q_p}}$ without fixed determinant. By \cite[Corollary 4.29]{B_ckle_2023}, $R_p^\ps/\varpi$ is an integral domain. Using an argument similar to that in Proposition \ref{det}, we have $R_p^\ps/\varpi \cong (R_p^{\mathrm{ps},\chi}/\varpi) \widehat{\otimes}_{\F} (R_p^1/\varpi) $. Since $R_p^1/\varpi$ is a power series ring over $\F$, it follows that $R_p^{\mathrm{ps},\chi}/\varpi$ is an integral domain.
      	
      	Consider the finite map $R_p^{\mathrm{ps},\chi}\to R_{\mathbb Q}^{\mathrm{ps},\chi}/\kp$. Combining Proposition \ref{local pseudo}, Theorem \ref{equi dim 3} and the previous claim, we get a natural finite injection $R_p^{\mathrm{ps},\chi}/\varpi \hookrightarrow R_{\mathbb Q}^{\mathrm{ps},\chi}/\kp$ as both are domains with dimension $3$. As $R_p^{\mathrm{ps},\mathrm{ord}}/\varpi$ is a quotient of $ R_p^{\mathrm{ps},\chi}/\varpi$, and meanwhile a power series ring over $\F$ of dimension $2$ (by Lemma \ref{dim of v ord}), the going-up property gives a prime $\kp'$ of $R_{\mathbb Q}^{\mathrm{ps},\chi}/\kp $ of dimension $2$ corresponding to the domain $R_p^{\mathrm{ps},\mathrm{ord}}/\varpi$. By Theorem \ref{ord1}, we get a finite map $\Lambda_\Q/\varpi \to R_{\mathbb Q}^{\mathrm{ps},\chi}/\kp'$. As $\dim \Lambda_\Q/\varpi=1$, we get a contradiction in this case.
      	
      	For the case $\bar\chi|_{G_{\Q_p}}=\mathbf{1}$, apply the graph calculation in Proposition \ref{potential R=T} modulo $\varpi$. The local pseudo-deformation ring and character ring are power series rings over $\F$ in three and two variables, respectively. Lemma \ref{lem:ordinary-character-quotients}, with $I=\kp$, therefore gives $\dim(R_{\Q,1}^{\ps,\ord}/\kp R_{\Q,1}^{\ps,\ord})\ge\dim(R_\Q^{\ps,\chi}/\kp)-1=2$. This quotient is of characteristic $p$ and finite over $\Lambda_\Q/\varpi$ by Theorem \ref{ord2}, whereas $\dim\Lambda_\Q/\varpi=1$. This contradiction proves the first claim.
      	
      	The second claim follows from the first one and Corollary \ref{at least 4}.

      \end{proof}
      	      
      \begin{rem}
      	1) Unlike the classical approach to Mazur's dimension conjecture, our proof
      	does not take a big \(R=\T\) theorem over \(\mathbb Q\) as an input.
      	Nevertheless, it still relies essentially on modularity-theoretic methods:
      	the local-to-global finiteness theorem is obtained from potential
      	pro-modularity and local-global compatibility, while the possible
      	excess-dimensional component is excluded using Pan's ordinary finiteness
      	theorem. The advantage of this approach is that it avoids the additional
      	hypotheses required by the currently available big \(R=T\) theorems, such
      	as those appearing in Theorem \ref{big R=T Q}.
      	
      	2) We will prove similar results in the residually irreducible case in Appendix \ref{A} using different methods, and further show that $ R_\Q^\ps$ is a local complete intersection ring. However, in the residually reducible case, \(R^{\mathrm{ps}}_{\mathbb Q}\) is not
      	a local complete intersection in general, and its structure is mysterious. See \cite[Section 4]{Deo2022} for more examples and details.
      \end{rem}	    
      
      \subsection{Emerton's conjecture}\label{emerton's conj}
      In this subsection, we prove the equidimensionality of the $p$-adic Hecke algebra for all odd primes.
      
      We fix a positive integer $N \ge 1$ such that $p$ does not divide $N$. Write $\Sigma$ for the set of primes of $\Q$ dividing $Np$. Following Section \ref{sec for emerton}, we can define the $p$-adic Hecke algebra $\T:= \T(N)$. Here we keep most notations as in Section \ref{sec for emerton}.
    
   \begin{thm}\label{main thm for emerton}
   	Conjecture \ref{conj} holds for $p \ge 3$.
   \end{thm}
   
   \begin{proof}
   	Let $\km$ be a maximal ideal of $\T$. Combining Theorem \ref{lower bound}, we only need to show that $\dim \T_{\km}=4$. As in Section \ref{sec for emerton}, there exists a pseudo-representation $T_\km: G_{\Q, \Sigma} \to \T_\km$ such that $T_\km(\Frob_l^{\mathrm{arith}})=T_l$ and $\det(T_\km)(\Frob_l^{\mathrm{arith}})=lS_l$ for any $l \notin \Sigma$. Then we can divide the proof into two cases, depending on the reducibility of $T_\km~ \modd \km$. Let $ R_{\Q}^\ps$ be the universal pseudo-deformation ring of $T_\km~ \modd \km$ (for $G_{\Q, \Sigma}$) without fixed determinant. Then we get a natural surjection $ R_{\Q}^\ps \twoheadrightarrow \T_\km$.
   	
   	If $T_\km~ \modd \km$ is reducible, i.e. $\km$ is Eisenstein, combining Theorem \ref{mazur thm eisenstein} and Theorem \ref{lower bound} (after some finite twist), we have $4 \le \dim \T_\km \le \dim R_{\Q}^\ps =4$.
   	
   	If $T_\km~ \modd \km$ is irreducible, i.e. $\km$ is non-Eisenstein, then the universal deformation ring and the universal pseudo-deformation ring are the same (see \cite[Theorem 2.4.1, page 44]{Berger_2013}). Combining Theorem \ref{lower bound} and Theorem \ref{mazur non-eisenstein}, we have $4 \le \dim \T_\km \le \dim R_{\Q}^\ps =4$.
   \end{proof}
    
    \appendix
    \section{Mazur's dimension conjecture and irregular Fontaine-Mazur conjecture in the residually irreducible case}\label{A}
    
     We first prove local-to-global finiteness and Mazur's dimension conjecture over $\Q$ in the residually irreducible case. The dimension result supplies the non-Eisenstein input for Section \ref{emerton's conj}. We then prove $\cO$-flatness of the global deformation ring, using ordinary finiteness in Section \ref{app:ordinary-finiteness} when the local residual representation is reducible and fixed-type finiteness in Section \ref{app:nonordinary-finiteness} when it is absolutely irreducible. Finally, Section \ref{app:irregular-FM-nonEis} uses the two finiteness results to prove the weight-one Fontaine--Mazur conjecture for $p\geq 5$ without a Taylor--Wiles hypothesis.
     
     Throughout the appendix, we fix $p \ge 3$ as an odd prime. Let $\Sigma$ be a finite set of finite places of $\Q$ containing $p$. We denote by $K$  a $p$-adic local field with its ring of integers $\cO$, a uniformizer $\varpi$ and its residue field $\F$.
    
    \subsection{Local-to-global finiteness (III)}\label{app:nonEis}
    In this subsection, we prove a finiteness result in the residually irreducible case. The ordinary argument of Section \ref{sec finite 1} uses Pan's residually reducible finiteness theorems and does not apply directly here. Moreover, the ordinary result of Section \ref{app:ordinary-finiteness} requires the residual representation at $p$ to be reducible. To obtain a statement without this local restriction, we adapt the potential pro-modularity and local-global compatibility method of Section \ref{sec finite 2}. This gives the required local-to-global finiteness theorem and allows us to prove Mazur's dimension conjecture without imposing a Taylor--Wiles hypothesis or any additional local conditions, based on the ideas in \cite{pan2022fontaine} and \cite{Zhang:2024aa}.
    
    Let $\bar\rho:G_{\mathbb Q,\Sigma}\to\mathrm{GL}_2(\mathbb F)$
    be a continuous, odd, absolutely irreducible representation.
    By the level-refined form of Serre's modularity theorem
    \cite{Khare_2009a}, we may choose a regular algebraic cuspidal
    automorphic representation $\pi_0$ of
    $\mathrm{GL}_2(\mathbb A_{\mathbb Q})$, unramified outside
    $\Sigma$, whose associated $p$-adic Galois representation
    $\rho_0$ has residual representation $\bar\rho$.
    After enlarging the coefficient field, choose a stable lattice
    and write
    \(
    \rho_0:G_{\mathbb Q,\Sigma}
    \to \mathrm{GL}_2(\mathcal O),
    \chi:=\det\rho_0.
    \)
    In particular, $\chi$ is a de Rham lift of $\det\bar\rho$.
    Let $L_{\bar\rho}/\mathbb Q$ be the finite Galois extension
    fixed by $\ker\bar\rho$.
    
    Using \cite[Facts 4.23 and 4.27]{Gee_2022}, we  can
    choose a finite soluble totally real Galois extension
    $F^1/\mathbb Q$, linearly disjoint from $L_{\bar\rho}$,
    with the following properties.
    Write $\Sigma^1$ and $\Sigma_p^1$ for the sets of places
    of $F^1$ above $\Sigma$ and $p$, respectively, and put
    \(
    \pi_1:=\mathrm{BC}_{F^1/\mathbb Q}(\pi_0).
    \)
    \begin{itemize}
    	\item $p$ splits completely in $F^1$.
    	\item For every finite place $v\nmid p$ of $F^1$,
    	$\rho_0|_{I_{F^1_v}}$ is unipotent.
    	\item For every $v\in\Sigma^1\setminus\Sigma_p^1$,
    	$\bar\rho|_{G_{F^1_v}}$ is trivial and
    	$\mathrm{Nm}(v)\equiv1\pmod p$.
    	\item $[F^1:\mathbb Q]$ is even and
    	$[F^1:\mathbb Q]\ge
    	4|\Sigma^1\setminus\Sigma_p^1|+2$.
    	\item $\pi_1$ is regular algebraic cuspidal and
    	 unramified outside $\Sigma^1$, and satisfies
    	\(
    	\pi_{1,v}^{\mathrm{Iw}_v}\ne0, 
    	v\in\Sigma^1\setminus\Sigma_p^1.
    	\)
    \end{itemize}
    
    Since $F^1\cap L_{\bar\rho}=\mathbb Q$, we have
    $\bar\rho(G_{F^1})=\bar\rho(G_{\mathbb Q})$.
    Then $\bar\rho|_{G_{F^1,\Sigma^1}}$ remains
    absolutely irreducible.
    Moreover, $\chi|_{G_{F^1}}$ is unramified outside
    $\Sigma_p^1$, and $
    \chi(\mathrm{Frob}_v)\equiv1\pmod\varpi$ for  
   $ v\in\Sigma^1\setminus\Sigma_p^1. $

    Let $R_{F^1}^{\mathrm{univ}}$ be the universal deformation
    ring of $\bar\rho|_{G_{F^1,\Sigma^1}}$ with determinant
    $\chi|_{G_{F^1,\Sigma^1}}$, with no further local conditions
    at the places of $\Sigma^1$.
    Put $\psi=(\chi\varepsilon)|_{G_{F^1}}$, viewed as a
    character of $(\mathbb A_{F^1}^{\infty})^\times/(F^1)_+^\times$
    via global class field theory.
    
    
    	
    
    
     Write $\psi=\chi\varepsilon$. Let $D$ be the quaternion algebra over $F^1$ which is ramified exactly at all infinite places. Fix an isomorphism between $(D \otimes_{F^1}\A_{F^1}^\infty )^\times$ and $\GL_2(\A_{F^1}^\infty)$. We define a tame level $U^p_0=\prod_{v\nmid p}U_{0,v}$ as follows: $U_{0,v}=\GL_2(\cO_{F^1_v})$ if $v\notin \Sigma^1$ and 
    \[U_{0,v}=\mathrm{Iw}_v:=\{g\in\GL_2(\cO_{F^1_v}),g\equiv \begin{pmatrix}*&*\\0&*\end{pmatrix}\mod \varpi_v\}\]
    otherwise. For $v\in\Sigma^1\setminus\Sigma^1_p$, let $\Delta_v$ be the maximal $p$-power quotient of $k(v)^\times$ and let $U_{1,v}$ be the kernel of the map $U_{0,v}\to\Delta_v$ induced by $\begin{pmatrix}a&b\\c&d\end{pmatrix}\mapsto\frac{a}{d}\mod\varpi_v$; for $v\notin\Sigma^1$, put $U_{1,v}=U_{0,v}$. Write $U^p_1=\prod_{v\nmid p}U_{1,v}$. Then we can define a big Hecke algebra $\T:=\T_\psi(U^p_1)$, and the existence of $\pi_0$ implies that $\bar{\rho}|_{G_{F^1, \Sigma^1}}$ defines a maximal ideal $\km$ of $\T$. As in Section \ref{potential pro-modular}, the natural inclusions of spaces of twisted forms induce surjections from $\T_\km$ to all the corresponding twisted big Hecke algebras. As $\bar{\rho}|_{G_{F^1, \Sigma^1}}$ is irreducible, there exists a two-dimensional representation $\rho_\km: G_{F^1, \Sigma^1} \to \GL_2(\T_\km)$. By the universal property, we have a surjection $R_{F^1}^{\operatorname{univ}} \twoheadrightarrow \T_\km $, which induces a surjection $ (R_{F^1}^{\operatorname{univ}})^{\operatorname{red}} \twoheadrightarrow (\T_\km)^{\operatorname{red}}$.
    
    \begin{thm}[potential pro-modularity]\label{potential pro-modularity irr}
    	The surjection $ (R_{F^1}^{\operatorname{univ}})^{\operatorname{red}} \twoheadrightarrow (\T_\km)^{\operatorname{red}}$ is an isomorphism.
    \end{thm}
    
    \begin{proof}
    	It is a direct consequence of \cite[Proposition 4.3.3]{Zhang:2024aa}.
    \end{proof}
    
    For each $v\in\Sigma_p^1$, let $R_v^\ps$ be the local universal
    pseudo-deformation ring of $\operatorname{tr}\bar\rho|_{G_{F_v^1}}$
    with determinant $\chi|_{G_{F_v^1}}$. Let $R_p^\ps$ be the
    corresponding local pseudo-deformation ring over $\Q_p$, and put
    \[
    R_{F^1,p}^\ps=\widehat{\bigotimes}_{v\in\Sigma_p^1}R_v^\ps.
    \]
    Since $p$ splits completely in $F^1$, the identifications
    $F_v^1\cong\Q_p$ induce a natural surjection
    $R_{F^1,p}^\ps\twoheadrightarrow R_p^\ps$.
    
    Let $R_\Q^{\operatorname{univ}}$ be the universal deformation ring of $\bar{\rho}$ with determinant $\chi$. Then by the universal property, we get a natural map $f: R_{F^1}^{\operatorname{univ}} \to R_\Q^{\operatorname{univ}} $.
    
    \begin{prop}\label{finiteness 3}
    	The map $f$ induces a well-defined finite map $(R_{F^1}^{\operatorname{univ}})^{\operatorname{red}} \to (R_\Q^{\operatorname{univ}})^{\operatorname{red}}$.
    \end{prop}
    
    \begin{proof}
    	By \cite[Proposition 3.26]{Gee_2022}, the map $f$ is finite. Then our result is clear.
    \end{proof}
    
    The universal property also gives the following analogue of Lemma \ref{com dia2}.
    
    \begin{lem}\label{com dia irr}
    	We have the following diagram:
    	\[\begin{tikzcd}
    		R_{F^1,p}^\ps  \arrow[r] \arrow[d, two heads]& (R_{F^1}^{\operatorname{univ}})^{\operatorname{red}} \arrow[d]\\
    		 R_{p}^\ps \arrow[r] & (R_\Q^{\operatorname{univ}})^{\operatorname{red}} .
    	\end{tikzcd}\]
    \end{lem}
    
    Combining these results, we can prove a finiteness result in the residually irreducible case.
    
    \begin{prop}\label{finiteness irr}
    	The natural map $R_p^\ps\to R_\Q^{\operatorname{univ}}$ is finite.
    \end{prop}
    
    \begin{proof}
    	By Theorem \ref{potential pro-modularity irr} and Proposition
    	\ref{finiteness from lgc}, the map
    	\(
    	R_{F^1,p}^\ps\to
    	(R_{F^1}^{\operatorname{univ}})^{\operatorname{red}}
    	\cong(\T_\km)^{\operatorname{red}}
    	\)
    	is finite. Combining this with Proposition \ref{finiteness 3},
    	we see that the composite map
    	$R_{F^1,p}^\ps\to(R_\Q^{\operatorname{univ}})^{\operatorname{red}}$
    	is finite. By Lemma \ref{com dia irr}, it factors through
    	$R_p^\ps$. Thus
    	$R_p^\ps\to(R_\Q^{\operatorname{univ}})^{\operatorname{red}}$
    	is finite, and Proposition \ref{finite-after-reduction} gives
    	the result.
    \end{proof}
    
    \subsection{Mazur's dimension conjecture} In this subsection, we prove Mazur's dimension conjecture in the residually irreducible case. We keep notations as in the previous subsection.
    
     \begin{thm}\label{dim in irr case}
    	 The ring $R_\Q^{\operatorname{univ}}$ is a local complete intersection ring of dimension $3$.
    
    \end{thm}
    
    \begin{proof}
    	Note that by the classical Galois deformation theory and the global characteristic formula (see \cite[Proposition 6.1.4]{Allen_2019} for example), we only need to show $\dim R_\Q^{\operatorname{univ}}=3$. 
    	
       We first assume that $\Sigma\setminus\{p\}$ is not empty.
       After some twist, we may assume that $\chi$ is crystalline at $p$.
       For each $l\in\Sigma\setminus\{p\}$, choose an integer
       $n_l\geq\operatorname{Sw}_l(\bar\rho)+2$. For any characteristic-zero
       specialization $\rho$ of $R_\Q^{\operatorname{univ}}$, the Swan
       conductor is unchanged by reduction, and the tame contribution
       to the Artin conductor is at most $2$. Thus
       $a_l(\rho)\leq\operatorname{Sw}_l(\bar\rho)+2$.
       Choose a tame level $K^p=\prod_{l\ne p}K_l$ with
       $K_l=\GL_2(\Z_l)$ for $l\notin\Sigma$ and
       \[
       K_l\subseteq K_1(l^{n_l})
       :=\left\{g\in\GL_2(\Z_l):
       g\equiv\begin{pmatrix}*&*\\0&1\end{pmatrix}\bmod l^{n_l}\right\}
       \quad(l\in\Sigma\setminus\{p\}).
       \]
       Shrinking these factors if necessary, we may assume that
       $K^p\GL_2(\Z_p)$ is sufficiently small and the level is compatible
       with the fixed central character. This choice ensures
       that all automorphic specializations considered below occur at
       the same tame level $K^p$. As in Section \ref{sec for modular curve},
       we define the big Hecke algebra $\T$ with the central character
       corresponding to $\chi$ and obtain a natural surjection
       $R_\Q^{\operatorname{univ}}\twoheadrightarrow\T$.
       The non-Eisenstein maximal ideal $\km$ corresponding to
       $\bar\rho$ exists by Serre's modularity theorem
       \cite{Khare_2009a}.
       
       Let $\kp$ be a minimal prime of $ R_\Q^{\operatorname{univ}}$. By Proposition \ref{at least 3}, $\kp$ is of dimension at least $3$. By Proposition \ref{finiteness irr}, we get a finite map $R_p^\ps \to R_\Q^{\operatorname{univ}}/\kp$. Then by Proposition \ref{local pseudo}, $\kp$ is of dimension at most $4$.
       
       Assume that $\dim R_\Q^{\operatorname{univ}}/\kp=4$.
       Since $R_p^\ps$ is a four-dimensional integral domain by
       Proposition \ref{local pseudo}, the finite map
       $R_p^\ps\to R_\Q^{\operatorname{univ}}/\kp$ is injective.
       Put
       \(
       I=\operatorname{ker}\{R_p^\ps\to\T\}.
       \)
       We show that $I=0$. Let $\kq$ be an irreducible regular
       crystalline prime of $R_p^\ps$. By the going-up property, there is a prime
       $\kq'$ of $R_\Q^{\operatorname{univ}}/\kp$ above $\kq$.
       This is a one-dimensional characteristic-zero prime, and its
       associated representation
       $\rho(\kq'):G_{\Q,\Sigma}\to\GL_2(k(\kq'))$ is irreducible,
       since $\bar\rho$ is absolutely irreducible.
       By the Brauer--Nesbitt theorem
       \cite[Fact 2.5]{Gee_2022}, its restriction to $G_{\Q_p}$ is
       irreducible and crystalline with distinct Hodge--Tate weights.
       Hence $\rho(\kq')$ is automorphic by
       \cite[Theorem 1.0.2]{Zhang:2024aa}. By our choice of $K^p$,
       the corresponding eigensystem occurs in $\T$. It follows that
       $I\subseteq\kq$.
       
       Now Proposition \ref{density of crys} says that these local primes
       $\kq$ are Zariski dense in $\Spec R_p^\ps$. Since $R_p^\ps$ is
       reduced, their intersection is zero, and hence $I=0$.
       On the other hand, Proposition \ref{nonzero modular kernel}
       gives $I\ne0$, since $\km$ is non-Eisenstein.
       This contradiction rules out a four-dimensional component.
       
       If $\Sigma \setminus \{p\}$ is empty, the result follows from the surjection $R_{\Q, \Sigma'}^{\operatorname{univ}} \twoheadrightarrow R_\Q^{\operatorname{univ}}$, where $R_{\Q, \Sigma'}^{\operatorname{univ}}$ is the universal deformation ring of $\bar{\rho}:G_{\Q, \Sigma'} \to \GL_2(\F)$ with ramification set $\Sigma'$ larger than $ \{p\}$ and fixed determinant $\chi$.
    \end{proof}
    
    Let $R_\Q$ be the universal deformation ring of $\bar{\rho}$ for $G_{\Q, \Sigma}$ without fixed determinant.
    
    \begin{thm}\label{mazur non-eisenstein}
    	The ring $R_\Q$ is a local complete intersection ring of dimension $4$.
    \end{thm}
    
    \begin{proof}
    	This follows from the combination of Theorem \ref{dim in irr case}, Proposition \ref{lci} and \cite[Lemma 6.1.2]{Allen_2019}.
    \end{proof}
    
    \begin{rem}\label{rem:nonEis-dimension-method}
       The dimension argument above takes a different route from the big $R=\T$ methods of B\"ockle and Emerton \cite{B_ckle_2001,emerton11}. It does not require a big $R=\T$ theorem over $\Q$ as an input, but instead combines local-to-global finiteness with $p$-adic Langlands theory and Fontaine--Mazur results. Thus it still relies essentially on modularity-theoretic methods. The advantage is that, under our standing assumptions, neither a Taylor--Wiles hypothesis nor an additional local genericity assumption is required.

       On the other hand, this argument by itself establishes neither a big $R=\T$ theorem nor $\cO$-flatness in this generality. Using ordinary finiteness, we prove $\cO$-flatness in Corollary \ref{cor:nonEis-flat} when $\bar\rho|_{G_{\Q_p}}$ is reducible. The locally absolutely irreducible case is proved in Corollary \ref{cor:nonordinary-flat}.
    \end{rem}
    
    \subsection{Ordinary finiteness and flatness}\label{app:ordinary-finiteness}
    Assume that the local residual representation is reducible. We prove character-labelled ordinary finiteness without a distinguishedness assumption, using patching and connectedness after a soluble base change as in \cite[Sections 5 and 6]{pan2022fontaine}, followed by descent to $\Q$. We then exclude characteristic-$p$ components of the unrestricted global ring to obtain $\cO$-flatness.

    After enlarging $\F$, fix a residual subcharacter $\bar\psi_1$ of $\bar\rho|_{G_{\Q_p}}$ and put $\bar\psi_2=(\det\bar\rho)|_{G_{\Q_p}}\bar\psi_1^{-1}$. We call the local residual representation \textit{distinguished} if $\bar\psi_1\ne\bar\psi_2$ as characters of $G_{\Q_p}$, and \textit{non-distinguished} otherwise. Let $S_p$ be the universal deformation ring of $\bar\psi_1$. Starting with $R_\Q^{\operatorname{univ}}\widehat\otimes_\cO S_p$, let $R_\Q^{\ord,\bar\psi_1}$ be the quotient defined by $\tr\rho|_{G_{\Q_p}}=\psi_1+\psi_2$ and
    \[
       (\rho(\sigma)-\psi_1(\sigma)I_2)(\rho(\tau)-\psi_2(\tau)I_2)=0,
       \qquad \sigma,\tau\in G_{\Q_p},\quad \psi_2=\chi\psi_1^{-1}.
    \]
    Here $\rho$ and $\psi_1$ denote the universal representation and character. In the distinguished case it is a quotient of $R_\Q^{\operatorname{univ}}$, since the prescribed residual character determines its lifting whenever the ordinary condition holds. By local class field theory, its restriction to principal units defines a natural map $\Lambda_\Q:=\cO[[\Z_p^\times(p)]]\to R_\Q^{\ord,\bar\psi_1}$.

    We use the same definitions over a finite totally real extension $F/\Q$ for which $\bar\rho|_{G_F}$ remains absolutely irreducible, with a specified residual subcharacter $\bar\psi_{v,1}$ at each $v|p$. Write $R_F^{\ord}$ for the ordinary deformation ring, and $\Lambda_F=\widehat\otimes_{v|p}\cO[[\cO_{F_v}^{\times}(p)]]$ for its Iwasawa algebra. 
    

    \begin{thm}\label{thm:ordinary-finiteness-nonEis}
       The ring $R_\Q^{\ord,\bar\psi_1}$ is finite over $\Lambda_\Q$. 
    \end{thm}

    \begin{proof}
       We follow the strategy in \cite{pan2022fontaine} and \cite{skinnerwiles01}. 

       \textit{Step 1. Ordinary modular input and base change.}
       
       By Serre's modularity theorem \cite{Khare_2009a}, $\bar\rho$ is modular. We choose a regular nearly ordinary modular lift $\rho_0$ whose stable-line character at $p$ lifts $\bar\psi_1$. In the distinguished case, use the choice of a $\chi_2$-good lift in the proof of \cite[Theorem 5.2]{skinnerwiles01}, with $\chi_2=\bar\psi_2$. In the non-distinguished case, \cite[Section 1.1]{BLGG_2012}, after a finite-order twist if necessary, supplies such a lift; both residual characters are equal and we do not prescribe a residual line.

       Put $\chi_0=\det\rho_0$. We may enlarge $\Sigma$ to contain the ramification of $\rho_0$, since the original ordinary ring is a quotient of the enlarged one. Twisting by the unique square root of $\chi\chi_0^{-1}$ lifting $\mathbf{1}$ identifies the ordinary deformation problem with determinant $\chi_0$ with that with determinant $\chi$, preserving the residual character on the stable line. The twist changes the weight map by an automorphism of $\Lambda_\Q$, so it suffices to prove finiteness for determinant $\chi_0$.

       Using \cite[Lemma A.2.1]{BLGGT_2014}, choose a finite soluble totally real Galois extension $F_0/\Q$, linearly disjoint from the field cut out by $\bar\rho$ and $\omega$, with even degree and the following local properties. At $p$ its completions are unramified of degree at least two; we can prescribe an unramified extension of even degree here to ensure $[F_0:\Q]$ is even. At the places above $\Sigma\setminus\{p\}$, the residual representation is trivial, the determinant is unramified, the residue cardinality is $1$ modulo $p$, and $\rho_0$ has unipotent inertia. We allow Iwahori level, not necessarily spherical level, at these places. Soluble base change of $\rho_0$ gives a cuspidal ordinary modular lift over $F_0$, since the residual restriction remains absolutely irreducible. 

       Write $\Sigma_F$ and $\Sigma_{F,p}$ for the places of $F$ above $\Sigma$ and $p$, respectively, and put $d=[F:\Q]$, $s=\#(\Sigma_F\setminus\Sigma_{F,p})$ and $t=\#\Sigma_{F,p}$. In addition to the local properties above, we require
       \begin{itemize}
       	\item $d$ is even,
       	\item $ [F_v:\Q_p]\ge2$ for every $v|p$,
       	\item  $d \ge 4s+3$.
       \end{itemize}
       
       The first two conditions are used for quaternionic Hida theory and the local normality results; the last is a sufficient bound for the connectedness argument in Step 3. To obtain it, choose an odd prime $\ell\ne p$ outside $\Sigma$ and prime to the degree of a finite Galois field containing $F_0$ and the fields cut out by $\bar\rho$ and $\omega$. Let $\Q_n$ be the degree-$\ell^n$ layer of the cyclotomic $\Z_\ell$-extension and take $F=F_0\Q_n$. Each fixed tame prime has an open decomposition group in this tower, since its norm is not an $\ell$-adic root of unity. Thus $s$ stays bounded while $d=[F_0:\Q]\ell^n$ tends to infinity. The disjointness and local properties persist. The prime $\ell$ need not be added to $\Sigma$, since restriction preserves unramifiedness there. Fix $n$ large enough for the displayed bound.
       
       At each $v|p$, let $\bar\psi_{v,1},\bar\psi_{v,2}$ be the restricted residual characters. As $F/\Q$ is Galois and $\bar\rho$ comes from $G_\Q$, the local residual pairs at all $v|p$ are conjugate. Thus they are either distinguished at every $v|p$ or non-distinguished at every $v|p$.

       At each tame place $v$, fix a $p$-power character $\xi_v:k(v)^\times\to\cO^\times$ and impose the inertial trace $\xi_v+\xi_v^{-1}$. Write $R^{\ord,\xi}_F$ for the resulting quotient of $R_F^{\ord}$. The lift $\rho_0$ gives a nonzero Hida localization for the trivial type. Since $\xi_v\equiv1\pmod\varpi$, reduction identifies all these Hida modules at a common sufficiently small level, preserving the Hecke operators and specified ordinary characters. Thus the localizations are nonzero for every $\xi$. By \cite[Section 5.2 and Corollary 5.2.1]{pan2022fontaine}, their ordinary Hecke algebras $\T^{\ord,\xi}_F$ are finite over $\Lambda_F$ and equidimensional of dimension $d+1$. The Galois representation and ordinary trace identities give a surjection $R^{\ord,\xi}_F\twoheadrightarrow\T^{\ord,\xi}_F$. A prime is \textit{pro-modular} if it contains this kernel, and a component is pro-modular if its generic point is.

       \textit{Step 2. Patching at ordinary nice primes.} 
       
       Fix a tame type $\xi$. Our goal is the following localized statement, analogous to \cite[Proposition 5.3.3 and Proposition 6.3.1]{pan2022fontaine}. Call a prime $\kq$ of $\T^{\ord,\xi}_F$ \textit{nice} if $\varpi\in\kq$, $\dim\T^{\ord,\xi}_F/\kq=1$, and $\rho(\kq)|_{G_{F_v}}$ is the constant trivial representation at every tame place $v\in\Sigma_F\setminus\Sigma_{F,p}$. See \cite[5.3.4 \& 6.3.2]{pan2022fontaine} for the precise definition. We claim that
       \[
          (R^{\ord,\xi}_F)_{\kq\cap R^{\ord,\xi}_F}\longrightarrow
          (\T^{\ord,\xi}_F)_\kq
          \quad\text{has nilpotent kernel.}
       \]

       Thus every component through this point is pro-modular. We give the local and auxiliary-prime inputs, then deduce the claim from full support of the patched module.
       
       We explain the necessary changes to Pan's argument.

       \textit{Local input.} Write $R_v^\Delta$ for the character-labelled ordinary framed deformation ring at $v|p$. If the residual characters are distinguished, \cite[Lemma 5.3.4]{pan2022fontaine} gives an $\cO$-flat normal domain of dimension $1+2[F_v:\Q_p]+3$, since $[F_v:\Q_p]\ge2$. If they are non-distinguished, then \cite[Lemma 6.2.2]{pan2022fontaine} gives the same normality, flatness and dimension results. The same local descriptions apply after every finite extension of the coefficient DVR; in particular, these rings are geometrically integral over $\cO$.

       \textit{Auxiliary Taylor-Wiles primes.} The Galois-theoretic input is an ordinary analogue of \cite[Proposition 4.3.1]{pan2022fontaine}, which we formulate explicitly. Put $B=\T^{\ord,\xi}_F/\kq$ and let $A=\F'[[z]]$ be its normalization, with fraction field $K'$. We may enlarge the coefficient field so that $A$ and $B$ have the same residue field. Let $\rho^\circ$ be the representation over $B$ furnished by residual absolute irreducibility, and write $\rho=\rho^\circ\otimes_BK'$. The traces away from $\Sigma_F$ and the retained ordinary characters generate $B$, by the definition of the Hecke algebra in \cite[Section 5.2]{pan2022fontaine}. Choose a finite set $T'$ of those places whose traces suffice, and put $P=\Sigma_F\cup T'$. Let $R_{\loc}^{\ord,\xi}$ be the completed tensor product of the rings $R_v^\Delta$ at $p$, the fixed-type framed rings at the tame places, and the unramified framed rings at $T'$ (see \cite[Definition 4.2.1]{pan2022fontaine}). For $Q$ disjoint from $P$, let $R_Q$ be the corresponding global $P$-framed ordinary deformation ring, allowing additional ramification at $Q$. Evaluation at $(\rho^\circ;1)_{v\in P}$ defines primes $\kq_Q\subset R_Q$ and $\kq_{\loc}\subset R_{\loc}^{\ord,\xi}$ with both prime quotients equal to $B$. Set $r=\dim_{K'}H^1(G_{F,P},\ad^0\rho(1))$ and $g=r+|P|-d-1$. The sets $T'$ and $Q$ are auxiliary to patching at this point; they do not change the field or the set $\Sigma_F$ used in Step 3.

       \textit{Auxiliary-prime claim.} There is a constant $C$ such that, for every $N>0$, there is a finite set $Q_N$ disjoint from $P$ satisfying the following properties:
       \begin{enumerate}
          \item[(i)] $|Q_N|=r$ and $\mathrm{N}(v)\equiv1\pmod{p^N}$ for every $v\in Q_N$.
          \item[(ii)] The eigenvalues $\alpha_v,\beta_v$ of $\rho(\Frob_v)$ are distinct and $\ell_A(A/(\alpha_v-\beta_v)^2)<C$.
          \item[(iii)] Writing $D_N=\kq_{Q_N}/(\kq_{Q_N}^2+\kq_{\loc}R_{Q_N})$, we have $D_N\otimes_BA\cong A^g\oplus M_N$ with $\ell_A(M_N)<C$.
          \item[(iv)] There is a map $R_{\loc}^{\ord,\xi}[[x_1,\ldots,x_g]]\to R_{Q_N}$ with $x_i\in\kq_{Q_N}$ such that $D_N/(Bx_1+\cdots+Bx_g)$ is killed by a nonzero $f\in B$ independent of $N$.
       \end{enumerate}
       Thus the claim gives all five conclusions of \cite[Proposition 4.3.1]{pan2022fontaine}; in (iv), after increasing $C$, we also have $\ell_A(A/fA)<C$. 

       We prove the claim by explaining the changes to \cite[Section 4.3]{pan2022fontaine}. First, a nice prime satisfies the Taylor--Wiles condition at the point: $\rho|_{G_{F(\zeta_p)}}$ is absolutely irreducible. Indeed, at least one ordinary character at some $v|p$ has infinite image on principal units. Otherwise the image of $\Lambda_F$ in $B$ would have dimension zero, contrary to finiteness and $\dim B=1$. Since the determinant has finite image modulo $\varpi$, the ratio of the two local characters has infinite inertial image at that place. If $\rho|_{G_{F(\zeta_p)}}$ were reducible, Clifford theory would make $\rho$ induced from the quadratic subextension of $F(\zeta_p)/F$. This extension is ramified at every $v|p$, since $p$ is unramified in $F$, and the local induction could be reducible only if its inducing character extended to $G_{F_v}$. Its two characters would then differ by the local quadratic character, a contradiction. This is a condition on $\rho$, not an additional Taylor--Wiles hypothesis on $\bar\rho$, and it does not exclude all dihedral points.

       The first change concerns the local deformation conditions. In the relative tangent calculation of \cite[Lemma 4.3.2]{pan2022fontaine}, the framed local representations and the ordinary characters are fixed by $R_{\loc}^{\ord,\xi}$. Writing a first-order global deformation as $(1+\phi e)\rho^\circ$ and its framing changes as $1+m_ve$, the condition at $v\in P$ is still $\phi|_{G_{F_v}}=\partial m_v$. Hence the local cohomological condition is zero at $P$ and unrestricted at $Q$. The same calculation (using Greenberg-Wiles formula) and Poitou--Tate duality give $g=r+|P|-d-1$. Replacing the unrestricted rings at $p$ by ordinary rings changes the dimension of the local base, but not this relative calculation, even when the two local residual characters coincide.

       The second change is in the dihedral parts of \cite[Lemmas 4.3.5 and 4.3.6]{pan2022fontaine}. Put $W=\ad^0\rho^\circ(1)\otimes_BA$, $V=W\otimes_AK'$, $\Gamma=\im(\rho,\varepsilon)$ and $\widetilde G=\ker\rho\cap\ker\varepsilon$, where $\varepsilon$ is the $p$-adic cyclotomic character. The required assertions are that $H^1(\Gamma,W)$ has finite $A$-length and that every non-torsion class in $H^1(G_{F,P},W)$ is detected by some $\sigma\in\ker\varepsilon$ with $\rho(\sigma)$ having distinct eigenvalues. Here detection means that the class has nonzero image in $V/(\sigma-1)V$. For non-dihedral $\rho$, we use the non-dihedral cases of these two lemmas, which are also the cohomological inputs to \cite[Proposition 9.0.6]{pan2022fontaine}. In the dihedral case, however, Pan uses residual reducibility to identify the inducing quadratic field and prove a cyclotomic disjointness assertion for the field cut out by $\rho$. We replace that argument as follows; only finite length, not vanishing, of $H^1(\Gamma,W)$ is needed.

       For these two assertions we may make a finite extension of $K'$ and replace $A$ by its integral closure. Write $\rho=\Ind_{G_L}^{G_F}\theta$ and $\eta=\theta/\theta^\tau$, where $\tau$ is a lift of the nontrivial automorphism of $L/F$. The pointwise Taylor--Wiles condition gives $L\cap F(\zeta_p)=F$, hence also $L\cap F(\zeta_{p^\infty})=F$ since $p>2$. The infinite ordinary character ratio exhibited above implies that $\rho$ has infinite projective image. Thus $\eta$ has infinite order, and $V=\delta_L(1)\oplus(\Ind_{G_L}^{G_F}\eta)(1)$ is a sum of two irreducible constituents, where $\delta_L$ is the quadratic character of $L/F$. The image $H$ of $G_L$ in $\Gamma$ is abelian of index two, with characters $\omega,\eta\omega,\eta^{-1}\omega$ on $V$. Since $\omega|_{G_L}\ne1$ and $\eta$ has infinite image on $\ker\omega$, we can choose $h\in H$ such that $h-1$ is invertible on $V$. For a cocycle $c$ on $H$, we have $(h-1)c(g)=(g-1)c(h)$. Evaluation at $h$ therefore embeds $Z^1(H,W)$ into $W$, so $H^1(H,W)$ is finitely generated over $A$. Multiplication by the adjugate of $h-1$ shows that $\det(h-1)c$ is a coboundary. Thus $H^1(H,W)$ has finite $A$-length. Restriction from $\Gamma$ to $H$ is injective on $H^1$, because $2$ is invertible, so the same holds for $H^1(\Gamma,W)$.

       Consequently a non-torsion global class $[\phi]$ restricts nontrivially to $\Hom_\Gamma(\widetilde G,W)$, and the $K'$-span of its image contains one of the two irreducible constituents of $V$. To detect $\delta_L(1)$, choose $g\in G_L$ with $\eta(g)^2\ne1$ and put $\sigma_0=[\tau,g]$. Then $\sigma_0\in G_L\cap\ker\varepsilon$ and $\eta(\sigma_0)=\eta(g)^{-2}\ne1$, so $\rho(\sigma_0)$ has distinct eigenvalues and its invariant line on $V$ is $\delta_L(1)$. To detect the induced constituent, choose $\sigma_0\in\ker\varepsilon\setminus G_L$, using $L\cap F(\zeta_{p^\infty})=F$. The matrix $\rho(\sigma_0)$ has trace zero and distinct eigenvalues, and its invariant line on $V$ lies in the induced constituent. In either case, since $\phi(\sigma_0u)\equiv\phi(\sigma_0)+\phi(u)\pmod{(\sigma_0-1)V}$ for $u\in\widetilde G$, a suitable $u$ makes the evaluation nonzero. The element $\sigma=\sigma_0u$ has the same image under $(\rho,\varepsilon)$ as $\sigma_0$. This proves the two required assertions, which descend through the finite coefficient extension.

       \textit{Uniform bounds with finite coefficients.} Applying detection successively to the kernels of the preceding evaluation maps gives $\sigma_1,\ldots,\sigma_r\in\ker\varepsilon$ for which
       \[
          H^1(G_{F,P},W)\longrightarrow\bigoplus_{i=1}^r W/(\sigma_i-1)W
       \]
       has finite-length kernel and cokernel, as in \cite[Lemma 4.3.4]{pan2022fontaine}. Fix these elements and choose $M$ killing the kernel and cokernel by $z^M$, large enough also to preserve the nonzero discriminants under congruences modulo $z^{M+1}$. Since $\sigma_i\in\ker\varepsilon$, Chebotarev allows us, for every $N$, to choose distinct places $v_i$ with $\mathrm{N}(v_i)\equiv1\pmod{p^N}$ whose Frobenius matrices and evaluations on a finite set of cocycles generating $H^1(G_{F,P},W)$ agree with those of $\sigma_i$ modulo $z^{M+1}$. As in \cite[Section 4.3.5]{pan2022fontaine}, Nakayama's lemma shows that the resulting evaluation map has kernel and cokernel of uniformly bounded length. This gives (i) and (ii), with constants independent of $N$.

       Put $W_n=W/z^nW$. The exact sequence
       \[
          0\longrightarrow H^1(G_{F,P},W)/z^n\longrightarrow H^1(G_{F,P},W_n)
          \longrightarrow H^2(G_{F,P},W)[z^n]\longrightarrow0
       \]
       controls the passage to finite coefficients, since the last term has uniformly bounded length. Define
       \[
          e_N:H^1(G_{F,P},W)\longrightarrow\bigoplus_{v\in Q_N}W/(\Frob_v-1)W,
          \qquad [c]\longmapsto\bigl([c(\Frob_v)]\bigr)_{v\in Q_N}.
       \]
       By local Tate duality, the dual Selmer group in $H^1(G_{F,P\cup Q_N},W_n)$ has no condition at $P$ and zero restriction at $Q_N$. Its classes are thus unramified at $Q_N$, so inflation identifies it with $\ker e_{N,n}$, where $e_{N,n}$ is the evaluation above with $W$ replaced by $W_n$ (cf. \cite[Section 4.3.4]{pan2022fontaine}). On $H^1(G_{F,P},W)/z^n$, this is $e_N\bmod z^n$. Factoring $e_N$ through its image gives $\ell_A\ker(e_N\bmod z^n)\le\ell_A\ker e_N+\ell_A((\operatorname{coker}e_N)[z^n])$, with the second term arising from $\operatorname{Tor}_1^A(A/z^n,\operatorname{coker}e_N)$. Adding $\ell_A(H^2(G_{F,P},W)[z^n])$ bounds the dual Selmer length independently of $N$ and $n$.

       The pointwise Taylor--Wiles condition also gives $H^0(G_F,V)=0$, by Schur's lemma after restriction to $G_{F(\zeta_p)}$. Hence $\ell_A(H^0(G_F,W_n))$ is bounded independently of $n$. At each $v\in Q_N$, the distinct Frobenius eigenvalues give $\ell_A(H^0(G_{F_v},W_n))=n+O(1)$ uniformly in $N,n$: in Smith normal form, the sum of the two nonzero elementary-divisor valuations of $\Frob_v-1$ on $W$ is at most the valuation of $(\alpha_v-\beta_v)^2$. No integral diagonalization is needed here. The relative Poitou--Tate calculation above now gives $\ell_A((D_N\otimes_BA)/z^n)=gn+O(1)$ with a uniform error. The structure theorem for finitely generated $A$-modules proves (iii). For (iv), choose $c$ with $z^cA\subset B$. A basis of the free summand in (iii), multiplied by $z^c$, lifts to $D_N$, since the images of $D_N$ generate $D_N\otimes_B A$. Lift those elements further to $x_i\in\kq_{Q_N}$ and let $U=D_N/\sum_i Bx_i$. Then $U\otimes_B A$ is killed by $z^{\max(c,C)}$, by (iii). The kernel of $U\to U\otimes_B A$ is killed by $z^c$, using the $B$-linear map $A\to B$ given by multiplication by $z^c$. Thus $f=z^{c+\max(c,C)}\in B$ works for all $N$.

       \textit{Patching and full support.} The auxiliary-prime claim now supplies the input to \cite[Sections 4.4--4.8]{pan2022fontaine}. We use its patching construction with completed homology replaced by Hida modules and the local rings at $p$ replaced by $R_v^\Delta$. The ordinary projector commutes with the auxiliary-level maps, so the diamond-group flatness, augmentation maps and bounded-error degeneracy maps pass to the ordinary summands. 
       We retain the symmetric diamond algebra of \cite[Section 4.5.1]{pan2022fontaine}, since the eigenvalues at $\kq$ need not be distinct residually. In the comparison of \cite[Corollary 4.6.9]{pan2022fontaine}, the global trace recovers the representation by residual absolute irreducibility. With the ordinary characters retained on both sides, the trace identities of \cite[Definition 5.3.1]{pan2022fontaine} are equivalent to the matrix identities above: one pairs each matrix product with $\rho(\eta)$ for all $\eta\in G_F$, and these matrices span the full matrix algebra by residual absolute irreducibility and Nakayama's lemma.

       The point of the dimension calculation is to prove full support. The completed localized patched ring has dimension $2d+3|P|+g=d+4|P|+r-1$. The ordinary module at $\kq$ has dimension $d$ by \cite[Corollary 5.2.1]{pan2022fontaine}. The framing and augmentation variables form a regular sequence on the patched module by its $S_\infty$-flatness, so the support has dimension at least $d+4|P|+r-1$, as in \cite[Lemma 4.8.1]{pan2022fontaine}. If all tame characters are nontrivial, the local normality and geometric integrality results, together with \cite[Lemma 4.2.6]{pan2022fontaine}, make the patched ring a domain, so equality of dimensions gives full support. For arbitrary tame characters $\xi_v$, choose nontrivial $p$-power characters $\xi'_v$ over the same coefficient ring. At the same levels, their congruence modulo $\varpi$ identifies the reduced coefficient modules compatibly with the Hecke and diamond actions and the ordinary projector. The localized Hida module is finite and faithful over its Hecke algebra, so $\kq$, which contains $\varpi$, lies in the support of its reduction. The identification thus gives the corresponding nice prime for $\xi'$, with the same representation and labelled ordinary characters. The deformation problems agree modulo $\varpi$, including the relative cotangent quotients $D_N$ since $\varpi\in\kq_{\loc}$. We can therefore patch with the same $Q_N$ and the same parameter images modulo $\varpi$, identifying the patched modules and their ring actions modulo $\varpi$. The uniqueness of the characteristic-zero component above each special-fibre component, together with $\cO$-flatness of the patched module, transfers full support as in \cite[Sections 4.8.3--4.8.4]{pan2022fontaine}. Finally, the augmentation argument of \cite[Lemma 4.8.4]{pan2022fontaine} gives the nilpotent-kernel assertion needed in Step~3.

       \textit{Step 3. Propagation between components.} 
       
       Put $R=R^{\ord,\xi}_F$ and $J=\ker(R\to\T^{\ord,\xi}_F)$. To prove that $J$ is nilpotent, we first cut out nice primes to obtain one pro-modular component, then propagate pro-modularity to all components using connectedness.

       \textit{The dimension budget.} Keep $d,s,t$ as in Step 1 and put $h=\dim_\F H^0(G_{F,\Sigma_F},\ad^0\bar\rho(1))$. By Schur's lemma $h\le1$; if $h\ne0$, then $\bar\rho\cong\bar\rho\otimes\omega$, so $\omega^2=1$. Since $p$ is unramified in $F$, this forces $p=3$. The global-over-local presentation of \cite[Section 5.4.2]{pan2022fontaine} applies to the actual representation, with generator-minus-relation bound $|\Sigma_F|-1-d-h$. The ordinary characters are retained in both the global and local problems, so imposing the ordinary condition is base change on the local rings and leaves this relative bound unchanged. Each unrestricted tame framed ring has generator-minus-relation difference $3$, and fixing its inertial trace costs one equation. Removing $4|\Sigma_F|-1$ global framing parameters, as in the proof of \cite[Lemma 5.4.2]{pan2022fontaine}, gives
       \[
          R\cong R_p[[X_1,\ldots,X_a]]/(f_1,\ldots,f_b),\qquad
          a-b\ge-3t-s-d-h,
       \]
       where $R_p=\widehat\otimes_{v|p}R_v^\Delta$ is a domain of dimension $1+2d+3t$. Thus \cite[Proposition 5.4.4]{pan2022fontaine} gives $c(R)\ge\dim R_p+a-b-1\ge d-s-h$. Here $c(R)$ is the connectedness dimension of $\Spec R$, before reduction modulo $\varpi$. In particular, an intersection separating its irreducible components into two nonempty groups has dimension at least $d-s-h$.

       At each tame place choose a system of three parameters of the special fibre of its fixed-type framed ring, and let $\ka\subset R$ be generated by their images and $\varpi$. The corresponding local parameter quotients are Artinian with residue field $\F$, so every field-valued point of $V(\ka)$ has constant trivial representation on the full tame decomposition groups. For every prime $Q$ of $R$, the height theorem gives $\dim R/(Q+\ka)\ge\dim R/Q-(1+3s)$. Whenever $Q$ is pro-modular and this lower bound is positive, we can choose a prime above $Q+\ka$ with one-dimensional quotient; its image in the Hecke algebra is nice. Global absolute irreducibility follows from that of $\bar\rho|_{G_F}$.

       \textit{Starting propagation.} Let $Q_0$ be the inverse image of a minimal prime of $\T^{\ord,\xi}_F$. Then $\dim R/Q_0=d+1$, so cutting by $\ka$ leaves dimension at least $d-3s\ge1$. Choose a nice prime as above. Step 2 makes every component through this point pro-modular; in particular, at least one such component exists. We do not assume that $Q_0$ itself is a minimal prime of $R$.

       \textit{Reaching every component.} Partition the components into the pro-modular ones and the others. If both unions are nonempty, connectedness supplies a prime $Q$ in their intersection with $\dim R/Q\ge d-s-h$. This prime contains $J$, since it lies on a pro-modular component. Cutting by $\ka$ leaves dimension at least $d-4s-h-1$, which is positive provided $d\ge4s+h+2$. This explains the sufficient uniform choice $d\ge4s+3$ in Step 1; if $h=0$, the weaker bound $d\ge4s+2$ suffices. The resulting nice prime still lies on a component from each group, but Step 2 makes every component through it pro-modular, a contradiction. Thus $J$ is nilpotent. Since $R/J$ is finite over $\Lambda_F$, so is $R$, by Proposition \ref{finite-after-reduction}. 

       \textit{Step 4. Forgetting tame types and descent.} 
       
       We now remove $\xi$ and prove finiteness of the restriction map to the ordinary ring over $\Q$.

       \textit{Removing the tame types.} Only finitely many types occur. Indeed, residual tame triviality kills wild and prime-to-$p$ tame inertia in any deformation. Frobenius permutes the two eigenvalues of a pro-$p$ tame generator by raising them to $q_v=\mathrm{N}(v)$. Since the determinant is unramified, these eigenvalues are inverse to each other. As $q_v\equiv1\pmod p$ and $p>2$, their orders divide the $p$-part of $q_v-1$; in characteristic $p$ both are $1$. After enlarging coefficients, every prime of $R_F^{\ord}$ therefore factors through one of the finitely many rings $R_F^{\ord,\xi}$. Unipotent inertia is included in the trivial trace type. The map $R_F^{\ord}\to\prod_\xi R_F^{\ord,\xi}$ has nilpotent kernel, and its image is finite over the Noetherian ring $\Lambda_F$. Proposition \ref{finite-after-reduction} gives finiteness of $R_F^{\ord}$ over the Iwasawa algebra $\Lambda_F$. 

       \textit{Descent to $\Q$.} By the universal property, the restriction, including the specified characters, gives $R_F^{\ord}\to R_\Q^{\ord,\bar\psi_1}$. To prove that this map is finite, by Nakayama's lemma, it suffices to show that its closed fibre $D=R_\Q^{\ord,\bar\psi_1}/\km_{R_F^{\ord}}R_\Q^{\ord,\bar\psi_1}$ is Artinian. Over $D$, strict conjugation makes the restricted representation on $G_F$ equal to $\bar\rho|_{G_F}$. Since $G_F$ is normal in $G_\Q$, the matrix $\bar\rho(g)^{-1}\rho_D(g)$ centralizes $\bar\rho(G_F)$ for every $g\in G_\Q$. Absolute irreducibility gives $\rho_D(g)=a(g)\bar\rho(g)$ with $a(g)\in1+\km_D$. Fixed determinant implies $a(g)^2=1$, hence $a(g)=1$ because $p>2$. Thus the underlying global representation is constant on the entire closed fibre, including its infinitesimal part.

       The remaining data are the ordinary character values. The ratio $\psi_1\bar\psi_1^{-1}$ restricts trivially to $G_{F_v}$ for a chosen $v|p$, so it factors through the abelianization of the finite group $\Gal(F_v/\Q_p)$. Its character parameter ring is finite over $\F$, and therefore $D$ is Artinian. This proves finiteness of restriction without distinguishedness. Local class field theory gives the commutative diagram
       \[\begin{tikzcd}
          \Lambda_F \arrow[r] \arrow[d,"\mathrm{N}_{F/\Q}"'] & R_F^{\ord} \arrow[d,"\mathrm{res}"] \\
          \Lambda_\Q \arrow[r] & R_\Q^{\ord,\bar\psi_1},
       \end{tikzcd}\]
       where the left vertical map is induced by $(u_v)_{v|p}\mapsto\prod_{v|p}\mathrm{N}_{F_v/\Q_p}(u_v)$ on the pro-$p$ unit groups, and the right vertical map is induced by restriction of representations and their specified characters. The top and right maps are finite, so a finite set of $\Lambda_F$-module generators of $R_\Q^{\ord,\bar\psi_1}$ is also a set of $\Lambda_\Q$-module generators. This proves the theorem; the determinant twist and the initial enlargement of $\Sigma$ are removed as in Step 1.
    \end{proof}

    \begin{rem}\label{rem:ordinary-nonEis-comparison}
       Residual absolute irreducibility simplifies several parts of the argument in \cite[Sections 5 and 6]{pan2022fontaine}. Once the base change preserves residual absolute irreducibility, every specialization remains absolutely irreducible. Thus there is no global reducible locus to control when cutting out nice primes or propagating pro-modularity. This also simplifies the choice of base change, since the additional requirements used to control the reducible locus are unnecessary here.

       Moreover, Serre's modularity theorem and the ordinary lifting results over $\Q$ used in Step 1 provide the initial ordinary modular representation, which can then be base changed to $F$. We therefore do not need the Eisenstein congruence and special $L$-value arguments used to construct the ordinary modular input in the residually reducible case; see \cite[Section 3.4]{Skinner_1999} and \cite[Section 5.6]{pan2022fontaine}.
    \end{rem}

    \begin{cor}\label{cor:nonEis-flat}
       If $\bar\rho|_{G_{\Q_p}}$ is reducible, then both $R_\Q^{\operatorname{univ}}$ and $R_\Q$ are $\cO$-flat.
    \end{cor}

    \begin{proof}
       Put $R=R_\Q^{\operatorname{univ}}$ and $P=R_p^\ps$. By Theorem \ref{dim in irr case}, $R$ is a local complete intersection of dimension $3$, and Proposition \ref{finiteness irr} gives a finite map $P\to R$. Since $R$ has no embedded associated primes, our goal is to show that no minimal prime contains $\varpi$. Suppose otherwise and put $D=R/\kq$ for such a minimal prime $\kq$, so $D$ is a domain of dimension $3$. In each local case we construct a domain quotient $E$ of a character-labelled ordinary ring modulo $\varpi$ with $\dim E\ge2$. This contradicts Theorem \ref{thm:ordinary-finiteness-nonEis}, since $\dim\Lambda_\Q/\varpi=1$.

       \textit{The distinguished case.} The ring $P/\varpi$ is a domain of dimension $3$, by \cite[Corollary 4.29]{B_ckle_2023} and the determinant decomposition for odd $p$; here the determinant condition in that corollary is initially imposed on the local $p$-power roots of unity. Thus $P/\varpi\to D$ is finite and injective. Its reducible-trace quotient is the character deformation ring modulo $\varpi$, isomorphic to $\F[[u,v]]$. Lying over the prime defining this quotient gives a domain quotient $E$ of $D$ of dimension $2$, with local trace $\psi_1+\chi\psi_1^{-1}$. At its generic point choose the character which acts on the actual stable line. Residual multiplicity-freeness defines these characters over $E$, and the corresponding ordinary matrix identities hold over $E$ because they hold over its fraction field. Reducing these identities shows that the chosen residual character is a subcharacter of $\bar\rho|_{G_{\Q_p}}$. We therefore obtain a surjection from the appropriate ring $R_\Q^{\ord,\bar\psi_1}/\varpi$ to $E$, giving the required contradiction.

       \textit{The non-distinguished case.} Write the local residual semisimplification as $\bar\mu\oplus\bar\mu$. After a local twist, \cite[Corollary 9.13]{Pa_k_nas_2013} gives $P/\varpi\cong\F[[x_1,x_2,x_3]]$. Let $S_p$ be the local character deformation ring lifting $\bar\mu$, so $S_p/\varpi\cong\F[[u,v]]$. The character family defines a map $P\to S_p$. The ring $B=D\widehat\otimes_{P/\varpi}(S_p/\varpi)$ is a quotient of $D[[u,v]]$ by three graph relations, hence has dimension at least $2$. Choose a domain quotient $E$ of dimension at least $2$. Over its fraction field, the local representation has characters $\psi_1$ and $\chi\psi_1^{-1}$. If necessary exchange them so that the specified character is the subcharacter; both lift $\bar\mu$. The ordinary identities descend to $E$, giving a surjection $R_\Q^{\ord,\bar\mu}/\varpi\twoheadrightarrow E$. Surjectivity is preserved by the exchange, since either character determines the other. This again contradicts finiteness over $\Lambda_\Q/\varpi$.

       Thus every minimal prime of $R$ avoids $\varpi$. Since $R$ is Cohen--Macaulay, it has no embedded associated primes, so $\varpi$ is a non-zero-divisor. This proves $\cO$-flatness of $R$ itself. The determinant-character decomposition for $p>2$, whose character factor is $\cO$-flat, gives the assertion for $R_\Q$.
    \end{proof}

    \begin{rem}\label{rem:ordinary-RT-nonEis}
       The ordinary finiteness theorem also has a big $R=\T$ interpretation at the level of spectra. Suppose that $h:=\dim_\F H^0(G_{\Q,\Sigma},\ad^0\bar\rho(1))=0$, which is automatic for $p\ge5$. The usual global-over-local dimension estimate gives dimension at least $2$ for every component of $R_\Q^{\ord,\bar\psi_1}$, while Theorem \ref{thm:ordinary-finiteness-nonEis} gives the opposite bound. Thus every component is finite and dominant over $\Lambda_\Q$, and \cite[Theorem 1.0.2]{Zhang:2024aa} gives Zariski density of automorphic points. Consequently, taking a sufficiently small common tame level and retaining the same ordinary character data, we obtain a surjection to the ordinary big Hecke algebra with nilpotent kernel. No minimal ramification condition is needed. 
       
    \end{rem}

    \begin{rem}\label{rem:unramified-nonEis}
       There is also an unramified application, analogous to Remark \ref{6.1.2}. Suppose that $\bar\rho$ and $\chi$ are unramified at $p$, and let $R^{\operatorname{unr}}$ be the universal deformation ring of $\bar\rho$ for $G_{\Q,\Sigma\setminus\{p\}}$ with determinant $\chi$. The same Artinian-fibre argument shows that $R^{\operatorname{unr}}$ is finite over $\cO$. Indeed, a one-dimensional prime quotient of $R^{\operatorname{unr}}/\varpi$ yields, after a finite integral extension, unramified ordinary character data at $p$. The resulting map from the corresponding ordinary deformation ring has image containing this prime quotient, while the weight map kills the maximal ideal of $\Lambda_\Q$. Theorem \ref{thm:ordinary-finiteness-nonEis} would make this image finite over $\F$, a contradiction. This gives an unramified finiteness result in the spirit of \cite[Theorem 1]{Allen_2014}.
    \end{rem}

    \subsection{Non-ordinary finiteness and flatness}\label{app:nonordinary-finiteness}
    In this subsection, we assume that $\bar\rho|_{G_{\Q_p}}$ is absolutely irreducible. We first prove a finiteness theorem with a fixed regular $p$-adic Hodge type, in the spirit of \cite[Theorem 4.2.1]{EGPZ} using a different method, and then deduce $\cO$-flatness of the unrestricted global deformation ring. The advantage of our result here is that we do not impose a Taylor--Wiles hypothesis. The global input is the potential big $R=\T$ theorem of Section \ref{app:nonEis}; the additional local input comes from \cite{Pa_k_nas_2015,Pa_k_nas_2016}.

    Keep $\chi$, $F^1$, $U_1^p$ and $\km$ as in Section \ref{app:nonEis}. As $\bar\rho|_{G_{\Q_p}}$ is absolutely irreducible, the local pseudo-deformation ring $R_p^\ps$ is also the local deformation ring of $\bar\rho|_{G_{\Q_p}}$. To simplify our notations, we put $P=R_p^\ps$, $P_v=R_v^\ps$ for $v|p$ in $F^1$, and $A=R_{F^1,p}^\ps=\widehat\bigotimes_{v|p}P_v$, where all these local rings have fixed determinant $\chi$. Fix $w=(a,b)$ with $a>b$, and an inertial type $\tau:I_{\Q_p}\to\GL_2(K)$ compatible with $\chi$. We consider either potentially semistable or potentially crystalline representations, denoted by $\ast=\mathrm{pst}$ or $\mathrm{cr}$, respectively. Let $I_p=I_p^\ast(w,\tau)$ be the intersection of the kernels of the maps $P\to\cO_{K'}$, for finite extensions $K'/K$, whose associated representations have Hodge--Tate weights $\{a,b\}$ and Weil--Deligne inertial type $\tau$, with monodromy operator $N=0$ in the crystalline case. Assume that this type is nonempty. Thus $P/I_p$ is the usual reduced, $\cO$-torsion free integral type quotient. Transport the same type to every $v|p$, and write $I_v\subset P_v$ for its ideal and $I_A=\sum_{v|p}I_vA$.

    We first specify the normalization of the compact type. Recall that the block $\kB_v$ in Section \ref{sec lgc} corresponds to $\bar\rho|_{G_{F^1_v}}\otimes\omega$, and its central character is $\psi=\chi\varepsilon$. Thus Pašk\=unas' local deformation ring, whose determinant is $\psi\varepsilon$, is identified with $P_v$ by $r\mapsto r\otimes\varepsilon$. In his Hodge--Tate convention the corresponding weights are $(1-a,1-b)$. Let $K_v=\GL_2(\cO_{F^1_v})$ and $\mathcal K=\prod_{v|p}K_v$. Choose a $K_v$-invariant lattice $\Theta_v$ in
    $\sigma^\ast(\tau)\otimes_K(\Sym^{a-b-1}K^2\otimes\det^{1-a})$, where $\sigma^\ast(\tau)$ denotes $\sigma(\tau)$ or $\sigma^{\mathrm{cr}}(\tau)$, respectively, as in \cite[Section 2C]{Pa_k_nas_2016}. Its central character is $\psi|_{\cO_{F^1_v}^\times}$. Put $\Theta=\bigotimes_{v|p}\Theta_v$ and $\lambda=\Theta/\varpi\Theta$. As in Pašk\=unas' notation, set
    \[
       M_v(\Theta_v)=\Hom_{\cO[[K_v]]}^{\cont}(P_{\kB_v},\Theta_v^d)^d,
       \qquad M(\Theta)=\widehat\bigotimes_{v|p}M_v(\Theta_v),
    \]
    where $d$ denotes the Schikhof dual. For a finite-length smooth representation $\mu$ of $\mathcal K$, put $M(\mu)=\Hom_{\cO[[\mathcal K]]}^{\cont}(P_{\kB_\km},\mu^\vee)^\vee$, where $\vee$ denotes the Pontryagin dual. The $P_v$-actions are transported through the preceding cyclotomic twist.

    \begin{lem}\label{lem:nonordinary-type-modules}
       The following statements hold.
       \begin{enumerate}
          \item The module $M_v(\Theta_v)$ is a finite, $\cO$-torsion free $P_v$-module, Cohen--Macaulay of dimension $2$, with annihilator $I_v$.
          \item The module $M(\Theta)$ is finite over $A$, its annihilator is $I_A$, and $M(\Theta)/\varpi M(\Theta)\cong M(\lambda)$.
          \item We have $\dim P/I_p=2$ and $\dim P/(I_p,\varpi)=1$.
       \end{enumerate}
    \end{lem}

    \begin{proof}
       1) In the supersingular case, \cite[Theorem 2.22]{Pa_k_nas_2016} identifies $E_{\kB_v}$ with the local pseudo-deformation ring $P_v$. The assertions about $M_v(\Theta_v)$ follow from \cite[Propositions 2.28, 2.30 and 2.33]{Pa_k_nas_2016} (no restriction on $p$ there).

       2) The restriction of $P_{\kB_v}$ to $K_v$ is projective in the fixed-central-character compact category by \cite[Corollary 5.3]{Pa_k_nas_2015}. Hence \cite[Lemma 2.14]{Pa_k_nas_2015} gives the reduction identity for each local factor. Taking external tensor products gives $M(\Theta)/\varpi M(\Theta)\cong M(\lambda)$. Finiteness over $A$ follows by taking tensor products of finite sets of generators. Put $C_v=P_v/I_v$. Since $M_v(\Theta_v)$ is faithful over $C_v$, a finite set of generators gives an injection $C_v\hookrightarrow M_v(\Theta_v)^{\oplus s_v}$ for some $s_v$. Compact $\cO$-torsion free modules are topologically free over $\cO$, so their completed tensor products preserve these injections. Thus $A/I_A\cong\widehat\bigotimes_{v|p}C_v\hookrightarrow M(\Theta)^{\oplus\prod_{v|p}s_v}$, proving the assertion about the annihilator.

       3) The support of $M_v(\Theta_v)$ is $\Spec C_v$. Since $\varpi$ is regular on this Cohen--Macaulay module, its reduction has dimension $1$ and support $\Spec(C_v/\varpi)$, giving the two dimension assertions.
    \end{proof}

    \begin{prop}\label{prop:nonordinary-Hecke-finiteness}
       Put $T=\T_\psi(U_1^p)_\km$. Then $T/I_AT$ is finite over $\cO$.
    \end{prop}

    \begin{proof}
       Let $\mathcal M=M_\psi(U_1^p)_\km$ and let $\mm=\Hom_{\kC_{D_p^\times,\psi}(\cO)}(P_{\kB_\km},\mathcal M)$ as in Section \ref{sec lgc}. By Proposition \ref{finiteness from lgc}, $\mm$ is a faithful finite $T$-module and $T$ is finite over $A$. Since all local blocks are supersingular, the block equivalence gives $\mm\widehat\otimes_A P_{\kB_\km}\cong\mathcal M$. The two $A$-actions agree integrally by Theorem \ref{lgc}.

       Consider the finite $T$-module $N=\mm\otimes_A M(\Theta)$. Lemma \ref{lem:nonordinary-type-modules} (2) and the finite-length tensor--Hom comparison of \cite[Proposition 2.4]{Pa_k_nas_2015} give
       \[
          N/\varpi N\cong\mm\otimes_A M(\lambda)
          \cong\Hom_{\cO[[\mathcal K]]}^{\cont}(\mathcal M,\lambda^\vee)^\vee.
       \]
       The last space is finite-dimensional over $\F$: $\mathcal M$ is finite over $\cO[[\mathcal K]]$ by \cite[Remark 3.2.2]{pan2022fontaine}, and $\lambda$ is finite-dimensional and smooth. Topological Nakayama now shows that $N$ is finite over $\cO$. 

       We claim that $\Supp_TN=V(I_AT)$. For a prime $\kq$ of $T$, faithfulness of the finite $T$-module $\mm$ gives $\mm\otimes_Tk(\kq)\ne0$. The fibre of $N$ at $\kq$ is nonzero precisely when $M(\Theta)\otimes_Ak(\kq)\ne0$, which, by Lemma \ref{lem:nonordinary-type-modules} (2), is equivalent to $I_A\subset\kq\cap A$. This proves the claim. If $J=\operatorname{Ann}_TN$, then $\sqrt J=\sqrt{I_AT}$, while $T/J\hookrightarrow\End_\cO(N)$ is finite over $\cO$. It follows that $(T/I_AT)^{\operatorname{red}}$ is finite over $\cO$, and Proposition \ref{finite-after-reduction} gives the assertion for $T/I_AT$ itself.
    \end{proof}

    \begin{thm}\label{thm:nonordinary-type-finiteness}
       Let $R=R_\Q^{\operatorname{univ}}$. For every nonempty regular type as above, the ring $R/I_pR$ is finite over $\cO$.
    \end{thm}

    \begin{proof}
       The natural surjection $R_{F^1}^{\operatorname{univ}}\twoheadrightarrow T$ has nilpotent kernel by Theorem \ref{potential pro-modularity irr}. Proposition \ref{prop:nonordinary-Hecke-finiteness} and Proposition \ref{finite-after-reduction} therefore show that $R_{F^1}^{\operatorname{univ}}/I_AR_{F^1}^{\operatorname{univ}}$ is finite over $\cO$.

       The universal property gives an integral version of the diagram in Lemma \ref{com dia irr}:
       \[\begin{tikzcd}
          A \arrow[r] \arrow[d,two heads,"\Delta"'] & R_{F^1}^{\operatorname{univ}} \arrow[d,"\mathrm{res}"] \\
          P \arrow[r] & R_\Q^{\operatorname{univ}}.
       \end{tikzcd}\]
       Since $p$ splits completely in $F^1$, the map $\Delta$ identifies all local factors with $P$. We have used the same type at every $v|p$, so $\Delta(I_A)=I_p$. The right vertical map is finite by Proposition \ref{finiteness 3}. Hence the induced map
       $R_{F^1}^{\operatorname{univ}}/I_AR_{F^1}^{\operatorname{univ}}\to R/I_pR$ is finite, proving the result.
    \end{proof}

    \begin{cor}\label{cor:nonordinary-flat}
       If $\bar\rho|_{G_{\Q_p}}$ is absolutely irreducible, then both $R_\Q^{\operatorname{univ}}$ and $R_\Q$ are $\cO$-flat.
    \end{cor}

    \begin{proof}
       Put $R=R_\Q^{\operatorname{univ}}$. By Theorem \ref{dim in irr case}, $R$ is a local complete intersection and every irreducible component has dimension $3$. Proposition \ref{finiteness irr} gives a finite map $P\to R$. Moreover, $P/\varpi$ is an integral domain of dimension $3$. Indeed, the fixed-determinant framed local ring is $P[[x_1,x_2,x_3]]$, and this assertion follows from \cite[Corollary 5.4 and Theorem 5.6]{B_ckle_2023}.

       Suppose that a minimal prime $\kq$ of $R$ contains $\varpi$, and put $D=R/\kq$. Then $D$ is a domain of dimension $3$, so the finite map $P/\varpi\to D$ is injective. Choose a nonempty regular local type, for example the type of $\rho_0|_{G_{\Q_p}}$, and let $I_p$ be its ideal. Let $\kq_1$ be the inverse image in $P/\varpi$ of a one-dimensional minimal prime of $P/(I_p,\varpi)$, which exists by Lemma \ref{lem:nonordinary-type-modules} (3). By lying over, there is a one-dimensional prime $\kq_1'$ of $D$ lying above $\kq_1$. Then $D/\kq_1'$ is a quotient of $R/(I_p,\varpi)$, which is artinian by Theorem \ref{thm:nonordinary-type-finiteness}. We get a contradiction.
       
       
       Arguing as in the proof of Corollary \ref{cor:nonEis-flat}, we obtain our result.

    \end{proof}

    \begin{rem}\label{rem:nonordinary-type-nonemptiness}
       The method also suggests a fixed-type finiteness problem for residually reducible global pseudo-deformation rings. A potential pro-modularity input for the large-dimensional quotient is available in the setting of Section \ref{potential pro-modular}. However, the local argument requires a modification: the block endomorphism ring need not coincide with its centre, the local pseudo-deformation ring; see, for example, \cite[Theorem 2.23]{Pa_k_nas_2016}. One must therefore work over the full endomorphism ring and verify the corresponding integral type-support comparison before applying the finite descent argument above.

       Even if this yields a global type quotient finite over $\cO$, positive $\cO$-rank would require an additional argument. The dimension estimate of \cite[Proposition 2.1.4]{EGPZ} concerns deformation rings of representations and does not directly apply to these pseudo-deformation rings. It is therefore natural to ask for a corresponding lower bound in the pseudo-deformation setting, especially when the local base is not regular. When it is regular, the intersection argument of \cite[Lemmas 1.1.1 and 1.1.2]{Allen_2019} provides another route to lower bounds and branchwise nonemptiness, subject to the relevant dimension inequalities. We return to this argument in the next subsection.
    \end{rem}

    \subsection{Irregular Fontaine-Mazur conjecture in the residually irreducible case}\label{app:irregular-FM-nonEis}
    We prove weight-one Fontaine--Mazur results without a Taylor--Wiles hypothesis: an analogue of Theorem \ref{thm:irregular-FM-Eisenstein} for every odd $p$ in the locally residually reducible case, and an unconditional result for $p\ge5$ otherwise. Ordinary finiteness from Section \ref{app:ordinary-finiteness}, or fixed-type finiteness from Section \ref{app:nonordinary-finiteness} and Allen's intersection argument, supplies regular automorphic points. Adjoint Selmer vanishing and the infinite fern then give pro-modularity, as in Sections \ref{F=Q} and \ref{irr fm}, and Pan's classicality theorem gives the conclusion.

    Keep the odd, absolutely irreducible $\bar\rho:G_{\Q,\Sigma}\to\GL_2(\F)$. We choose a de Rham lift $\chi$ of $\det\bar\rho$ separately below, and write $R_\Q^{\operatorname{univ}}$ for its fixed-determinant deformation ring. Determinant twists preserve dimension, flatness and local-to-global finiteness, changing the ordinary weight map by an automorphism of $\Lambda_\Q$. For crystalline types we use only finite-order twists unramified at $p$, preserving the weights and inertial type.

    Recall $R_\Q\cong R_\Q^{\operatorname{univ}}\widehat\otimes_\cO R_{\mathbf{1}}$ for odd $p$ \cite[Lemma 6.1.2]{Allen_2019}. Corollaries \ref{cor:nonEis-flat} and \ref{cor:nonordinary-flat} show that all components are of characteristic $0$. As in Corollary \ref{at least 4} (2), after enlarging coefficients each component is given by a fixed-determinant component and a torsion character in $R_{\mathbf{1}}$, leaving the cyclotomic variable free. Thus twisting by a character lifting $\mathbf{1}$ and factoring through the cyclotomic $\Z_p$-extension preserves each component.

    Choose $N$, prime to $p$ and supported on $\Sigma\setminus\{p\}$, with $v_\ell(N)\ge\operatorname{Sw}_\ell(\bar\rho)+2$ for every $\ell\in\Sigma\setminus\{p\}$. As in Section \ref{F=Q}, this bounds the prime-to-$p$ conductor of every automorphic lift under consideration. Strong Serre's modularity theorem supplies the residual maximal ideal $\km$ of $\T(N)$, and we have a natural surjection $R_\Q\twoheadrightarrow\T(N)_\km$. We also choose a common tame level $K^p$ as in Section \ref{irr fm}, so that all these cuspidal eigensystems occur in $\mathbb T(K^p)$.

    \textbf{Case 1: Locally reducible residual representations.} Write $\bar\rho|_{G_{\Q_p}}^{\operatorname{ss}}=\bar{\chi}_1\oplus\bar{\chi}_2$, where $\bar{\chi}_1$ and $\bar{\chi}_2$ are characters of $G_{\Q_p}$. In this case we take the determinant lift $\chi$ to have finite order.

    \begin{lem}\label{lem:nonEis-ordinary-component}
       Let $E$ be a two-dimensional domain quotient of a character-labelled ordinary ring $R_\Q^{\ord,\bar\psi_1}$ with finite-order determinant $\chi$. If $C$ is an irreducible component of $\Spec R_\Q$ containing the image of $\Spec E$, then $C$ is a component of $\Spec\T(N)_\km$. Moreover, modular points are Zariski dense in $C$, and $C$ is pro-modular in the sense of Section \ref{irr fm}.
    \end{lem}

    \begin{proof}
       By Theorem \ref{thm:ordinary-finiteness-nonEis}, $E$ is finite over $\Lambda_\Q$. Since both rings have dimension $2$ and $\Lambda_\Q$ is a domain, this map is injective, hence the weight map is surjective on spectra. We identify local characters with characters of $\Q_p^\times$ by local class field theory. After a finite-order cyclotomic twist lifting $\mathbf{1}$, we may assume that $\chi$ is trivial on principal units. By our previous discussions, this does not change $C$. Write $\chi|_{I_p}=\widetilde\omega^t$ and $\bar\psi_1|_{I_p}=\omega^i$, where $\widetilde\omega$ is the Teichm\"uller lift of $\omega$.

       Choose a positive integer $a$ divisible by $p-1$, and a one-dimensional characteristic-zero point $\kp_f'$ of $\Spec E$ with $\psi_1|_{I_p}=\widetilde\omega^i\varepsilon^a$. Its quotient character has restriction $\widetilde\omega^{t-i}\varepsilon^{-a}$ to inertia. Put $\rho_f=\rho(\kp_f')\otimes\varepsilon^a$. Its stable-line and quotient characters have Hodge--Tate weights $-2a$ and $0$, respectively. Since $2a>1$, these ordered weights make the local extension de Rham, as in \textbf{Case 2.1} of the proof of Theorem \ref{big R=T Q}. By \cite[Theorem 1.0.2]{Zhang:2024aa}, $\rho_f$ is automorphic. Its quotient character at $p$ is unramified up to a finite-order character of $p$-power conductor, so the same argument as in \textbf{Case 2.1} gives a modular point $\kp_f$ of $\Spec\T(N)_\km$. Both twists used above lift $\mathbf{1}$ and vary only along the cyclotomic line, so $\kp_f$ still belongs to $C$.

       Corollary \ref{cor:ordinary-adjoint-vanishing} applies to $\rho_f$, including when it has CM. The proof of Lemma \ref{iso} only uses the characteristic-zero representation and therefore applies here as well, giving $\widehat{(R_\Q)_{\kp_f}}\cong K_f[[x_1,x_2,x_3]]$, where $K_f$ is the residue field. Thus $C$ is the unique component of $\Spec R_\Q$ through $\kp_f$. Write $\kq_C$ for its minimal prime, and choose a minimal prime of $\T(N)_\km$ below $\kp_f$, with contraction $\kq$ to $R_\Q$. Then $\kq_C\subset\kq$, while Theorems \ref{lower bound} and \ref{mazur non-eisenstein} give $4\le\dim R_\Q/\kq\le\dim R_\Q/\kq_C=4$. Hence $\kq=\kq_C$. This proves the component assertion, and Theorem \ref{lower bound} gives density of modular points. All these points are cuspidal, since $\bar\rho$ is absolutely irreducible. The choice of $K^p$ and the same density argument as in Section \ref{irr fm} now place $C$ in $\Spec\mathbb T(K^p)_\km$, proving pro-modularity.
    \end{proof}

    \begin{prop}\label{prop:nonEis-generic-big-RT}
       Suppose that $\bar{\chi}_1/\bar{\chi}_2\ne\omega^{\pm1}$. Then the natural surjection $R_\Q\twoheadrightarrow\T(N)_\km$ has nilpotent kernel. In particular, modular points are Zariski dense in $\Spec R_\Q$, and every characteristic-zero deformation of $\bar\rho$ is pro-modular.
    \end{prop}

    \begin{proof}
       Let $C$ be any component of $\Spec R_\Q$. In its determinant-character decomposition, specialize the character factor at a finite-order character and absorb its square into $\chi$. This gives a component $C_\chi$ of $\Spec R_\Q^{\operatorname{univ}}$ whose image is contained in $C$, with $\chi$ of finite order. Put $D=R_\Q^{\operatorname{univ}}/\kq_\chi$, where $\kq_\chi$ defines $C_\chi$, so $\dim D=3$ by Theorem \ref{dim in irr case}. Let $P$ be the local pseudo-deformation ring with determinant $\chi|_{G_{\Q_p}}$.

       First suppose that $\bar{\chi}_1=\bar{\chi}_2=\bar\mu$. By \cite[Corollary 9.13]{Pa_k_nas_2013}, $P\cong\cO[[x_1,x_2,x_3]]$. Let $S_p\cong\cO[[u,v]]$ be the character deformation ring lifting $\bar\mu$. The character family $\psi_1+\chi\psi_1^{-1}$ defines a finite map $P\to S_p$, as in \cite[the proof of Lemma 7.3.1]{pan2022fontaine}. Thus $B=D\widehat\otimes_P S_p$ is a quotient of $D[[u,v]]$ by three graph relations, giving $\dim B\ge3+2-3=2$. Choose a domain quotient $E$ of $B$ of dimension at least $2$. At its generic point the local representation has characters $\psi_1$ and $\chi\psi_1^{-1}$. Exchange them if necessary so that the specified character acts on the stable line. The ordinary matrix identities hold over $\operatorname{Frac}(E)$, hence over $E$. Exactly as in Lemma \ref{lem:ordinary-character-quotients}, this makes $E$ an $R_\Q^{\operatorname{univ}}$-algebra quotient of $R_\Q^{\ord,\bar\mu}$, without asserting factorization through a fixed choice of orientation. Theorem \ref{thm:ordinary-finiteness-nonEis} gives $\dim E\le2$, so equality holds. Its image is contained in $C_\chi$, and Lemma \ref{lem:nonEis-ordinary-component} applies.

       Now suppose that $\bar{\chi}_1\ne\bar{\chi}_2$. Our hypothesis excludes the cyclotomic ratios, so the local reducibility ideal $J\subset P$ is principal by \cite[Corollary B.20]{Pa_k_nas_2013}. Hence $\dim D/JD\ge2$. Choose a domain quotient $E$ of dimension at least $2$. Local residual multiplicity-freeness defines the two characters over $E$. At the generic point choose the character acting on the actual stable line; the corresponding ordinary matrix identities descend to $E$. Their reductions show that the chosen residual character is a subcharacter of $\bar\rho|_{G_{\Q_p}}$. Thus $E$ is a quotient of the appropriate ring $R_\Q^{\ord,\bar\psi_1}$, and ordinary finiteness again gives $\dim E=2$. Lemma \ref{lem:nonEis-ordinary-component} applies as before.

       We have proved that every component of $\Spec R_\Q$ is a component of $\Spec\T(N)_\km$. Therefore the kernel of the surjection is contained in every minimal prime of $R_\Q$, and is nilpotent. The density and pro-modularity assertions follow from Lemma \ref{lem:nonEis-ordinary-component}.
    \end{proof}

    \begin{thm}\label{thm:irregular-FM-nonEis-locred}
       Let $p>2$ be a prime, and let $\rho:G_\Q\to\GL_2(\overline{\Q}_p)$ be a continuous irreducible representation. Assume that:
       \begin{enumerate}
          \item $\rho$ is unramified outside finitely many primes;
          \item $\rho$ is odd;
          \item $\rho|_{G_{\Q_p}}$ is Hodge--Tate of weights $\{0,0\}$;
          \item $\bar\rho$ is absolutely irreducible, and $\bar\rho|_{G_{\Q_p}}^{\operatorname{ss}}=\bar\chi_1\oplus\bar\chi_2$ for characters $\bar\chi_1,\bar\chi_2$ of $G_{\Q_p}$;
          \item either $\bar\chi_1/\bar\chi_2\ne\omega^{\pm1}$ or $\rho|_{G_{\Q_p}}$ is reducible.
       \end{enumerate}
       Then $\rho$ arises from a classical cuspidal eigenform of weight $1$. In particular, $\rho$ has finite image. If $p\ge5$, hypothesis \textup{(5)} is automatic.
    \end{thm}

    \begin{proof}
       Choose $\Sigma$ containing $p$ and all ramified primes of $\rho$, and enlarge coefficients so that $\rho$ is defined over $K$. Write $\bar\rho|_{G_{\Q_p}}^{\operatorname{ss}}=\bar{\chi}_1\oplus\bar{\chi}_2$. If $\bar{\chi}_1/\bar{\chi}_2\ne\omega^{\pm1}$, Proposition \ref{prop:nonEis-generic-big-RT} makes $\rho$ pro-modular, and Theorem \ref{thm:Pan-classicality} gives the conclusion.

       It remains to consider $\bar{\chi}_1/\bar{\chi}_2=\omega^{\pm1}$. By \textup{(5)}, $\rho|_{G_{\Q_p}}$ is reducible; by Lemma \ref{lem:weight-zero-local-induction}, this is automatic for $p\ge5$. After enlarging $K$, choose a stable line $L_p$. Its intersection with a global stable lattice is saturated, so its reduction gives a residual subcharacter $\bar\psi_1$. The local residual characters are distinct, so the character-labelled ordinary ring of Section \ref{app:ordinary-finiteness} also represents the ordinary deformation problem with this prescribed stable line. Put $\chi=\det\rho$, which has finite order by global class field theory, and view $\rho$ as a point $\kp_\rho$ of $\Spec R_\Q^{\ord,\bar\psi_1}$.

       We use the characteristic-zero argument of Section \ref{irr fm}. We have $H^0(G_{\Q,\Sigma},\ad^0\rho(1))=0$: a non-zero element would give $\rho\cong\rho(1)$, contradicting $\varepsilon^2\ne\mathbf{1}$. The completed local ring at $\kp_\rho$ represents ordinary deformations of $(\rho,L_p)$ with fixed determinant. The calculation in \cite[Section 2.5]{Zhang:2024aa} applies unchanged: the framed local dimensions are at least $5$ at $p$ and $3$ elsewhere, and the relative generator-minus-relation bound is $|\Sigma|-2$. Removing $4|\Sigma|-1$ framing variables gives $\dim\widehat{(R_\Q^{\ord,\bar\psi_1})_{\kp_\rho}}\ge1$. Thus some component through $\kp_\rho$ has dimension at least $2$. For its domain quotient $E$, Theorem \ref{thm:ordinary-finiteness-nonEis} gives $\dim E\le2$, hence equality. No residual adjoint $H^0$-vanishing is needed, even at $p=3$.

       Let $C$ be a component of $\Spec R_\Q$ containing the image of $\Spec E$. It contains the original point defined by $\rho$, and Lemma \ref{lem:nonEis-ordinary-component} makes $C$ pro-modular. Thus $\rho$ is pro-modular, and Theorem \ref{thm:Pan-classicality} again applies.
    \end{proof}

    \textbf{Case 2: Locally absolutely irreducible residual representations.} In this case we assume $p\ge5$. We replace ordinary finiteness by the crystalline case of Theorem \ref{thm:nonordinary-type-finiteness}. The regularity of the local deformation ring allows us to find points of one fixed crystalline type on every chosen fixed-determinant component.

    \begin{prop}\label{prop:nonEis-crystalline-big-RT}
       Suppose that $p\ge5$ and $\bar\rho|_{G_{\Q_p}}$ is absolutely irreducible. Then the natural surjection $R_\Q\twoheadrightarrow\T(N)_\km$ has nilpotent kernel. In particular, modular points are Zariski dense in $\Spec R_\Q$, and every characteristic-zero deformation of $\bar\rho$ is pro-modular.
    \end{prop}

    \begin{proof}
       By Serre's modularity theorem \cite{Khare_2009a}, choose a normalized cuspidal newform $f_0$ lifting $\bar\rho$, of weight $k_0\ge2$ and prime-to-$p$ level dividing $N$. This is a permitted choice of the initial modular lift in Section \ref{app:nonEis}. Put $\rho_0=\rho_{f_0}$ and $\chi_0=\det\rho_0$. Then $\rho_0|_{G_{\Q_p}}$ is crystalline of Hodge--Tate weights $\{0,1-k_0\}$, so $\chi_0$ is crystalline at $p$.

       Let $C$ be an irreducible component of $\Spec R_\Q$. In the determinant-character decomposition with respect to $\chi_0$, specialize the free character factor, given by the image of $1+p\Z_p$ under local reciprocity, to the trivial character. The resulting character $\nu$ is of finite order, lifts $\mathbf{1}$ and is unramified at $p$. Absorbing $\nu^2$ into the determinant gives a component $C_\chi$ of the fixed-determinant ring $R=R_\Q^{\operatorname{univ}}$ with $\chi=\chi_0\nu^2$, whose image is contained in $C$. In particular, $\chi$ is still crystalline at $p$, and $\dim C_\chi=3$ by Theorem \ref{dim in irr case}.

       Let $P$ be the local deformation ring with determinant $\chi|_{G_{\Q_p}}$, identified with the local pseudo-deformation ring. We have $H^0(G_{\Q_p},\ad^0\bar\rho(1))=0$ as $p\ge5$. Local Tate duality and the local Euler characteristic formula therefore give $P\cong\cO[[x_1,x_2,x_3]]$.

       Take $w=(0,1-k_0)$ and the trivial two-dimensional inertial type, and let $I_p=I_p^{\mathrm{cr}}(w,\mathbf{1}^{\oplus2})$ be the ideal of Section \ref{app:nonordinary-finiteness}. This type is nonempty: $\rho_0|_{G_{\Q_p}}\otimes\nu|_{G_{\Q_p}}$ gives a point of it. Here we take the crystalline quotient, including monodromy operator $N=0$. Lemma \ref{lem:nonordinary-type-modules} (3) gives $\dim P/I_p=2$, and Theorem \ref{thm:nonordinary-type-finiteness} gives finiteness of $R/I_pR$ over $\cO$. The latter statement is preserved by the finite-order twist $\nu$, since it is unramified at $p$. Thus \cite[Lemma 1.1.2]{Allen_2019} applies with $R^{\loc}=P$ and $X^{\loc}=\Spec P/I_p$: the required inequality is $\dim C_\chi+\dim P/I_p-\dim P=3+2-4=1$. We obtain a characteristic-zero point $\kp_f$ on $C_\chi$ whose associated representation is crystalline of weights $\{0,1-k_0\}$. The integral finiteness is essential here, as it excludes an intersection confined to the special fibre.

        Put $\rho_f=\rho(\kp_f)$. By \cite[Theorem 1.0.2]{Zhang:2024aa}, $\rho_f$ is automorphic. In particular, it comes from a cuspidal eigenform of weight $k_0$, and the classical local-global compatibility shows that the corresponding newform has level prime to $p$ and dividing $N$. Hence $\kp_f$ is already a modular point of $\Spec\T(N)_\km$ lying on $C$.

       We check the adjoint Selmer vanishing needed at this point. If the associated automorphic representation had CM by a quadratic field $L\subset\Q(\mu_{p^\infty})$, the quadratic character $\delta_L$ would be ramified at $p$ and satisfy $\rho_f\otimes\delta_L\cong\rho_f$. This is impossible: $\operatorname{WD}(\rho_f|_{G_{\Q_p}})$ has trivial inertia, whereas its twist by $\delta_L|_{G_{\Q_p}}$ does not. Thus the exceptional CM case in Theorem \ref{thm:Newton-Thorne-5.4} is excluded, and $H^1_f(\Q,\ad\rho_f)=0$. As in the first paragraph of the proof of Corollary \ref{cor:ordinary-adjoint-vanishing}, genericity of the cuspidal local components and local-global compatibility give $H^1_{g,\Sigma}(\Q,\ad\rho_f)=H^1_f(\Q,\ad\rho_f)=0$.

       The argument of Lemma \ref{iso} gives $\widehat{(R_\Q)_{\kp_f}}\cong k(\kp_f)[[x_1,x_2,x_3]]$. In particular, $C$ is the unique unrestricted component through $\kp_f$. The component comparison in the last paragraph of the proof of Lemma \ref{lem:nonEis-ordinary-component} applies unchanged: Theorem \ref{lower bound} makes $C$ a component of $\Spec\T(N)_\km$, with Zariski dense modular points, and places $C$ in $\Spec\mathbb T(K^p)_\km$. Since $C$ was arbitrary, the kernel of $R_\Q\twoheadrightarrow\T(N)_\km$ is nilpotent, and the assertions follow.
    \end{proof}

    \begin{cor}\label{cor:irregular-FM-nonEis}
       Let $p\ge5$, and let $\rho:G_\Q\to\GL_2(\overline{\Q}_p)$ be a continuous irreducible representation, unramified outside finitely many primes. Suppose that $\rho$ is odd, that $\rho|_{G_{\Q_p}}$ is Hodge--Tate of weights $\{0,0\}$, and that $\bar\rho$ is absolutely irreducible. Then $\rho$ arises from a classical cuspidal eigenform of weight $1$. In particular, $\rho$ has finite image.
    \end{cor}

    \begin{proof}
       After enlarging the coefficient field, the local residual representation is either reducible or absolutely irreducible. The first case is Theorem \ref{thm:irregular-FM-nonEis-locred}. In the second case, Proposition \ref{prop:nonEis-crystalline-big-RT} makes $\rho$ pro-modular, and Theorem \ref{thm:Pan-classicality} gives the result.
    \end{proof}

    \begin{rem}
       Previously, this result was proved in \cite{PS16} under the Taylor-Wiles hypothesis, including the case $p=3$. Our method here is different, based on Pan's classicality \cite{Pan_2022} and several pro-modularity results developed in this paper.
    \end{rem}

    \begin{rem}\label{rem:irregular-FM-nonEis-p3}
       At $p=3$, Case 1 now has the same remaining exceptional configuration as Section \ref{irr fm}: $\bar{\chi}_1/\bar{\chi}_2=\omega$ and $\rho|_{G_{\Q_3}}$ is absolutely irreducible, hence induced from a character of $G_{\Q_3(\mu_3)}$ by Lemma \ref{lem:weight-zero-local-induction}. Such a point is not ordinary. When $\rho|_{G_{\Q_3}}$ is reducible, Theorem \ref{thm:irregular-FM-nonEis-locred} applies even if $\bar\rho$ has a cyclotomic self-twist, since its dimension argument is at a characteristic-zero point.

       In Case 2, a local cyclotomic self-twist $\bar\rho|_{G_{\Q_3}}\cong\bar\rho|_{G_{\Q_3}}(1)$ can likewise occur. The vanishing of the local obstruction group, and hence the regularity needed to apply Allen's intersection lemma as above, is then no longer available. Fixed-type finiteness alone does not show that the chosen crystalline locus meets every global component. Also see \cite[Section 6.4]{Allen_2019}.
    \end{rem}

\renewcommand{\bibliofont}{\footnotesize\linespread{0.93}\selectfont}
\bibliography{non-ordinary}
\bibliographystyle{alpha}
\end{document}